%% file: KKLRU.tex
\documentclass[a4paper,leqno,12pt]{article}

\usepackage[utf8]{inputenc}
\usepackage[T1]{fontenc}
\usepackage[english]{babel}

\usepackage{crimson}

\usepackage{geometry}

\usepackage{fancyhdr}

\usepackage[inline]{enumitem}

\usepackage{graphicx}

\usepackage{color}

\usepackage[intlimits]{amsmath}
\usepackage{amssymb, amsfonts}

\usepackage[thmmarks, amsmath, amsthm]{ntheorem}

\usepackage{mathtools}
\usepackage{tensor}

\usepackage{MnSymbol}
\usepackage{mathbbol}

\usepackage{bbm}

\usepackage{mathrsfs}

\usepackage[all]{xy}

\usepackage{hyperref}
\hypersetup{
  colorlinks   = true, 
  urlcolor     = blue, 
  linkcolor    = blue, 
  citecolor   = black 
}

\usepackage[nameinlink]{cleveref}

\numberwithin{equation}{section}
\newtheorem{theoremcounter}{theoremcounter}[section]

\theoremstyle{plain}

\newtheorem{corollary}[theoremcounter]{Corollary}
\newtheorem{lemma}[theoremcounter]{Lemma}
\newtheorem{proposition}[theoremcounter]{Proposition}

\newtheorem{theorem}[theoremcounter]{Theorem}

\theoremnumbering{Alph}
\newtheorem{introtheorem}{Theorem}
\newtheorem{introcorollary}[introtheorem]{Corollary}
\newtheorem{introexample}[introtheorem]{Example}

\theoremstyle{definition}

\newtheorem{definition}[theoremcounter]{Definition}

\theoremstyle{remark}

\newtheorem{example}[theoremcounter]{Example}

\newtheorem{notation}[theoremcounter]{Notation}

\newtheorem{remark}[theoremcounter]{Remark}
\newtheorem*{remarkNN}{Remark}

\usepackage{xargs}                      
\usepackage[pdftex,dvipsnames]{xcolor}  
\usepackage[colorinlistoftodos,prependcaption,textsize=tiny]{todonotes}
\newcommandx{\unsure}[2][1=]{\todo[linecolor=red,backgroundcolor=red!25,bordercolor=red,#1]{#2}}
\newcommandx{\change}[2][1=]{\todo[linecolor=blue,backgroundcolor=blue!25,bordercolor=blue,#1]{#2}}
\newcommandx{\info}[2][1=]{\todo[linecolor=OliveGreen,backgroundcolor=OliveGreen!25,bordercolor=OliveGreen,#1]{#2}}
\newcommandx{\improvement}[2][1=]{\todo[linecolor=Plum,backgroundcolor=Plum!25,bordercolor=Plum,#1]{#2}}

\usepackage{tikz-cd}
\usetikzlibrary{matrix,arrows,decorations.pathmorphing}

\usepackage[style=alphabetic]{biblatex}
\DeclareFieldFormat{eprint}{\href{https://arxiv.org/abs/#1}{\texttt{#1}}}

\DeclareFieldFormat{eprint:urn}{%
  URN\addcolon\space
  \ifhyperref
  {\href{http://www.nbn-resolving.org/#1}{\nolinkurl{#1}}}
  {\nolinkurl{#1}}}

\renewbibmacro*{in:}{}

\DeclareFieldFormat
[article,inbook,incollection,inproceedings,patent,thesis,unpublished,misc]
{title}{#1}

\input{shortcuts.tex}

\renewcommand{\Cstar}{\ensuremath{\displaystyle \mathrm{C}^*}}

\newcommand{\Gnaught}{{\cG^{(0)}}}
\newcommand{\Hnaught}{\cH^{(0)}}
\newcommand{\Iso}{\mathrm{Iso}}
\newcommand{\fb}{\partial_{\mathrm F}}
\newcommand{\hb}{\partial_{\mathrm H}}
\newcommand{\bb}{\partial_{\mathrm B}}

\newcommand{\Dix}{\ensuremath{\mathrm{Dix}}}
\newcommand{\Ered}{\mathrm{E}_{\mathrm{red}}}
\newcommand{\Eess}{\mathrm{E}_{\mathrm{ess}}}
\newcommand{\Cstaress}{\Cstar_{\mathrm{ess}}}
\newcommand{\Haus}{\mathrm{Haus}}

\newcommand{\Mloc}{\rM_{\mathrm{loc}}}
\newcommand{\cconv}{\mathop{\ol{\conv}}}

\newcommand{\Sub}{\mathrm{Sub}}
\newcommand{\redtimes}{\rtimes_{\mathrm{red}}}
\newcommand{\Binfty}{\mathcal{B}^\infty}

\DeclareMathOperator{\weakstar}{w*}

\newcommand{\CP}{\mathcal{CP}}

\renewcommand{\Cstar}{\ensuremath{\text{\rm C}^*}}

\newcommand{\authors}{Matthew Kennedy $\bullet$ Se-Jin Kim $\bullet$ Xin Li $\bullet$ Sven Raum $\bullet$ Dan Ursu}
\renewcommand{\title}{The ideal intersection property for \\ essential groupoid C$^*$-algebras}

\begin{document}


\thispagestyle{empty}

\noindent
\begin{minipage}{\linewidth}
  \begin{center}
    \textbf{\Large \title} \\[0.3em]
    \authors    
  \end{center}
\end{minipage}

\renewcommand{\thefootnote}{}
\footnotetext{
  \textit{MSC classification:
    46L55, 
    37A55, 
    46L05, 
    20M18, 
    22A22, 
  }
}
\footnotetext{
  \textit{Keywords: essential groupoid \Cstar-algebra, simplicity, \Cstar-irreducible inclusion, ideal intersection property, Furstenberg boundary, confined subgroup, Powers averaging, Thompson's group $\rT$}
}

\vspace{2em}
\noindent
\begin{minipage}{\linewidth}
  \textbf{Abstract}.
  We characterise, in several complementary ways, {\'e}tale groupoids with locally compact Hausdorff space of units whose essential groupoid \Cstar-algebra has the ideal intersection property, assuming that the groupoid is topologically transitive and either Hausdorff or $\sigma$-compact.  This leads directly to a characterisation of the simplicity of this \Cstar-algebra which, for Hausdorff groupoids, agrees with the reduced groupoid \Cstar-algebra.  Specifically, we prove for topologically transitive groupoids that the ideal intersection property is equivalent to the absence of essentially confined amenable sections of isotropy groups.  For topologically transitive groupoids with compact space of units we moreover show that this is equivalent to the uniqueness of equivariant pseudo-expectations.  A key technical idea underlying our results is a new notion of groupoid action on \Cstar-algebras including the essential groupoid \Cstar-algebra itself.  For minimal groupoids, we further obtain a relative version of Powers averaging property.  Examples arise from suitable group representations into simple groupoid \Cstar-algebras.  This is illustrated by the example of the quasi-regular representation of Thompson's group $\rT$ with respect to Thompson's group $\rF$, which satisfies the relative Powers averaging property in the Cuntz algebra $\cO_2$.
\end{minipage}

\tableofcontents


\section{Introduction}
\label{sec:introduction}

Groupoids provide a framework encompassing both groups and topological spaces, and provide an abstraction of the notion of a quotient space in these settings, much like stacks in algebraic geometry. They naturally arise as transformation groupoids encoding the topological dynamical structure of discrete groups, and as transversal groupoids of foliations in differential geometry.  Groupoids arising from these and many other examples are {\'e}tale, which is a notion suitably abstracting the topological properties of groupoids arising from actions of discrete groups.  Renault \cite{renault80} and Connes \cite{connes1982-foliations}, by introducing various \Cstar-algebraic completions of an appropriate convolution algebra, discovered that {\'e}tale groupoids give rise to an extraordinarily rich class of operator algebras. On the other hand, since the structure of these algebras encodes much of the structure of the underlying groupoid, they are capable of serving as a proxy for the study of structures relevant to other areas of mathematics, such as semigroup \Cstar-algebras \cite{cuntzechterhoffliyou2017} and Bost-Connes systems for arbitrary number fields \cite{lacalarsenneshveyev2009}. Furthermore, {\'e}tale groupoids also naturally arise from within the theory of operator algebras, where they encode important structural features through the notion of Cartan subalgebras \cite{renault2008-cartan,l2020-cartan-classifiable}.

Groupoids appearing in applications are frequently {\'e}tale or, after choosing a transversal, Morita equivalent to an {\'e}tale groupoid.  However, they may be non-Hausdorff, as the example of groupoids of germs associated with group actions shows, and there is no way to remedy this fact. Nevertheless, even in the non-Hausdorff case, the unit space of an {\'e}tale groupoid is typically locally compact and Hausdorff. Therefore, understanding the structural properties of the \Cstar-algebras associated to {\'e}tale groupoids with locally compact Hausdorff unit spaces is an important challenge at the intersection of operator algebras and many other areas of mathematics.  Before stating our main results and putting them into the context of previous work, let us summarise how we meet this challenge in four ways. First, we provide a complete characterisation of {\'e}tale Hausdorff groupoids whose reduced groupoid \Cstar-algebra is simple.  Second, for non-Hausdorff groupoids satisfying the mild countability assumption of $\sigma$-compactness, we characterise the simplicity of their essential groupoid \Cstar-algebra.  It has recently become clear that this \Cstar-algebra, which agrees with the reduced groupoid \Cstar-algebra in the Hausdorff setting, is the correct replacement for the reduced groupoid \Cstar-algebra in the non-Hausdorff setting, when searching for an algebraic-dynamical description of the ideal space of a groupoid \Cstar-algebra.  It comes as a surprise that a careful development of the necessary techniques and notions leads to a clean characterisation of \Cstar-simplicity also in the non-Hausdorff case.  Third, we obtain appropriate analogues in the setting of groupoids of the averaging results considered by Powers and many others for reduced group \Cstar-algebras.  This allows us to strengthen previously known results on \Cstar-irreducibility of inclusions arising from groupoids of germs.  Fourth, behind these results lies the development of a novel category of groupoid \Cstar-dynamical systems, including a theory of boundaries for groupoids that culminates in a dynamical construction of the Furstenberg boundary of a groupoid.  Our novel notion of groupoid action on \Cstar-algebras has to the best of our knowledge no predecessor in the literature.

The first main result of this article completely solves the problem of characterising the simplicity of reduced groupoid \Cstar-algebras of {\'e}tale Hausdorff groupoids, and the simplicity of essential groupoid \Cstar-algebras of $\sigma$-compact {\'e}tale groupoids.
\begin{introtheorem}[See Theorem~\ref{thm:characterisation-confined-subgroups-minimal-case}]
  \label{thmintro:characterisation-confined-subgroups-minimal-case}
  Let $\cG$ be an {\'e}tale groupoid with locally compact Hausdorff space of units.  Assume that $\cG$ is Hausdorff, $\cG$ is $\sigma$-compact or $\cG$ has a compact space of units.  Then the essential groupoid \Cstar-algebra $\Cstaress(\cG)$ is simple if and only if $\cG$ is minimal and has no essentially confined amenable sections of isotropy groups.
\end{introtheorem}
Let us explain the notation we use and put our results into the context of previous work.  Theorem~\ref{thmintro:characterisation-confined-subgroups-minimal-case} generalises the breakthrough results characterising \Cstar-simple discrete groups \cite{kalantarkennedy2017,breuillardkalantarkennedyozawa14,kennedy15-cstarsimplicity,haagerup15}, and completes a sequence of partial results obtained for dynamical systems \cite{kawabe17, kalantarscarparo2021} and special classes of groupoids \cite{fackskandalis1982,renault1991,khoshkamskandalis2002,brownclarkfarthingsims2014,borys2020-thesis,borys2019-boundary,kwasniewskimeyer2019-essential}.  In particular, Borys had shown that for {\'e}tale Hausdorff groupoids with a compact space of units, the absence of confined amenable sections of isotropy groups is a sufficient condition for simplicity of the reduced groupoid \Cstar-algebra (there is however a gap in Borys' argument, as we explain in Remark~\ref{rem:NotAllStabAmen}).  Further, Kwa{\'s}niewski-Meyer had shown that for {\'e}tale groupoids with locally compact Hausdorff space of units, which are topologically free -- and thus have no essentially confined amenable sections of isotropy groups -- their essential groupoid \Cstar-algebra is simple.  The notion of confined sections of isotropy groups emerged from work on group rings \cite{hartleyyalesskij1997} and \Cstar-simple groups \cite{kennedy15-cstarsimplicity} and was subsequently generalised to dynamical systems \cite{kawabe17} and Hausdorff groupoids \cite{borys2019-boundary}.  In this setting, a section of isotropy groups is confined if, roughly speaking, it cannot be approximately conjugated into the groupoid's unit space.  The terminology stems from the special case of groups, where a confined subgroup is separated in a strong sense from the trivial subgroup.  For non-Hausdorff groupoids, it is necessary to consider a more general notion of essential confinedness, taking into account the closure of the unit space. Definition~\ref{def:section-isotropy} captures this notion rigorously.

While the essential groupoid \Cstar-algebras of a Hausdorff groupoid is identified with its reduced groupoid \Cstar-algebra, for non-Hausdorff groupoids, it is necessary to adapt the construction of the reduced groupoid \mbox{\Cstar-algebra}, if one aims to relate the ideal structure of a groupoid \Cstar-algebra to the algebraic and dynamical structure of the groupoid.  The need to consider a modification of the reduced groupoid \Cstar-algebra became apparent beginning with work of Khoshkam-Skandalis \cite{khoshkamskandalis2002}, followed by instructive examples of Exel \cite{exel2011}, and work of Exel-Pitts \cite{exelpitts2019-characterizing} and Clark-Exel-Pardo-Sims-Starling \cite{clarkexelpardosimsstarling2019} and Kwa{\'s}niewski-Meyer \cite{kwasniewskimeyer2019-essential}.  The essential groupoid \mbox{\Cstar-algebra} is the quotient of the reduced groupoid \mbox{\Cstar-algebra} by an ideal of singular elements \cite{clarkexelpardosimsstarling2019,kwasniewskimeyer2019-essential}.  This explains the need for some kind of countability assumption, such as $\sigma$-compactness, in order to control the singular elements in a suitable way.  By adopting this perspective, \cite{clarkexelpardosimsstarling2019,kwasniewskimeyer2019-essential} obtained a characterisation of certain amenable groupoids whose essential groupoid \Cstar-algebra is simple.  In particular \cite{kwasniewskimeyer2019-essential} showed that topologically free and minimal groupoids have simple essential groupoid \Cstar-algebras.

Let us now explain the technical core of the present work.  In the context of groupoid \mbox{\Cstar-algebras}, the problem of proving simplicity is subdivided by considering the ideal intersection property \cite{KT90, tomiyama1992-lecture-notes} and the study of orbits.  The inclusion $\conto(\Gnaught) \subseteq \Cstaress(\cG)$ has the ideal intersection property if every nonzero ideal of $\Cstaress(\cG)$ has nonzero intersection with $\conto(\Gnaught)$.  We abuse terminology slightly and  refer to $\Cstaress(\cG)$ having the ideal intersection property.  This \Cstar-algebra is simple if and only if it has the ideal intersection property and $\cG$ is minimal.  While it is straightforward to determine the minimality of a groupoid, it has been a major open problem to characterise the ideal intersection property for essential groupoid \Cstar-algebras.

Existing work on simplicity and the ideal intersection property for groupoid \Cstar-algebras has generally proceeded in two different directions. First, for amenable groupoids, the reduced and the full groupoid \Cstar-algebras agree, making characterisations of simplicity more accessible, as is also visible from our main results, where the amenability condition on sections of isotropy groups is automatically satisfied.  From an operator algebraic perspective this is reflected in the abundance of \mbox{*-homomorphisms} defined on the full groupoid \Cstar-algebra.  Previous work on amenable groupoids includes the work of Kawamura-Tomiyama and Archbold-Spielberg on transformation groups \cite{KT90, archboldspielberg94},  the work of Kumjian-Pask \cite{kumjianpask2000} and Robertson-Sims \cite{robertsonsims2007} characterising simplicity of groupoid \Cstar-algebras arising from higher-rank graphs, and finally the work of Brown-Clark-Farthing-Sims \cite{brownclarkfarthingsims2014} resulting in a complete characterisation of {\'e}tale second countable Hausdorff groupoids with simple full groupoid \Cstar-algebra.  As explained above, characterisations of the simplicity of essential groupoid \Cstar-algebras of amenable non-Hausdorff groupoids can be deduced from recent work by Clark-Exel-Pardo-Sims-Starling \cite{clarkexelpardosimsstarling2019} and by Kwa{\'s}niewski-Meyer \cite{kwasniewskimeyer2019-essential}.

Second, a completely different approach to simplicity and the ideal intersection property for groupoid \Cstar-algebras emerged from the breakthrough results about \Cstar-simple discrete groups obtained by Kalantar-Kennedy \cite{kalantarkennedy2017} and Breuillard-Kalantar-Kennedy-Ozawa \cite{breuillardkalantarkennedyozawa14}.  The crucial insight of \cite{kalantarkennedy2017} was that the simplicity of reduced group \Cstar-algebras has to be coupled to and understood through the group's action on its Furstenberg boundary.  Combined with Le Boudec's work \cite{leboudec15-cstarsimplicity}, along with further characterisations obtained by Kennedy \cite{kennedy15-cstarsimplicity} and Haagerup \cite{haagerup15}, this resolved the long-standing problem of characterising discrete groups whose reduced group \mbox{\Cstar-algebra} is simple, as reported in the S{\'e}minaire Bourbaki \cite{raum19-bourbaki}.

The success of these methods in understanding the \Cstar-simplicity of groups triggered subsequent work on simplicity and the ideal intersection property for reduced crossed product \mbox{\Cstar-algebras} associated with dynamical systems of discrete groups \cite{bryderkennedy2018-twisted,bryder2017,kawabe17,kennedyschafhauser2019}.  In all of this work, the \Cstar-algebra being acted on is unital, which in the setting of a dynamical system translates to the assumption that the underlying topological space is compact.  The next advance in this direction was obtained by Borys \cite{borys2020-thesis,borys2019-boundary}, who introduced an analogue of the Furstenberg boundary for {\'e}tale Hausdorff groupoids with compact unit space that enabled him to prove the ideal intersection property for groupoids with no confined amenable sections of isotropy groups (but, as above, see Remark~\ref{rem:NotAllStabAmen}). This necessary condition was a partial analogue of results obtained by Kennedy \cite{kennedy15-cstarsimplicity} for groups and by Kawabe \cite{kawabe17} for topological dynamical systems.  The most recent advance was achieved by Kalantar-Scarparo \cite{kalantarscarparo2021}, who utilised the Alexandrov one-point compactification to extend results of Kawabe to group actions on locally compact spaces.

Our proof of Theorem~\ref{thmintro:characterisation-confined-subgroups-minimal-case} begins with the study of the ideal intersection property.  We first treat in Theorem~\ref{thm:seciso} groupoids with compact space of units, and subsequently obtain a result for groupoids with locally compact space of units in Section~\ref{sec:alexandrov-groupoid}.  For groupoids that do not necessarily arise from a group action, a one-point compactification of the unit space leads to the notion of Alexandrov groupoid $\cG^+$, which allows us to obtain the following theorem from the corresponding result for groupoids with compact space of units.
\begin{introtheorem}[See Theorems~\ref{thm:seciso} and~\ref{thm:characterisation-locally-compact-case}]
  \label{thmintro:characterisation-confined-subgroups}
    Let $\cG$ be an {\'e}tale groupoid with locally compact Hausdorff space of units.  Assume that $\cG$ is topologically transitive, and that either $\cG$ is Hausdorff or $\cG^+$ is $\sigma$-compact.  Then $\Cstaress(\cG)$ has the ideal intersection property if and only if $\cG$ has no essentially confined amenable sections of isotropy groups.  Further, $\cG^+$ is $\sigma$-compact if $\cG$ is $\sigma$-compact.
\end{introtheorem}
Note that we actually prove a more general statement (see Theorems~\ref{thm:seciso} and~\ref{thm:characterisation-locally-compact-case} for details). The reason why an additional hypothesis like topological transitivity is required is the surprising observation that stabilizer groups of the Hamana-Furstenberg groupoid (see Theorem~\ref{thmintro:injective} below) are not always amenable, in contrast to the case of transformation groupoids arising from group actions on compact spaces. We present a concrete and elementary example in Remark~\ref{rem:NotAllStabAmen} (see also the remark at the end of the introduction).

Following the paradigm established by \cite{kalantarkennedy2017, breuillardkalantarkennedyozawa14}, an important intermediate step in establishing our elementary characterisation of the ideal intersection property is the control of equivariant ucp maps defined on the essential groupoid \Cstar-algebras of $\cG$ by means of a suitable boundary, which implicitly requires a $\cG$-action on the groupoid \Cstar-algebra.  However, the lack of such an action has been recognised as a major obstruction to extending the previous-mentioned work on groups and dynamical systems to the present setting. In particular, the \Cstar-algebra $\Cstaress(\cG)$ is not a $\cG$-\Cstar-algebra in the sense of \cite{renault87-produits-croises}, which is also the definition employed in \cite{borys2020-thesis, borys2019-boundary}, and requires a $\cG$-\Cstar-algebra to be fibered over the unit space of the groupoid.  Therefore, an important step in developing the results of the present article is the introduction of a new notion of groupoid action on \Cstar-algebras.  For this, we replace elements of the groupoid with elements from the pseudogroups of open bisections \cite{lawsonlenz13}. Groupoid actions on \Cstar-algebras are then defined in terms of families of hereditary \mbox{\Cstar-subalgebras} with *-isomorphisms associated to open bisections of the groupoid. See Section~\ref{sec:groupoid-cstar-algebras} for details.

With respect to the above definition, we prove that $\Cstaress(\cG)$ is naturally a $\cG$-\Cstar-algebra. The following result is an important first step towards Theorem~\ref{thmintro:characterisation-confined-subgroups}.
\begin{introtheorem}[See Proposition~\ref{prop:ell_infty_injective} and Theorem~\ref{thm:injective_envelope}]
  \label{thmintro:injective}
    Let $\cG$ be an {\'e}tale groupoid with compact Hausdorff unit space.  Then $\linfty(\cG)$ is an injective object in the category of $\cG$-\Cstar-algebras.  Further, there is a unique injective envelope of $\cont(\Gnaught)$ in the category of $\cG$-\Cstar-algebras, which is commutative.
\end{introtheorem}
The spectrum of the $\cG$-injective envelope of $\cont(\Gnaught)$ is, by definition, the Hamana boundary $\hb \cG$ of $\cG$.  Expressed using this terminology, Borys actually constructed the Hamana boundary of an {\'e}tale Hausdorff groupoid with compact space of units, considered within the category of classical $\cG$-\Cstar-algebras from \cite{renault87-produits-croises}.  We show in Theorem~\ref{thm:injective_envelope} that, in the setting considered by Borys, his construction of the Hamana boundary agrees with our construction of the Hamana boundary in our larger category of $\cG$-\Cstar-algebras.  We also develop a dynamical approach to boundary theory for an {\'e}tale groupoid $\cG$ with compact Hausdorff unit space and construct the Furstenberg boundary $\fb \cG$ of $\cG$ within this framework.  Specifically, we consider the category of $\cG$-flows, which are compact Hausdorff spaces equipped with a $\cG$-action.  We single out the $\cG$-boundaries, which are the $\cG$-flows that are both minimal and strongly proximal in an appropriate sense. The Furstenberg boundary $\fb \cG$ is the universal $\cG$-boundary, meaning that every $\cG$-boundary is the image of $\fb \cG$ under a morphism of $\cG$-flows.  Using the universal properties satisfied by the Hamana boundary and the Furstenberg boundary, we establish in Theorem~\ref{thm:hamana-equals-furstenberg} that they coincide.  This identification is an important ingredient in the study of Powers averaging property, leading to Theorem~\ref{thmintro:powers-averaging} below.

Having developed a boundary theory for groupoids and, specifically, having constructed the Hamana boundary, we are able to introduce an equivariant analogue of Pitts' pseudo-expectations \cite{pitts2017}, in particular obtaining a natural $\cG$-pseudo-expectation $\Cstaress(\cG) \to \cont(\fb \cG)$ in Section~\ref{sec:local-conditional-expectation}. This provides a new perspective on work of Kwa{\'s}niewski-Meyer on local conditional expectations \cite{kwasniewskimeyer2019-essential} in terms of \Cstar-simplicity theory.  We are able to prove the following characterisation of the ideal intersection property, which is the foundation for all of our subsequent results.
\begin{introtheorem}[See Theorem~\ref{thm:intersection-property-essential-algebras}]
  \label{thmintro:characterisation-ucp-classification}
Let $\cG$ be an {\'e}tale groupoid with compact Hausdorff space of units.  Assume that $\cG$ is Hausdorff, that $\cG$ is minimal or that $\cG$ is $\sigma$-compact.  Then the following statements are equivalent.
  \begin{itemize}
  \item $\cont(\fb \cG) \subseteq \Cstaress(\cG \ltimes \fb \cG)$ has the ideal intersection property.
  \item Every $\cG$-pseudo expectation $\Cstaress(\cG) \to \cont(\fb \cG)$ is faithful.
  \item $\cont(\Gnaught) \subseteq \Cstaress(\cG)$ has the ideal intersection property.
  \end{itemize}

  Assume, in addition to the assumptions above, that $\cG$ is topologically transitive, in the sense that there exists a point in $\Gnaught$ with dense $\cG$-orbit. Then the three statements above are also equivalent to
  \begin{itemize}
  \item $\Iso(\cG \ltimes \fb \cG) = \ol{\fb \cG}^{\cG \ltimes \fb \cG}$.
  \item There is a unique $\cG$-pseudo expectation $\Cstaress(\cG) \to \cont(\fb \cG)$.
  \end{itemize}
\end{introtheorem}
Note that we actually prove a more general statement (see Theorem~\ref{thm:intersection-property-essential-algebras} for details).

Let us mention Remark~\ref{rem:reduced-groupoid-cstar-algebra-intersection-property}, which shows that the intersection property for $\Cstarred(\cG \ltimes \fb \cG)$ implies that $\cG$ is Hausdorff.  Hence, for non-Hausdorff groupoids, there is no possible variant of Theorem~\ref{thmintro:characterisation-ucp-classification} that solely relies on the reduced groupoid \Cstar-algebra.

Research on \Cstar-simplicity began with an averaging argument devised by Powers \cite{powers75}, based on the Dixmier averaging theorem for von Neumann algebras.  Following the work on \mbox{\Cstar-simplicity} in \cite{kalantarkennedy2017,breuillardkalantarkennedyozawa14}, it was subsequently shown independently by Haagerup \cite{haagerup15} and Kennedy \cite{kennedy15-cstarsimplicity} that the Powers averaging property characterises \Cstar-simplicity.  That is, $G$ is a \Cstar-simple group if and only if the norm-closed convex hull of the set of $G$-conjugates of an element of the reduced group \Cstar-algebra contains the trace of the element.

This is not only of aesthetic value, but is the basis for transferring simplicity results to related operator algebras, as happened for example with Hecke \Cstar-algebras in \cite{caspersklisselarsen19,klisse2021-simplicity} and in Phillips' work on $\rL^p$-simplicity \cite{phillips19}. 

Very recently, the correct analogue of Powers averaging for crossed product \Cstar-algebras associated with group actions on compact spaces was developed in \cite{amrutamursu2021}, using the concept of a generalised probability measure.  We introduce a suitable notion of generalised probability measure along with a suitable variant of the Powers averaging property for {\'e}tale groupoids in Section~\ref{sec:powers-averaging}.  We then obtain a corresponding characterisation of simple essential \Cstar-algebras generalising the results in \cite{amrutamursu2021}, as well as the results about relative versions of Powers averaging property in \cite{amrutamkalantar2020,amrutam2018,ursu2021-relative-simplicity}.

An inclusion $A \subseteq B$ of unital \Cstar-algebras was termed \Cstar-irreducible in \cite{rordam2021} if every intermediate \Cstar-algebra is simple.  Such inclusions are an important ingredient in an emerging \mbox{\Cstar-algebraic} analogue of Jones' subfactor theory \cite{jones83}. In particular, \Cstar-irreducible inclusions arising from group \Cstar-algebras, crossed products and groupoids of germs have recently received a great deal of interest \cite{amrutamkalantar2020,amrutamursu2021,rordam2021,kalantarscarparo2021}.  The relative Powers averaging property implies \mbox{\Cstar-irreducibility} of the associated inclusion of \Cstar-algebras.
\begin{introtheorem}[See Theorem~\ref{thm:powers-averaging}]
  \label{thmintro:powers-averaging}
  Let $\cG$ be a minimal {\'e}tale groupoid with compact Hausdorff space of units.  Then the following statements are equivalent.
  \begin{itemize}
  \item $\Cstaress(\cG)$ is simple.
  \item $\Cstaress(\cG)$ satisfies the relative Powers averaging property with respect to any covering and contracting semigroup of generalised probability measures.
  \item $A \subseteq \Cstaress(\cG)$ is \Cstar-irreducible for every \Cstar-subalgebra $A \subseteq \Cstaress(\cG)$ supporting a covering and contracting semigroup of generalised probability measures.
  \end{itemize}
\end{introtheorem}

Applying Theorem~\ref{thmintro:powers-averaging} to suitable subgroups of the topological full group, we obtain many examples of groups of unitaries satisfying the relative Powers averaging property.  We remark that in view of the degeneration phenomena described in \cite{brixscarparo2019}, it is important to allow for the consideration of proper subgroups of the topological full group.
\begin{introcorollary}[See Corollary~\ref{cor:contracting-groups-implies-irreducibility}]
  \label{corintro:irreducible-topological-full-group}
    Let $\cG$ be an {\'e}tale groupoid with compact Hausdorff space of units.  Assume that there is a subgroup of the topological full group $G \leq \mathbf{F}(\cG)$ that covers $\cG$ and such that $G \grpaction{} \Gnaught$ is a $G$-boundary.  Denote by $\pi: G \to \Cstaress(\cG)$ the unitary representation of $G$ in the essential groupoid \Cstar-algebra of $\cG$.  If $\Cstaress(\cG)$ is simple, then $\Cstaress(\cG)$ satisfies Powers averaging property relative to $\pi(G)$.
\end{introcorollary}


Considering groupoids of germs, simplicity of the associated essential groupoid \Cstar-algebra can be guaranteed thanks to topological freeness and minimality.  Applied to this situation, our Theorem~\ref{corintro:irreducible-topological-full-group} yields the following result, extending work of Kalantar-Scarparo \cite{kalantarscarparo2020,kalantarscarparo2021}.
\begin{introtheorem}[See Theorem~\ref{thm:relative-powers-averaging-boundary-action}]
  \label{thmintro:relative-powers-averaging-boundary-action}
    Let $G$ be a countable discrete group and $G \grpaction{} X$ a boundary action.  Denote by $\cG$ its groupoid of germs and by $\pi: G \to \Cstaress(\cG)$ the associated unitary representation.  Then $\pi(G) \subseteq \Cstaress(\cG)$ satisfies the relative Powers averaging property.
\end{introtheorem}
Let us explain the terminology in the statement of the previous theorem, referring to Section~\ref{sec:examples} for more details.  An action of a discrete group $G \grpaction{} X$ is a boundary action if $X$ is compact and the action is minimal and strongly proximal.  These actions were introduced by Furstenberg \cite{furstenberg63} and  further developed by Glasner \cite{glasner1976}, and are of fundamental importance in topological dynamics.  Given an action of a discrete group $G \grpaction{} X$, its groupoid of germs is the quotient groupoid $G \ltimes X / \Iso(G \ltimes X)^\circ$, dividing out the interior of its isotropy from the associated transformation groupoid.

Recall that Thompson's group $\rF$ consists of the piecewise linear homeomorphisms of $[0,1)$ with slopes in $2^\ZZ$ and breakpoints in $\ZZ[\frac{1}{2}]$.  The amenability of $\rF$ is a major open problem in group theory.  It is a subgroup of Thompson's group $\rT$, which is the subgroup of homeomorphisms of $\rS^1$ generated by $\rF$ under the identification of $[0,1)$ and $\rS^1$ together rotations by an angle in $2\pi \ZZ[\frac{1}{2}]$.  The quasi-regular representation of $\rT$ with respect to $\rF$ has received a great deal of attention since Haagerup-Olesen and Le Boudec-Matte Bon considered it within the context of the amenability of $\rF$ \cite{haagerupolesen17,leboudecmattebon18}. The amenability of $\rF$ is equivalent to the statement that this quasi-regular representation is weakly contained in the regular representation of $\rT$.  This fact motivated Kalantar-Scarparo to establish the simplicity of the \Cstar-algebra generated by this representation in \cite{kalantarscarparo2020}.  Thanks to Theorem~\ref{thmintro:relative-powers-averaging-boundary-action}, we obtain a significant strengthening of their result.
\begin{introexample}[See Example~\ref{ex:irreducible-inclusion-cuntz-algebra}]
  \label{exintro:cuntz-irreducible-inclusion}
  Denote by $K$ the totally disconnected cover of $\rS^1$, doubling dyadic integer points.  Let $\cG$ be the groupoid of germs of the action $\rT \grpaction{} K$ and $\pi: \rT \to \Cstarred(\cG)$ the associated unitary representation.  Then $\pi(\rT) \subseteq \Cstarred(\cG)$ satisfies the relative Powers averaging property.
\end{introexample}
Since it is known that $\Cstarred(\cG) \cong \cO_2$, in the setting of Examples~\ref{exintro:cuntz-irreducible-inclusion}, the above result shows that if $\rF$ is amenable, then there is a unitary representation of $\rT$ into the Cuntz algebra $\cO_2$ that enjoys the relative Powers averaging property.  Many more examples of unitary representations satisfying the relative Powers averaging property can be obtained from groups of homeomorphisms of the circle, as we explain in Remark~\ref{rem:circle-dynamics-powers-averaging}.

\begin{remarkNN}
We found examples (see Remark~\ref{rem:NotAllStabAmen}) which show that, contrary to what we claimed in the previous version of our paper and what is also claimed in \cite[Proposition~4.2.16]{borys2020-thesis} and \cite[Proposition~5.1]{borys2019-boundary}, stabilizer groups of the Hamana-Furstenberg groupoid need not be amenable. This surprising observation, which reveals a phenomenon not present in the case of groups and transformation groupoids arising from group actions on compact spaces, led to this revised version of our paper. Compared to the previous version, the main changes occur in Sections~\ref{sec:characterisations} and \ref{sec:confined-subgroupoids}. In particular, we had to modify the notion of essentially confined section of isotropy subgroups. However, note that for minimal groupoids, stabilizer groups of the Hamana-Furstenberg groupoid are always amenable (see Proposition~\ref{prop:SomeStabAmen} and Corollary~\ref{cor:MinAllStabAmen}).
\end{remarkNN}

\subsection*{Organisation of the article}

This article has 9 sections.  After this introduction, Section~\ref{sec:preliminaries} describes preliminary results and fixes notation concerning groupoids and inverse semigroups.  In Section~\ref{sec:groupoid-cstar-algebras}, we introduce our new notion of groupoid actions on \Cstar-algebras.  In Section~\ref{sec:boundary-theory}, we prove Theorem~\ref{thmintro:injective} and develop the dynamical approach to boundary theory for {\'e}tale groupoids and prove the identification of the Furstenberg and the Hamana boundary.  In Section~\ref{sec:non-hausdorff-groupoids}, we study essential groupoid \Cstar-algebra of groupoids with compact space of units from the point of view of the Furstenberg boundary and show that there is a natural inclusion of \Cstar-algebras $\Cstaress(\cG) \subseteq \Cstaress(\cG \ltimes \fb \cG)$ if $\cG$ is either minimal or $\sigma$-compact.  In Section~\ref{sec:characterisations} we prove some fundamental characterisations of the ideal intersection property for groupoids with compact space of units, eventually leading to Theorem~\ref{thmintro:characterisation-ucp-classification}.  In Section~\ref{sec:confined-subgroupoids}, we introduce the notion of essentially confined sections of isotropy groups, as well as the Alexandrov groupoid, leading to a proof of Theorem~\ref{thmintro:characterisation-confined-subgroups}.  Theorem~\ref{thmintro:characterisation-confined-subgroups-minimal-case} is a special case of this.  In Section~\ref{sec:powers-averaging}, we single out the appropriate notion of relative Powers averaging property for essential groupoid \Cstar-algebras and prove Theorem~\ref{thmintro:powers-averaging} as well as Corollary~\ref{corintro:irreducible-topological-full-group}.  In Section \ref{sec:examples}, we apply the previous results and prove Theorem~\ref{thmintro:relative-powers-averaging-boundary-action}.  We also describe Example~\ref{exintro:cuntz-irreducible-inclusion}.

\subsection*{Acknowledgements}

M. K.\ supported by the Canadian Natural Sciences and Engineering Research Council (NSERC) grant number 50503-10787.

\noindent 
S.-J. K.\ and X. L.\ received funding from the European Research Council (ERC) under the European Union's Horizon 2020 research and innovation programme (grant agreement No. 817597).

\noindent S. R.\ is supported by the Swedish Research Council through grant number 2018-04243 and by the European Research Council (ERC) under the European Union's Horizon 2020 research and innovation programme (grant agreement no. 677120-INDEX).  He is grateful to Adam Skalski and Piotr Nowak for their hospitality at IM PAN, where part of this work was completed.

\noindent D. U.\ supported by the  Canadian Natural Sciences and Engineering Research Council (NSERC) grant number PGSD3-535032-2019.

\vspace{1em}

\noindent The authors are grateful to Jean Renault for pointing out the references \cite{fackskandalis1982} and \cite{renault1991} and to Eduardo Scarparo for useful remarks on an earlier version of this article.

\section{Preliminaries}
\label{sec:preliminaries}

\subsection{Groupoids and their \Cstar-algebras}
\label{sec:groupoids}

For basics on {\'e}tale groupoids and their \Cstar-algebras, we refer the reader to Renault's book \cite{renault80} and Sims' lecture notes \cite{sims2020}.

A groupoid is a small category whose morphisms are invertible.  We denote by $\rmr$ and $\rms$ the range and source map of a groupoid and we adopt the convention that $g \cdot h$ is defined if $\rmr(h) = \rms(g)$.  We denote by $\Gnaught = \{g g^{-1} \mid g \in \cG\} \subseteq \cG$ the set of units of a groupoid $\cG$ and by $\Iso(\cG) = {\{g \in \cG \mid \rms(g) = \rmr(g)\}}$ its isotropy bundle.  For $x,y \in \Gnaught$, we denote by $\cG_x = \rms^{-1}(x)$ and $\cG^x = \rmr^{-1}(x)$ the fibres of range and source map.  We also write $\cG_x^x = \cG_x \cap \cG^x$ for the isotropy group at $x$.

A topological groupoid is a groupoid equipped with a topology such that multiplication and inversion become continuous and the range and source maps are open.  An {\'e}tale groupoid is a topological groupoid whose range and source maps are local homeomorphisms.  Every {\'e}tale groupoid $\cG$ has a basis of its topology consisting of open bisections, that is open subsets $U \subseteq \cG$ such that $s|_U$ and $r|_U$ are homeomorphisms onto their image.  A topological groupoid $\cG$ is effective if it satisfies $\Iso(\cG)^\circ = \Gnaught$, that is the interior of its isotropy bundle equals its space of units.  We will need the following proposition, which generalises \cite[Proposition 4.2.19]{borys2020-thesis} to arbitrary {\'e}tale groupoids with extremally disconnect space of units.
\begin{proposition}
  \label{prop:isotropy-clopen}
  Let $\cG$ be an {\'e}tale groupoid with extremally disconnected, locally compact space of units.  Then the isotropy bundle of $\cG$ is clopen.
\end{proposition}
\begin{proof}
  It follows from continuity of range and source map and Hausdorffness of $\Gnaught$ that $\Iso(\cG)$ is closed.  Let $g \in \Iso(\cG)$ and let $U$ be a compact open bisection containing $g$.  Since range and source map are continuous and open, $U$ defines a partial homeomorphism $\vphi: \rms(U) \to \rmr(U)$ between clopen subsets of $\Gnaught$.  So \cite[Proposition 2.11]{pitts2017} (see also \cite[Proof of Theorem 1]{arhangelskii2000}) implies that $\Fix(\vphi) \subseteq \Gnaught$ is clopen.  Now $(s|_U)^{-1}(\Fix(\vphi)) \subseteq U \subseteq \cG$ is an open neighbourhood of $g$ inside $\Iso(\cG)$.  So $\Iso(\cG)$ is open.
\end{proof}

Given a topological groupoid $\cG$, a $\cG$-space is a topological space with a surjection $p:X \to \Gnaught$ and a continuous action map $\cG \tensor[_\rms]{\times}{_p} X \to X$ satisfying the natural associativity condition.  Here
\begin{gather*}
  \cG \tensor[_\rms]{\times}{_p} X = \{(g,x) \in \cG \times X \mid \rms(g) = p(x)\}
\end{gather*}
denotes the fibre product with respect to $\rms$ and $p$. 


Given a $\cG$-space $X$, one defines the transformation groupoid $\cG \ltimes X$ as the topological space $\cG \tensor[_\rms]{\times}{_p} X$ equipped with the range and source maps $\rms(g,x) = x$ and $\rmr(g,x) = gx$, respectively, and the multiplication $(g,x)(h,y) = (gh,y)$.

A subgroupoid $\cH \subseteq \Iso(\cG)$ is normal if $ghg^{-1} \in \cH$ holds for all $h \in \cH$ and all $g \in \cG_{\rmr(h)}$.  Assuming that $\cH \subseteq \Iso(\cG)$ is an open normal subgroupoid of an {\'e}tale groupoid, the quotient space $\cG/\cH$ carries a canonical structure of an {\'e}tale groupoid.   It does not need to be Hausdorff.

The next definition of groupoid \Cstar-algebras for not-necessarily Hausdorff groupoids goes back to Connes' work on foliations \cite{connes1982-foliations}.  We also refer to the article of Khoshkam-Skandalis for the construction of the regular representation for non-Hausdorff groupoids \cite{khoshkamskandalis2002}.
\begin{definition}
  \label{def:groupoid-operator-algebras}
  Let $\cG$ be an {\'e}tale groupoid with locally compact Hausdorff space of units.  Denote by $\cC(\cG)$ the linear span inside $\linfty(\cG)$ of all subspaces $\contc(U)$, where $U \subseteq \cG$ runs through open bisections.  For $x \in \Gnaught$, denote by $\lambda_x: \cC(\cG) \to \bo(\ltwo(\cG_x))$ the convolution representation satisfying $\lambda_x(f)\delta_g = \sum_{h \in \cG_{\rmr(g)}} f(h) \delta_{hg}$.  We adapt the following notation:
  \begin{itemize}
  \item The maximal groupoid \Cstar-algebra $\Cstarmax(\cG)$ is the universal enveloping \Cstar-algebra of $\cC(\cG)$.
  \item The reduced groupoid \Cstar-algebra $\Cstarred(\cG)$ is the \Cstar-completion of $\cC(\cG)$ with respect to the family of *-representations $(\lambda_x)_{x \in \Gnaught}$.
  \item The universal enveloping von Neumann algebra of the maximal groupoid \Cstar-algebra is denoted by $\Wstar(\cG) = \Cstarmax(\cG)^{**}$.
  \end{itemize}
\end{definition}

To each element $a \in \Cstarred(\cG)$ we associate the function
\begin{gather*}
  g \mapsto \hat a(g) = \langle \lambda_{\rms(g)}(a) \delta_{\rms(g)}, \delta_g \rangle
  \eqstop
\end{gather*}
If $f \in \cC(\cG)$, then $\hat f = f$ holds.  We remark that unless $\cG$ is Hausdorff, there are always non-continuous functions in $\cC(\cG)$ and a fortiori elements $a \in \Cstarred(\cG)$ such that $\hat a$ is non-continuous.

While the correct definition of the maximal and reduced groupoid \Cstar-algebra for non-Hausdorff groupoids was already clarified in the 1980's by Connes, it is only much more recently that a clear picture of the essential groupoid \Cstar-algebra has emerged.   After work of Khoshkam-Skandalis \cite{khoshkamskandalis2002}, Exel \cite{exel2011}, Exel-Pitts \cite{exelpitts2019-characterizing} and Clark-Exel-Pardo-Sims-Starling \cite{clarkexelpardosimsstarling2019}, the abstract framework of local conditional expectations described in \cite{kwasniewskimeyer2019-essential} allowed to define the essential crossed product of semigroup actions on Fell bundles of bimodules.  Even in the setup of groupoid \mbox{\Cstar-algebras}, removing any action from the picture, the work of Kwa{\'s}niewski-Meyer provided the first complete and systematic account on this problem (see \cite[End of Section 4]{kwasniewskimeyer2019-essential}).  It slightly differs from the treatment given in \cite{clarkexelpardosimsstarling2019} in as far as the reference to functions with meager strict support allows for a more systematic framework than considering functions whose strict support has empty interior.

For a locally compact Hausdorff space $X$ denote by $\Binfty(X)$ the algebra of all bounded Borel functions on $X$ and by $\cM^\infty(X)$ its ideal of functions with meager support.  We denote by $\Dix(X) = \Binfty(X) / \cM^\infty(X)$ the Dixmier algebra of $X$.  Then by a result of Gonshor, $\Dix(X)$ is isomorphic with the local multiplier algebra $\rM_{\mathrm{loc}}(\conto(X))$ of $\conto(X)$ and the injective envelope of $\conto(X)$ in the category of \Cstar-algebras with *-homomorphisms as morphisms \cite{gonsher1970}.  Recall that $\rM_{\mathrm{loc}}(\conto(X)) = \varinjlim \contb(U)$, where $U$ runs through dense open subsets of $X$.

Let $\cG$ be an {\'e}tale groupoid with locally compact Hausdorff space of units.  The local conditional expectation $\Ered: \Cstarred(\cG) \to \rM_{\mathrm{loc}}(\conto(\Gnaught))$ defined by Kwa{\'s}niewski-Meyer is characterised by the formula $\Ered(f) = f|_U$ for every $f \in \cC(\cG)$, where $U \subseteq \Gnaught$ is a dense open subset on which $f$ is continuous.  Thanks to  \cite[Proposition 4.3]{kwasniewskimeyer2019-essential}, it can be identified with the continuous extension of the natural map $\cC(\cG) \to \Dix(\Gnaught)$ induced by restriction.

There is a representation of $\Cstarred(\cG)$ into the adjointable operators on a Hilbert-$\Dix(\Gnaught)$-module associated with $\Ered$ by the KSGNS-construction \cite[Theorem 5.6]{lance1995-toolkit}.
\begin{definition}
  \label{def:essential-groupoid-cstar-algebra}
  Let $\cG$ be an {\'e}tale groupoid with locally compact Hausdorff space of units.  The \emph{essential groupoid \Cstar-algebra} $\Cstaress(\cG)$ of $\cG$ is the range of the representation of $\Cstarred(\cG)$ on the Hilbert-$\Dix(\Gnaught)$-module associated with $\Ered$.
\end{definition}

By construction, the essential groupoid \Cstar-algebra comes with a generalised conditional expectation $\Eess: \Cstaress(\cG) \to \Dix(\Gnaught)$, which is faithful by \cite[Theorem 4.11]{kwasniewskimeyer2019-essential}.  We follow Kwa{\'s}niewski-Meyer's development and call the kernel of the map $\Cstarred(\cG) \to \Cstaress(\cG)$ its ideal of \emph{singular elements} and denote it by $J_{\mathrm{sing}}$.  By faithfulness of $\Eess$, we have
\begin{gather*}
  J_{\mathrm{sing}} = \{a \in \Cstarred(\cG) \mid \Ered(a^*a) = 0\}
  \eqstop
\end{gather*}
For groupoids covered by countably many open bisections, \cite[Proposition 7.18]{kwasniewskimeyer2019-essential} shows that $J_{\mathrm{sing}}$ consists exactly of those elements such that $\rms(\supp \hat a) \subseteq \Gnaught$ is meager, thereby connecting to the treatment in \cite{clarkexelpardosimsstarling2019}.  For general groupoids the following description of elements vanishing under the local conditional expectation, extracted from \cite[Section 4]{kwasniewskimeyer2019-essential}, is useful.  We give a proof for the reader's convenience.
\begin{proposition}
  \label{prop:singular-elements-vanishing-generically}
  Let $\cG$ be an {\'e}tale groupoid with locally compact Hausdorff space of units.  If $a \in \ker(\Ered)$, then $\hat{a}$ vanishes on a dense subset of $\Gnaught$.
\end{proposition}
\begin{proof}
  Let $a \in \ker(\Ered)$ and let $(a_n)_{n \in \NN}$ be a sequence in $\cC(\cG)$ converging to $a$ in $\Cstarred(\cG)$.  Then also $a_n = \hat a_n \to \hat a$ in $\| \cdot \|_\infty$.  For every $n \in \NN$ there is a dense open subset $U_n \subseteq \Gnaught$ such that $a_n|_{U_n}$ is continuous and $\Ered(a_n) = a_n|_{U_n} \in \contb(U_n) \subseteq \rM_{\mathrm{loc}}(\conto(\Gnaught))$.  Since $\Ered(a_n) \to \Ered(a) = 0$, it follows that $\|a_n|_{U_n}\|_\infty \to 0$.  Let $U = \bigcap_n U_n$, which is a comeager subset of $\Gnaught$.  Since $\Gnaught$ is locally compact, we infer that $U \subseteq \Gnaught$ is dense.  We also have $\hat a|_U = \lim \hat a_n|_U = 0$. 
\end{proof}

We say that an inclusion of \Cstar-algebras $A \subseteq B$ has the \emph{ideal intersection property} if zero is the only ideal $I \unlhd B$ satisfying $I \cap A = 0$.  We will in particular consider the intersection property for $\conto(\Gnaught) \subseteq \Cstarred(\cG)$ and $\conto(\Gnaught) \subseteq \Cstaress(\cG)$.  In these cases we will write that $\Cstarred(\cG)$ and $\Cstaress(\cG)$, respectively, have the ideal intersection property.

  We will need several times the following criterion when, given a normal open subgroupoid $\cN \subseteq \Iso(\cG)$, the canonical quotient map $\cG \to \cG/\cN$ induces a quotient map $\Cstarred(\cG) \to \Cstarred(\cG/\cN)$.  While our criterion seems to be a folklore result, we are not aware of any written account.  We therefore provide a full proof for the convenience of the reader.
\begin{proposition}
  \label{prop:quotient-reduced-groupoid-cstar-algebras}
  Let $\cG$ be an {\'e}tale groupoid with locally compact Hausdorff space.  Let $\cN \subseteq \Iso(\cG)$ be a normal open subgroupoid such that there exists a dense subset $\mathfrak{Y}$ of $\Gnaught$ such that $\cN_x$ is amenable for all $x \in \mathfrak{Y}$. Further suppose that $\cG/\cN$ is Hausdorff. Let $p: \cG \to \cG/\cN$ be the canonical quotient map.  Then there is a unique *-homomorphism $\pi: \Cstarred(\cG) \to \Cstarred(\cG/\cN)$ which restricts to the natural isometry $p_*: \contc(U) \to \contc(p(U))$ on every open bisection $U \subseteq \cG$.
\end{proposition}
\begin{proof}
  Write $\cH = \cG/\cN$ for the quotient and observe that the quotient map $p:\cG \to \cH$ restricts to a homeomorphism on every open bisection of $\cG$.  So there is a well-defined *-homomorphism $\pi_{\mathrm{alg}}: \cC(\cG) \to \cC(\cH)$.  Let us show that $\pi_{\mathrm{alg}}$ extends to a *-homomorphism $\pi: \Cstarred(\cG) \to \Cstarred(\cH)$.  Fix $x \in \mathfrak{Y} \subseteq \Hnaught \cong \Gnaught$ and we will show that $\lambda^\cH_x \circ \pi_{\mathrm{alg}}$ extends to $\Cstarred(\cG)$, where $\lambda^\cH_x$ denotes the left-regular representation $\lambda^\cH_x: \Cstarred(\cH) \to \bo(\ltwo(\cH_x))$.  Since the group $\cN_x^x$ is amenable, we can find a F{\o}lner net $(F_i)_i$ in there and consider the states
  \begin{gather*}
    \vphi_i = \frac{1}{|F_i|} \sum_{g \in F_i} \langle \lambda^\cG_x( \cdot ) \delta_{g}, \delta_{g} \rangle_{\ltwo(\cG_x)}
    \eqstop
  \end{gather*}
  Passing to a subnet, we may assume that $\vphi_i \to \vphi \in (\Cstarred(\cG))^*$ in the weak-*-topology.  Then an elementary computation shows that for all open bisections $U \subseteq \cG$ and all $f \in \contc(U)$ we have
  \begin{gather*}
    \vphi(f) =
    \begin{cases}
      f(g) & \text{if } U \cap \cN_x^x = \{g\} \eqcomma \\
      0 & \text{if } U \cap \cN_x^x = \emptyset \eqstop
    \end{cases}
  \end{gather*}
  So $\vphi|_{\cC(\cG)} = \langle \lambda^\cH_x \circ \pi_{\mathrm{alg}}( \cdot ) \delta_x , \delta_x \rangle$.  For $a,b \in \cC(\cG)$ we find that
  \begin{align*}
    \langle \lambda^\cH_x \circ \pi_{\mathrm{alg}} (a^*a) ( \lambda^\cH_x \circ \pi_{\mathrm{alg}}(b) \delta_x) , \lambda^\cH_x \circ \pi_{\mathrm{alg}}(b) \delta_x \rangle 
    & = 
      \vphi(b^*a^*ab) \\
    & \leq
      \|a\|_{\Cstarred(\cG)}^2 \vphi(b^*b) \\
    & =
      \|a\|_{\Cstarred(\cG)}^2  \|\lambda^\cH_x \circ \pi_{\mathrm{alg}}(b) \delta_x \|^2
      \eqstop
  \end{align*}
Since $\pi_{\mathrm{alg}}(\cC(\cG)) = \cC(\cH)$ and $\cC(\cH)$ acts cyclically on $\ltwo(\cH_x)$, this shows that $\|\lambda^\cH_x \circ \pi_{\mathrm{alg}}(a)\| \leq \|a\|_{\Cstarred(\cG)}$. Thus $\Vert \pi_{\mathrm{alg}}(a) \Vert_{\Cstarred(\cH)} = \sup_{x \in \mathfrak{Y}} \Vert \lambda^\cH_x \circ \pi_{\mathrm{alg}}(a) \Vert \leq \Vert a \Vert_{\Cstarred(\cG)}$. For the first equality, we used the assumption that $\cH$ is Hausdorff. This proves that $\pi_{\mathrm{alg}}$ extends to a *-homomorphism $\pi_{\mathrm{red}}: \Cstarred(\cG) \to \Cstarred(\cH)$.
\end{proof}

\subsection{Inverse semigroups and pseudogroups of open bisections}
\label{sec:pseudogroups}

We refer the reader to Lawson's book \cite{lawson1998} for a comprehensive introduction to inverse semigroups. Recall that an inverse semigroup is a semigroup $S$ such that for every $s \in S$ there is a unique $s^* \in S$ such that $ss^*s = s$ and $s^*ss^* = s^*$.  The following notion of pseudogroups goes back to Resende \cite{resende2007} and Lawson-Lenz \cite{lawsonlenz13}.
\begin{definition}
  \label{def:pseudogroup}
  Let $S$ be an inverse semigroup.  Elements $s,t \in S$ are \emph{compatible} if $s^*t$ and $t^*s$ are idempotent.  Given a family of compatible elements $(s_i)_i$ in $S$ its join is the minimal element $s \in S$ such that $s_i \leq s$ for all $i$.  We say that $S$ has \emph{infinite compatible joins} if every if any compatible family in $S$ admits a join.  A \emph{pseudogroup} is an inverse semigroup with infinite compatible joins in which multiplication distributes over arbitrary joins.
\end{definition}

\begin{example}[Proposition 2.1 of \cite{lawsonlenz13}]
  \label{ex:pseudogroup-bisections}
  Let $\cG$ be an {\'e}tale groupoid with compact Hausdorff space of units.  Then the semigroup of open bisections $\Gamma(\cG)$ is a pseudogroup when equipped with the multiplication $U \cdot V = \{gh \mid g \in U, h \in V, \rms(g) = \rmr(h)\}$.  Then $U^* = \{g^{-1} \mid g \in U\}$ follows.
\end{example}

\begin{notation}
  \label{notation:notation-open-bisections}
  In this article, we will refer to open bisections of an {\'e}tale groupoid $\cG$ either as subsets, usually denoted by $U, V \subseteq \cG$ or alternatively as elements of the pseudogroup of $\cG$, usually denoted by $\gamma \in \Gamma(\cG)$. Both kinds of notation make sense in different contexts.  We denote by $\supp \gamma = \gamma^* \gamma = \rms(\gamma)$ the support and by $\im \gamma = \gamma \gamma^* = \rmr(\gamma)$ the image of $\gamma \in \Gamma(\cG)$.
\end{notation}

\section{A new notion of groupoid actions on C$^*$-algebras}
\label{sec:groupoid-cstar-algebras}

In this section we introduce a new notion of groupoid actions on \Cstar-algebras that will allow us to apply methods from the toolbox of \Cstar-simplicity developed over the past years.  The key point is that the reduced groupoid \Cstar-algebra $\Cstarred(\cG)$, as well as the essential groupoid \Cstar-algebra $\Cstaress(\cG)$, become $\cG$-\Cstar-algebras, which is not the case for existing notions of groupoid actions on \Cstar-algebras used in \cite{borys2019-boundary,borys2020-thesis}.

The next example frames the action of a groupoid on its base space in a way that is compatible with the perspective of operator algebras and pseudogroups.  It will serve as a building block for the subsequent definition of groupoid actions on \Cstar-algebras.
\begin{example}
  \label{ex:pseudogroup-action-unit-space}
  Let $\cG$ be an {\'e}tale groupoid with locally compact Hausdorff space of units $\Gnaught$.  Every open bisection $\gamma \in \Gamma(\cG)$ defines a partial homeomorphism $\psi_\gamma = \rmr|_\gamma \circ (\rms|_\gamma)^{-1}: \supp \gamma \to \im \gamma$.  Dually, we obtain a *-homomorphism $\alpha_\gamma: \conto(\supp \gamma) \to \conto(\im \gamma)$ by the assignment $\alpha_\gamma(f) = f \circ \psi_{\gamma^*} = f \circ \psi_\gamma^{-1}$.  Associativity of the multiplication in $\Gamma(\cG)$ implies that $\psi_{\gamma_1 \gamma_2} = \psi_{\gamma_1} \circ \psi_{\gamma_2}$ on $\psi_{\gamma_2^*}(\supp \gamma_1 \cap \im \gamma_2)$ and $\alpha_{\gamma_1} \circ \alpha_{\gamma_2} = \alpha_{\gamma_1 \gamma_2}$ on $\conto(\psi_{\gamma_2^*}(\supp \gamma_1 \cap \im \gamma_2))$.
\end{example}

Let us also introduce the following notation for hereditary \Cstar-subalgebras.
\begin{notation}
  \label{not:hereditary-subalgebra}
  Let $X$ be compact Hausdorff space, $A$ a unital \Cstar-algebra and $\cont(X) \subseteq A$ a unital inclusion of \Cstar-algebras.  For an open subset $U \subseteq X$ we denote by
  \begin{gather*}
    A_U = \ol{\conto(U) A \conto(U)}
  \end{gather*}
  the hereditary \Cstar-subalgebra of $A$ associated with $U$.
\end{notation}

We are now ready to formulate our definition of groupoid actions on \Cstar-algebras.  In Proposition~\ref{prop:bundle-G-cstar-algebra-integration-desintegration} we will compare this new definition with the classical definition of $\cG$-\Cstar-bundles introduced by Renault \cite{renault87-produits-croises}, showing that our definition is a suitable generalisation.  We will only require this definition for groupoids with compact space of units, fitting the needs of Sections~\ref{sec:boundary-theory} and~\ref{sec:characterisations}.
\begin{definition}
  \label{def:groupoid-cstar-algebra}
  Let $\cG$ be an {\'e}tale groupoid with compact Hausdorff space of units.  A unital $\cG$-\Cstar-algebra is a unital \Cstar-algebra $A$ with an injective unital \mbox{*-homomorphism} \mbox{$\iota: \cont(\Gnaught) \to A$} and a family of *-isomorphisms $\alpha_\gamma: A_{\supp \gamma} \to A_{\im \gamma}$ indexed by $\gamma \in \Gamma(\cG)$ such that
    \begin{itemize}
    \item for all $\gamma \in \Gamma(\cG)$ and all $f \in \conto(\supp \gamma)$ we have $\iota(f \circ \psi_{\gamma^*}) = \alpha_\gamma \circ \iota(f)$, and
    \item for all $\gamma_1,\gamma_2 \in \Gamma(\cG)$ the following diagram commutes.
      \begin{gather*}
        \xymatrix{
          A_{\psi_{\gamma_2^*}(\supp \gamma _1)} \ar[d]^{\alpha_{\gamma_2}} \ar[r]^{\alpha_{\gamma_1\gamma_2}} & A_{\psi_{\gamma_1}(\im \gamma_2)} \\
          A_{\supp \gamma_1 \cap \im \gamma_2} \ar[ur]_{\alpha_{\gamma_1}}
        }
      \end{gather*}
    \end{itemize}
\end{definition}
When working with a unital $\cG$-\Cstar-algebra $A$, we will frequently identify $\cont(\Gnaught)$ with its image in $A$ under $\iota$ and suppress the explicit map $\iota$.

\begin{remark}
  \label{rem:smaller-semigroups-acting}
  In the setting of Fell bundles over semigroups as used in \cite{kwasniewskimeyer2019-essential}, it is possible to replace actions of the pseudogroup of bisections by actions of wide subsemigroups \cite[Definition 2.1 and Proposition 2.2]{kwasniewskimeyer2019-essential}.  The analogue of this fact for $\cG$-\Cstar-algebras as introduced here does not hold, since this notion is genuinely noncommutative. An example can be found by considering the discrete groupoid $\cG = \{0,1\} \times \ZZ$ with unit space $\{0,1\}$.  Its group of global bisections can be identified with $\ZZ \oplus \ZZ$.  We identify bisections supported on $0$ and $1$, respectively, with $\{\emptyset\} \times \ZZ$ and $\ZZ \times \{\emptyset\}$.   Then the pseudogroup of bisections can be described as
  \begin{gather*}
    \Gamma(\cG)
    =
    (\ZZ \oplus \ZZ) \sqcup (\{\emptyset\} \times \ZZ) \sqcup (\ZZ \times \{\emptyset\}) \sqcup \emptyset
    \eqstop
  \end{gather*}
  The subsemigroup $S = {(\{\emptyset\} \times \ZZ)} \sqcup (\ZZ \times \{\emptyset\}) \sqcup \emptyset \subseteq \Gamma(\cG)$ is wide in the sense of \cite[Definition~2.1]{kwasniewskimeyer2019-essential}.  Consider the non-central, diagonal, unitary matrix $\mathrm{diag}(1, -1)$ and the embedding as diagonal matrices $\cont(\Gnaught) \cong \CC^2 \subseteq \rM_2(\CC)$.  The action of $\Gamma(\cG)$ on $\rM_2(\CC)$ for which the global bisection $(1,1)$ acts by conjugation with $\mathrm{diag}(1,-1)$ is non-trivial.  However, its restriction to $S$ is trivial.

\end{remark}

A key role in the theory of \Cstar-simplicity is played by unital completely positive maps.  Let us fix the corresponding notion of $\cG$-ucp maps and define the category of unital $\cG$-\Cstar-algebras considered in subsequent sections.
\begin{definition}
  \label{def:G-ucp-map}
  Let $\cG$ be an {\'e}tale groupoid with compact Hausdorff space of units.
  \begin{itemize}
  \item A $\cG$-ucp map between unital $\cG$-\Cstar-algebras $(A,\iota_A, \alpha)$ and $(B, \iota_B, \beta)$ is a unital completely positive map $\vphi: A \to B$ such that $\vphi \circ \iota_A = \iota_B$ holds and the diagram
    \begin{gather*}
      \xymatrix{
        A_{\supp \gamma} \ar[r]^{\alpha_\gamma} \ar[d]^{\vphi|_{\supp \gamma}} & A_{\im \gamma} \ar[d]^{\vphi|_{\im \gamma}} \\
        B_{\supp \gamma} \ar[r]^{\beta_\gamma} & B_{\im \gamma}
      }
    \end{gather*}
    commutes for every $\gamma \in \Gamma(\cG)$, where $\vphi|_U: A_U \to B_U$ is the restriction of $\vphi$.
  \item The category of unital $\cG$-\Cstar-algebras has as its objects unital $\cG$-\Cstar-algebras and as morphisms $\cG$-ucp maps.
  \item For unital $\cG$-\Cstar-algebras $A$ and $B$, a $\cG$-ucp map $\phi : A \to B$ is an embedding if it is a complete order embedding.
  \end{itemize}
\end{definition}

The main motivation to pass to the greater generality of Definition~\ref{def:groupoid-cstar-algebra} is the fact that the maximal groupoid \Cstar-algebra $\Cstarmax(\cG)$ becomes a $\cG$-\Cstar-algebra in a natural way.  Indeed, our approach to $\cG$-\Cstar-algebras via a pseudogroup action allows for a straightforward definition of inner actions, which we will now present.  We need the following preparatory lemma.

Let $\cG$ be an {\'e}tale groupoid with a compact Hausdorff space of units and consider its universal enveloping von Neumann algebra $\Wstar(\cG)$. Fix $\gamma \in \Gamma(\cG)$ and consider the subspace $\contc(\gamma) \subseteq \Cstarmax(\cG)$.  For every $g \in \contc(\gamma)$, the convolution product satisfies $g * g^* \in \cont(\Gnaught)$.  So the \Cstar-identity implies that there is an isometric isomorphism of Banach spaces $\conto(\gamma) \cong \ol{\contc(\gamma)} \subseteq \Cstarmax(\cG)$.  Thus, we obtain an isometric embedding
  \begin{gather*}
    \Binfty(\gamma) \subseteq \conto(\gamma)^{**} \subseteq \Cstarmax(\cG)^{**} = \Wstar(\cG)
    \eqcomma
  \end{gather*}
  where $\Binfty(\gamma)$ denotes the space of bounded Borel functions on $\gamma$.  The net $(f)_{0 \leq f \leq \mathbb{1}_\gamma}$ is monotone and converges pointwise to the indicator function $\mathbb{1}_{\gamma}$ in $\Binfty(\gamma)$.  By the monotone convergence theorem, it also converges in the weak-*-topology.  Since $\conto(\gamma)^{**}$ is weak-*-closed in $\Wstar(\cG)$, the net $(f)_{0 \leq f \leq \mathbb{1}_\gamma}$ converges also in $\Wstar(\cG)$.  We denote its limit by $u_\gamma$.
\begin{lemma}
  \label{lem:pseudogroup-partial-isometries}
  Let $\cG$ be an {\'e}tale groupoid with a compact Hausdorff space of units. The map $\Gamma(\cG) \to \Wstar(\cG): \gamma \mapsto u_\gamma$ is a semigroup homomorphism from $\Gamma(\cG)$ to the set of partial isometries in $\Wstar(\cG)$ such that, for all $\gamma \in \Gamma(\cG)$, the following statements hold.
  \begin{enumerate}
  \item \label{it:pseudogroup-partial-isometries:partial-isometry}
    We have $u_\gamma u_\gamma^* = \mathbb{1}_{\im \gamma}$ and $u_\gamma^* u_\gamma = \mathbb{1}_{\supp \gamma}$.
  \item \label{it:pseudogroup-partial-isometries:compatibility}
    If $f' \in \contc(\supp \gamma)$, then we have $u_\gamma f' = f' \circ \rmm_{\gamma^*}$, where $m_\gamma$ denotes left multiplication with $\gamma$.
  \end{enumerate}
\end{lemma}
\begin{proof}
  Take $f' \in \contc(V)$ for some open bisection $V$ satisfying $\rms(V) \subseteq \supp \gamma$.  Then for every $0 \leq f \leq \mathbb{1}_\gamma$ in $\contc(\gamma)$, we have
  \begin{gather*}
    (f * f') (x) =
    \begin{cases}
      f(y) f'(y^{-1}x) & \text{ for } y \in \gamma \text{ and } y^{-1}x \in V \eqcomma \\
      0 & \text{otherwise.}
    \end{cases}
  \end{gather*}
  Considering the map $\Binfty(\gamma) \to \Binfty(\gamma V): \: f \mapsto f * f'$, the definition of the convolution product shows continuity with respect to the topology of pointwise convergence in $\Binfty(\gamma)$ and $\Binfty(\gamma V)$.  This implies that $u_\gamma f' = f'( \gamma^* \cdot ) = f' \circ \rmm_{\gamma^*}$ holds.
  
  Let us now show that $\gamma \mapsto u_\gamma$ is a semigroup homomorphism.  We have
  \begin{gather*}
    u_{\gamma_1} u_{\gamma_2}
    =
    \lim_{0 \leq f \leq \mathbb{1}_{\gamma_2}} u_{\gamma_1} f
    =
    \lim_{0 \leq f \leq \mathbb{1}_{\gamma_2}} f \circ \rmm_{\gamma_1^*}
    =
    u_{\gamma_1 \gamma_2}
  \end{gather*}
  where the convergence claimed by the last equality holds because of the first part of this proof.  It similarly follows that
  \begin{gather*}
    u_\gamma^* = \lim_{0 \leq f \leq \mathbb{1}_{\gamma}} f^* = u_{\gamma^*}
    \eqcomma
  \end{gather*}
  and
  \begin{gather*}
    u_\gamma u_\gamma^*
    =
    \lim_{0 \leq f \leq \mathbb{1}_{\gamma^*}} u_\gamma f
    =
    \lim_{0 \leq f \leq \mathbb{1}_{\gamma^*}} f \circ \rmm_{\gamma^*}
    =
    \mathbb{1}_{\im \gamma}
    \eqstop
  \end{gather*}
  Combining the previous statements, we also find that
  \begin{gather*}
    u_\gamma^* u_\gamma
    =
    u_{\gamma^*} u_{\gamma^*}^*
    =
    \mathbb{1}_{\im \gamma^*}
    =
    \mathbb{1}_{\supp \gamma}
    \eqstop
  \end{gather*}
\end{proof}

We can now conclude that there is a natural structure of a $\cG$-\Cstar-algebra on the maximal groupoid \Cstar-algebra of an {\'e}tale groupoid with compact unit space.  Applied to the quotient maps $\Cstarmax(\cG) \to \Cstarred(\cG)$ and $\Cstarmax(\cG) \to \Cstaress(\cG)$ it also exhibits natural $\cG$-\Cstar-algebra structures on the reduced and the essential groupoid \Cstar-algebra, respectively.
\begin{proposition}
  \label{prop:groupoid-cstar-algebras-carries-action}
  Let $\cG$ be an {\'e}tale groupoid with compact Hausdorff space of units.  Consider the standard embedding $\cont(\Gnaught) \subseteq \Cstarmax(\cG)$.  Then there is a unique structure $(\alpha_\gamma)_{\gamma \in \Gamma(\cG)}$ of a unital $\cG$-\Cstar-algebra on $\Cstarmax(\cG)$ satisfying
  \begin{gather*}
    \alpha_\gamma (f) = f \circ \rmc_{\gamma^*}
  \end{gather*}
  for all $\gamma \in \Gamma(\cG)$, all open bisections $U \subseteq \cG$ satisfying $\rms(U), \rmr(U) \subseteq \supp \gamma$ and all $f \in \contc(U)$.  Here $\rmc_{\gamma}: h \mapsto \gamma h \gamma^*$ denotes conjugation with $\gamma$.

  Further, if $\pi: \Cstarmax(\cG) \to A$ is a unital *-homomorphism that is injective on $\cont(\Gnaught)$, there is the structure of a unital $\cG$-\Cstar-algebra on $A$ so that $\pi$ is $\cG$-equivariant.  If $\pi$ is surjective, this structure is unique.
\end{proposition}
\begin{proof}
  Consider the partial isometries $u_\gamma$, $\gamma \in \Gamma(\cG)$ provided by Lemma~\ref{lem:pseudogroup-partial-isometries}.  It follows from this lemma that the inclusion $\cont(\Gnaught) \subseteq \Cstarmax(\cG)$ and the maps $\Ad u_\gamma: \Wstar(\cG)_{\supp \gamma} \to \Wstar(\cG)_{\im \gamma}$ turn the enveloping von Neumann algebra into a $\cG$-\Cstar-algebra in such a way that the formula $\Ad u_\gamma (f) = f \circ \rmc_{\gamma^*}$ holds for all $\gamma \in \Gamma(\cG)$, all open bisections $U \subseteq \cG$ satisfying $\rms(U), \rmr(U) \subseteq \supp \gamma$ and all $f \in \contc(U)$.  In order to show that $\Cstarmax(\cG)$ is a unital $\cG$-algebra satisfying the conditions of the proposition, it suffices to fix $\gamma \in \Gamma(\cG)$ and show that $\Ad u_\gamma$ maps $\Cstarmax(\cG)_{\supp \gamma}$ to $\Cstarmax(\cG)$.  To this end, recall that $\Cstarmax(\cG)_{\supp \gamma} = \ol{\conto(\supp \gamma) \Cstarmax(\cG) \conto(\supp \gamma)}$.  Invoking item~\ref{it:pseudogroup-partial-isometries:compatibility} of Lemma~\ref{lem:pseudogroup-partial-isometries} we thus find that
  \begin{align*}
    u_\gamma \Cstarmax(\cG)_{\supp \gamma} u_\gamma^*
    & =
    u_\gamma \ol{\contc(\supp \gamma) \Cstarmax(\cG) \contc(\supp \gamma)} u_\gamma^* \\
    & \subset
    \ol{\cC(\cG) \Cstarmax(\cG) \cC(\cG)} \\
    & =
    \Cstarmax(\cG)
    \eqstop
  \end{align*}
  Note that this is the unique $\cG$-\Cstar-algebra structure on $\Cstarmax(\cG)$ satisfying the conditions of the proposition, since
  \begin{gather*}
    \lspan_{\rms(U) \subseteq \supp \gamma} \contc(U)
    =
    \conto(\supp \gamma) \cC(\cG) \conto(\supp \gamma)
    \subseteq
    \Cstarmax(\cG)_{\supp \gamma}
  \end{gather*}
  is dense.

  Assume now that $\pi: \Cstarmax(\cG) \to A$ is a unital *-homomorphism that is injective on $\cont(\Gnaught)$.  We identify $\cont(\Gnaught)$ with a \Cstar-subalgebra of $A$, that is $\pi|_{\cont(\Gnaught)} = \id$.  For $\gamma \in \Gamma(\cG)$, we start by defining a contractive *-homomorphism
  \begin{gather*}
    \beta_{\gamma, \rmc}:
    \lspan \contc(\supp \gamma) A \contc(\supp \gamma)
    \to
    \lspan \contc(\im \gamma) A \contc(\im \gamma)
    \eqstop
  \end{gather*}
  Let $a = \sum_{i = 1}^n f_{1, i} a_i f_{2, i} \in \lspan \contc(\supp \gamma) A \contc(\supp \gamma)$.  Put $K = \bigcup_{i = 1}^n \supp f_{1,i} \cup \supp f_{2, i}$.  Then $K \subseteq \supp \gamma$ is compact, so that there is some function $g \in \contc(\supp \gamma)$ satisfying $0 \leq g \leq 1$ and $g|_K \equiv 1$.  Observe that $u_\gamma g \in \cC(\cG)$.  We claim that the expression $\pi(u_\gamma g) a \pi(u_\gamma g)^*$ does not depend on the choice of $g$.  Indeed, if $h \in \contc(\supp \gamma)$ is another function satisfying $0 \leq h \leq 1$ and $h|_K \equiv 1$, then
  \begin{align*}
    \pi(u_\gamma g h ) a \pi(u_\gamma g h)^*
    & =
      \sum_{i = 1}^n \pi(u_\gamma g h ) f_{1, i} a_i f_{2, i} \pi( h g u_{\gamma^*} ) \\
    & =
      \sum_{i = 1}^n \pi(u_\gamma f_{1, i}) a_i  \pi(f_{2, i} u_{\gamma^*} ) \\
    & =
      \pi(u_\gamma g) a \pi(u_\gamma g)^*
      \eqstop
  \end{align*}
  Also observe that
  \begin{gather*}
    \pi(u_\gamma g) a \pi(u_\gamma g)^*
    =
    \sum_{i = 1}^n \alpha_{\gamma}(f_{1, i}) \pi(u_\gamma) a_i  \pi( u_{\gamma^*}) \alpha_\gamma(f_{2, i}) \\
  \end{gather*}
  So we can put $\beta_{\gamma, \rmc}(a) = \pi(u_\gamma g) a \pi(u_\gamma g)^*$.  Since $\beta_{\gamma, \rmc}$ is contractive it extends to a *-homomorphism $\beta_\gamma: A_{\supp \gamma} \to A_{\im \gamma}$.

  If $\gamma_1, \gamma_2 \in \Gamma(\cG)$ satisfy $\im \gamma_2 \cap \supp \gamma_1 \neq \emptyset$, we want to show that $\beta_{\gamma_1 \gamma_2} = \beta_{\gamma_1} \circ \beta_{\gamma_2}|_{A_{\supp \gamma_1 \gamma_2}}$.  To this end let $K \subseteq \supp \gamma_1 \gamma_2$ be a compact subset, let $g_2 \in \contc(\supp \gamma_1 \gamma_2)$ satisfy $0 \leq g_2 \leq 1$ and $g_2|_K \equiv 1$ and let $g_1 \in \contc(\im \gamma_2 \cap \supp \gamma_1)$ satisfy $g_1|_{\vphi_{\gamma_2}(K)} \equiv 1$.  Then Lemma~\ref{lem:pseudogroup-partial-isometries} implies that
  \begin{gather*}
    u_{\gamma_1}f_1 u_{\gamma_2} f_2
    =
    u_{\gamma_1} u_{\gamma_2} \alpha_{\gamma_2^*}(f_1) f_2
    =
    u_{\gamma_1 \gamma_2} \alpha_{\gamma_2^*}(f_1) f_2
    \eqstop
  \end{gather*}
  Since $\alpha_{\gamma_2^*}(f_1) f_2 |_K \equiv 1$, this implies
  \begin{gather*}
    \beta_{\gamma_1 \gamma_2, \rmc}
    =
    \beta_{\gamma_1, \rmc} \circ \beta_{\gamma_2, \rmc}|_{\lspan \contc(\supp \gamma_1 \gamma_2) A \contc(\supp \gamma_1 \gamma_2)}
    \eqstop
  \end{gather*}
  By continuity the desired equality follows.

    Assume now that $\pi: \Cstarmax(\cG) \thra A$ is surjective *-homomorphism and let $\beta_\gamma: A_{\supp \gamma} \to A_{\im \gamma}$ define some $\cG$-action on $A$ making $\pi$ equivariant.  Fix $\gamma \in \Gamma(\cG)$ and $a \in A_{\supp \gamma}$.  By approximation, we may assume that $a = \pi(f) a_0 \pi(f)$ for some $f \in \conto(\supp \gamma)$ and $a_0 \in A$.  Let $b_0 \in \Cstarmax(\cG)$ be a preimage of $a_0$ and put $b = f b_0 f$.  Then $\pi(b) = a$ and thus
  \begin{gather*}
    \beta_\gamma(a)
    =
    \beta_\gamma(\pi(b))
    =
    \pi(\Ad u_\gamma (b))
    =
    \pi(u_\gamma) \pi(b) \pi(u_\gamma^*)
    = \Ad \pi(u_\gamma) (a)
    \eqstop
  \end{gather*}
  This shows uniqueness of the $\cG$-\Cstar-algebra structure on $A$.
\end{proof}

Next we will show that the classical notion of $\cG$-\Cstar-algebras introduced in \cite{renault87-produits-croises} is in a precise sense subsumed by our notion.  Recall that given a compact Hausdorff space $X$, a unital $\cont(X)$-algebra is a unital \Cstar-algebra $A$ with a unital inclusion $\cont(X) \to \cZ(A)$ into the centre of $A$.  Given such algebra, the evaluation maps $\ev_x: \cont(X) \to \CC$ extend to quotient maps $A \to A_x$ onto the fibres of a \Cstar-bundle.
\begin{definition}
  \label{def:groupoid-cstar-bundle}
  Let $\cG$ be an {\'e}tale groupoid with compact Hausdorff space of units.  A unital $\cG$-\Cstar-bundle is a unital $\cont(\Gnaught)$-algebra $A$ with an associative and continuous map $\alpha: \cG \prescript{}{\rms}{\times} \bigsqcup_{x \in \Gnaught} A_x \to \bigsqcup_{x \in \Gnaught} A_x$ such that $\alpha_x = \id$ for all $x \in \Gnaught$ and $\alpha_g \circ \alpha_h = \alpha_{gh}$ whenever $\rms(g) = \rmr(h)$.

  A $\cG$-ucp map between unital $\cG$-\Cstar-bundles $(A, \alpha)$ and $(B, \beta)$ is a unital $\cont(\Gnaught)$-modular unital completely positive map $\vphi: A \to B$ satisfying $\vphi_{\rmr(g)} \circ \alpha_g = \beta_g \circ \vphi_{\rms(g)}$ for all $g \in \cG$.
\end{definition}

Let us make the following ad-hoc definition.
\begin{notation}
  \label{not:compatible-structures}
  Let $A$ be a \Cstar-algebra.  Writing groupoid actions implicitly, we call a structure of a $\cG$-\Cstar-bundle on $A$ and the structure of a $\cG$-\Cstar-algebra on $A$ compatible, if they define the same inclusion $\cont(\Gnaught) \subseteq A$ and the equality
    \begin{gather*}
      (\alpha_{\gamma} a)_{\rmr(g)} =  g a_{\rms(g)}
    \end{gather*}
    is satisfied for every $\gamma \in \Gamma(\cG)$, every $g \in \gamma$ and every $a \in A_{\supp \gamma}$.
\end{notation}


\begin{proposition}
  \label{prop:bundle-G-cstar-algebra-integration-desintegration}
  Let $\cG$ be an {\'e}tale groupoid with compact Hausdorff space of units.
  \begin{itemize}
  \item If $A$ is a unital $\cG$-\Cstar-bundle, then there is a unique compatible structure of a unital $\cG$-\Cstar-algebra on $A$.
  \item If $A$ is a unital $\cG$-\Cstar-algebra such that $\cont(\Gnaught) \subseteq A$ is central, then there is a unique compatible structure of a unital $\cG$-\Cstar-bundle on $A$.
  \item If $A$ and $B$ are unital $\cG$-\Cstar-algebras such that $\cont(\Gnaught)$ is central in $A$ and $B$ and $\vphi: A \to B$ is a unital completely positive map, then $\vphi$ is $\cG$-equivariant in the sense of $\cG$-\Cstar-algebras if and only if it is $\cG$-equivariant in the sense of $\cG$-\Cstar-bundles.
  \end{itemize}
\end{proposition}
\begin{proof}
  Let $A$ be a unital $\cG$-\Cstar-bundle.  Let $\gamma \in \Gamma(\cG)$ be an open bisection, $f_1, f_2 \in \conto(\supp \gamma)$ and $a \in A$.  Then $f_1 a f_1  \in  A_{\supp \gamma}$ and
  \begin{gather*}
    \supp \gamma \ra \cG \tensor[_s]{\times}{} \bigsqcup_{x \in \Gnaught} A_x, \, x \mapsto ((s|_\gamma)^{-1}(x), (f_1 a f_2)_{\rms(g)})
  \end{gather*}
  extends by zero to a continuous section of $\cG \tensor[_s]{\times}{} \bigsqcup_{x \in \Gnaught} A_x$.  Its image under the action of $\cG$ defines a continuous section in $\bigsqcup_{x \in \Gnaught} A_x$ whose support lies in $\im \gamma$.  We denote this element by $\alpha_\gamma (f_1 a f_2) \in A_{\im \gamma} \subseteq A$.  The map $f_1 a f_2 \mapsto \alpha_\gamma (f_1 a f_2)$ is bounded on $\conto(\supp \gamma) A \conto(\supp \gamma)$ and thus extends to a bounded map $A_{\supp \gamma} \to A_{\im \gamma}$.  It is straightforward to check that this defines a compatible structure of a unital $\cG$-\Cstar-algebra on $A$.

  Vice versa, assume that $(A, \alpha)$ is a unital $\cG$-\Cstar-algebra such that $\cont(\Gnaught) \subseteq A$ is central.  Let $g \in \cG$ and $a_{\rms(g)} \in A_{\rms(g)}$.  Let $a \in A$ be a lift of $a_{\rms(g)}$ and $\gamma \in \Gamma(\cG)$ an open bisection containing $g$.  Let $f_1,f_2 \in \contc(\supp \gamma)$ be such that $0 \leq f_1,f_2 \leq 1$ and $f_1(\rms(g)) = 1 = f_2(\rms(g))$.  Then
  \begin{gather*}
    \alpha_\gamma(f_1 f_2 a f_2 f_1)|_{\rmr(g)}
    =
    f_1(\gamma^*\rmr(g))^2 \alpha_\gamma(f_2 a f_2)_{\rmr(g)}
    =
    \alpha_\gamma(f_2 a f_2)_{\rmr(g)}
    \eqstop
  \end{gather*}
  This shows well-definedness of a map $\cG \tensor[_\rms]{\times}{} \bigsqcup_{x \in \Gnaught} A_x \to \bigsqcup_{x \in \Gnaught} A_x$ satisfying the compatibility condition of Notation~\ref{not:compatible-structures}.  A straightforward calculation shows that this defines a groupoid action, and it remains to check continuity.  We denote by $p$ the map from the bundle $\bigsqcup_{x \in \Gnaught} A_x$ to its base space.  Let $(g,a_{\rms(g)}) \in \cG \tensor[_\rms]{\times}{} \bigsqcup_{x \in \Gnaught} A_x$ and take a basic neighbourhood of $g a_{\rms(g)}$, given by
  \begin{gather*}
    N(\tilde b, U , \veps)
    =
    \{b \in \bigsqcup_{x \in \Gnaught} A_x \mid p(b) \in U \text{ and } \| b - \tilde b_{p(b)} \| < \veps\}
  \end{gather*}
  for $\tilde b \in A$, $U \subseteq \Gnaught$ open and $\veps > 0$.  Without loss of generality, we may reduce the size of $U$ in order to assume that there is an open bisection $\gamma \in \Gamma(\cG)$ that contains $g$ and satisfies $\ol{U} \subseteq \im \gamma$.  Further, replacing $\tilde b$ by $f \tilde b f$ for some $f \in \conto(\im \gamma)$ satisfying $f|_{\ol{U}} \equiv 1$, we may assume that $\tilde b \in A_{\im \gamma}$.  Put $\tilde a = \gamma^* \tilde b$. Then by construction $N(\tilde a, U, \veps)$ maps into $N(\tilde b, U, \veps)$, proving continuity of the action at $(g, a_{\rms(g)})$.

  A straightforward calculation shows the notions of $\cG$-equivariance for ucp maps agree.
\end{proof}

\section{Boundary theory for \'{e}tale groupoids}
\label{sec:boundary-theory}

Boundary theory for discrete groups has played an important role in the analysis of the algebraic structure of the reduced \Cstar-algebra of the group. A key result established in \cite{kalantarkennedy2017} is that the Furstenberg boundary and the Hamana boundary of a discrete group coincide, that is the \Cstar-algebra of continuous functions on the Furstenberg boundary is the unique essential and injective object in the category of \Cstar-dynamical systems.  In this section, we will construct the Hamana boundary and the Furstenberg boundary of any {\'e}tale groupoid $\cG$ with compact Hausdorff unit space and identify the two.

A Hamana boundary of an {\'e}tale Hausdorff groupoid $\cG$ with compact unit space was constructed by Borys in \cite{borys2019-boundary,borys2020-thesis} and directly termed Furstenberg boundary.  In this section, we denote the object he constructed by $\bb \cG$.   He constructed $\cont(\bb \cG)$ as the injective envelope of $\cont(\Gnaught)$ in the category of concrete $\cG$-operator systems, which agrees with the injective envelope in the category of $\cG$-\Cstar-bundles.  

We will construct the Hamana boundary $\hb \cG$ of $\cG$, for groupoids that are not necessarily Hausdorff.  The Hamana boundary is the spectrum of the injective envelope of $\cont(\Gnaught)$ in the category of unital $\cG$-\Cstar-algebras as introduced in Section~\ref{sec:groupoid-cstar-algebras}.  This terminology highlights that the identification with the Furstenberg boundary is an {\`a} posteriori result and honours Hamana's development of the theory of injective envelopes \cite{hamana79-operator-systems,hamana85-c-star-dynamical-systems}.  Our proof of the existence of injective envelopes in the category of unital $\cG$-\Cstar-algebras builds on Sinclair's proof from \cite{sinclair2015} of the existence of injective envelopes in the category of operator systems. The key technical device utilised in his proof is the existence of idempotents in compact right topological semigroups. Utilising the convexity of the specific semigroup under consideration allows us to further deduce the rigidity and essentiality of the Hamana boundary.

We also develop a dynamical approach to boundary theory for an {\'e}tale groupoid $\cG$ with compact Hausdorff unit space and construct the Furstenberg boundary $\fb \cG$ of $\cG$ within this framework.  Specifically, consider the category of $\cG$-flows, which are compact Hausdorff spaces equipped with a $\cG$-action.  We single out the $\cG$-boundaries, which are the $\cG$-flows that are both minimal and strongly proximal in a sense made precise in Definition~\ref{def:boundary-actions}.  The Furstenberg boundary $\fb \cG$ is the universal $\cG$-boundary, meaning that every $\cG$-boundary is the image of $\fb \cG$ under a morphism of $\cG$-flows.

Once we have established the existence and uniqueness of the Furstenberg boundary $\fb \cG$, we will prove that it coincides with the Hamana boundary $\hb \cG$.  Moreover, we will prove that if $\cG$ is Hausdorff, then $\fb \cG$ also coincides with the boundary $\bb \cG$ constructed by Borys.

Both perspectives on the Furstenberg boundary developed here are of critical importance to our results. We will utilise the operator algebraic approach to the Furstenberg boundary throughout this paper, and in particular when we introduce pseudo-expectations in Section \ref{sec:characterisations}. We will utilise the dynamical approach to the Furstenberg boundary when we consider Powers averaging property in Section \ref{sec:powers-averaging}.

\subsection{Groupoid actions on states and probability measures}
\label{sec:groupoid-spaces}

Let $\cG$ be an {\'e}tale groupoid with compact Hausdorff unit space and let $A$ be a unital $\cG$-\Cstar-algebra.  Let $\iota_* : \cS(A) \to \cP(\Gnaught)$ denote the restriction map, where $\cS(A)$ denotes the state space of $A$ equipped with the weak* topology and $\cP(\Gnaught)$ denotes the space of probability measures on $\Gnaught$ equipped with the weak* topology. Although the $\cG$-\Cstar-algebra structure on $A$ does not necessarily induce a $\cG$-space structure on $\cS(A)$, it does induce a $\cG$-space structure on a canonical closed subspace of $\cS(A)$.

\begin{definition}
  \label{def:restricted-state-space}
  Let $\cG$ be an {\'e}tale groupoid with compact Hausdorff unit space.  For a unital $\cG$-\Cstar-algebra $A$, we denote by $\cS_{\Gnaught}(A) \subseteq \cS(A)$ the closed subspace defined by
  \begin{gather*}
    \cS_{\Gnaught}(A)
    =
    \{\vphi \in \cS(A) \mid \iota_*\vphi = \delta_x \text{ for some } x \in \Gnaught \}
    \eqstop
  \end{gather*}
\end{definition}

The next lemma shows how the $\cG$-\Cstar-algebras structure on $A$ induces a $\cG$-space structure on $\cS_{\Gnaught}(A)$.
\begin{lemma}
  \label{lem:translated-state-existence}
  Let $\cG$ be an {\'e}tale groupoid with compact Hausdorff space of units.  Let $(A,\alpha)$ be a unital $\cG$-\Cstar-algebra, let $\vphi \in \cS(A)$ be such that $\vphi|_{\cont(\Gnaught)} = \ev_x$ for some $x \in \Gnaught$ and let $g \in \cG_x$.  Let $\gamma$ be an open bisection of $\cG$ containing $g$ and let $f \in \contc(\im \gamma)$ be a positive function satisfying $f(\rmr(g)) = 1$.  The formula
  \begin{gather*}
    g \vphi(a) = \vphi(\alpha_{\gamma^*}(f a f))
  \end{gather*}
  defines a state on $A$ satisfying $(g \vphi)|_{\cont(\Gnaught)} = \ev_{\rmr(g)}$.  It does not depend on the choice of $\gamma$ and $f$.
\end{lemma}
\begin{proof}
  Since the intersection of any pair of open bisections containing $g$ is another open bisection containing $g$, it suffices to observe that for any pair of functions $f,h \in \contc(\im \gamma)$ satisfying $0 \leq f,h \leq 1$ and $f(\rmr(g)) = h(\rmr(g)) = 1$, the equality
  \begin{gather*}
    \vphi(\alpha_{\gamma^*}(hf a fh))
    =
    \ev_x(\alpha_{\gamma^*}(h))^2 \vphi(\alpha_{\gamma^*}(f a f))
    =
    h(\rmr(g))^2 \vphi(\alpha_{\gamma^*}(f a f))
    =
    \vphi(\alpha_{\gamma^*}(f a f))
  \end{gather*}
  holds.  It is clear that $(g \vphi)|_{\cont(\Gnaught)} = \ev_{\rmr(g)}$, showing in particular that $g \vphi$ is a state.
\end{proof}

Let us consider the special case of commutative $\cG$-\Cstar-algebras separately.  We observe that there is a correspondence between commutative unital $\cG$-\Cstar-algebras $\cont(X)$ and $\cG$-spaces $p: X \to \Gnaught$.  This leads to the following reformulation of Definition~\ref{def:restricted-state-space}.
\begin{definition}
  \label{def:restricted-probability-measures}
  Let $\cG$ be an {\'e}tale groupoid with compact Hausdorff unit space. For a compact $\cG$-space $p : Z \to \Gnaught$, we define
  \begin{gather*}
    \cP_{\Gnaught}(Z)
    =
    \{\nu \in \cP(Z) \mid p_*\nu = \delta_x \text{ for some } x \in \Gnaught\}
    \eqstop
  \end{gather*}
\end{definition}

\begin{remark}
  \label{rem:reforumation-action-restricted-probability-measures}
  It is clear that $\cP_{\Gnaught}(Z)$ is closed in the relative weak* topology, because $\Gnaught \subseteq \cP(\Gnaught)$ is weak-*-closed.  The $\cG$-space structure on $Z$ induces a canonical $\cG$-space structure on $\cP_{\Gnaught}(Z)$. Namely, for $\nu \in \cP_{\Gnaught}(Z)$ and $g \in \cG$ with $\delta_{\rms(g)} = p_* \nu$, we observe that $\supp(\nu) \subseteq p^{-1}(p_*\nu) \subseteq Z$ and we can define the probability measure $g \nu \in \cP(Z)$ by
  \begin{gather*}
    \int_Z f\, \rmd(g\nu)
    =
    \int_{p^{-1}(\rms(g))} f(g z)\, \rmd\nu(z) \qquad \text{for } f \in \cont(Z)
    \eqstop
  \end{gather*}
  This is the dual of the action of $\cG$ on $\cS_{\Gnaught}(\cont(Z))$ described by Lemma~\ref{lem:translated-state-existence}.
\end{remark}

\subsection{The Hamana boundary}
\label{sec:hamana-boundary}

Let $\cG$ be an {\'e}tale groupoid with compact Hausdorff unit space. We first observe that $\linfty(\cG)$ is a $\cG$-\Cstar-algebra.  Since $\Gnaught$ is compact, the canonical embedding of $\cont(\Gnaught)$ into $\linfty(\cG)$ given by $f \mapsto f \circ \rmr$ is a unital *-homomorphism.  Recall from Notation~\ref{not:hereditary-subalgebra} that for an open subset $U \subseteq \Gnaught$,  we write $\linfty(\cG)_U = \overline{\conto(U) \linfty(\cG)}^{\|\cdot\|}$ for the hereditary \Cstar-subalgebra associated with $U$.  It is straightforward to see that
\begin{gather*}
  \linfty(\cG)_U \subseteq \{f \in \linfty(\cG) \mid \supp(f) \subseteq \rmr^{-1}(U) \} = \ol{\linfty(\cG)_U}^{\weakstar}
  \eqstop
\end{gather*}
For an open bisection $\gamma \in \Gamma(\cG)$, the partial homeomorphism of $\gamma$ on $\cG$ defines a *-isomorphism $\beta_\gamma : \linfty(\cG)_{\supp(\gamma)} \to \linfty(\cG)_{\im(\gamma)}$ satisfying
\begin{gather*}
  \beta_\gamma(f)(g)
  =
  \begin{cases*}
    f(\gamma^*g) & if $r(g) \in \im(\gamma)$ \\
    0 & otherwise.
  \end{cases*}
\end{gather*}

The next result characterises $\cG$-\Cstar-algebra morphisms into $\linfty(\cG)$. 
\begin{proposition} 
  \label{prop:poisson_maps}
  Let $\cG$ be an {\'e}tale groupoid with compact Hausdorff unit space. Let $(A, \alpha)$ be a unital $\cG$-\Cstar-algebra. There is a bijective correspondence between the following objects:
  \begin{enumerate}
  \item \label{it:poisson_maps:morphism}
    $\cG$-ucp maps $\phi : A \to \linfty(\cG)$.
  \item \label{it:poisson_maps:base-space}
    Unital completely positive maps $\psi : A \to \linfty(\Gnaught)$ satisfying $\psi|_{\cont(\Gnaught)} = \id_{\cont(\Gnaught)}$.
  \item\label{it:poisson_maps:states}
    Families of states $\{\mu_x\}_{x \in \Gnaught}$ on $A$ satisfying $\mu_x|_{\cont(\Gnaught)} = \delta_x$.
  \end{enumerate}
  For a map $\phi$ as in \ref{it:poisson_maps:morphism}, the map $\psi$ is defined by $\psi(f) = \phi(f)|_{\Gnaught}$.  For a map $\psi$ as in \ref{it:poisson_maps:base-space}, the family $\{\mu_x\}_{x \in \Gnaught}$ is defined by setting $\mu_x = \delta_x \circ \psi$ for each $x \in \Gnaught$.  For a family $\{\mu_x\}_{x \in \Gnaught}$ as in \ref{it:poisson_maps:states}, the map $\phi$ is defined by $\phi(a)(g) = (g\mu_{s(g)})(a)$ for $a \in A$.
\end{proposition}
\begin{proof}
  It is easy to verify the bijective correspondence between maps as in \ref{it:poisson_maps:base-space} and families of states as in \ref{it:poisson_maps:states}.  We will prove the bijective correspondence between maps as in \ref{it:poisson_maps:morphism} and families of states as in \ref{it:poisson_maps:states}. 

  Denote by $(\beta_\gamma)_{\gamma \in \Gamma(\cG)}$ the $\cG$-action on $\linfty(\cG)$.  Let $\{\mu_x\}_{x \in \Gnaught}$ be a family of states on $A$ as in \ref{it:poisson_maps:states}.  Then defining $\phi : A \to \linfty(\cG)$ by $\phi(a)(g) = (g\mu_{\rms(g)})(a)$ for $a \in A$ yields a unital completely positive map $\phi : A \to \linfty(\cG)$.  Furthermore, $\phi$ is the identity on $\cont(\Gnaught)$ since for $f \in \cont(\Gnaught)$ and $g \in \cG$,
  \begin{gather*}
    \phi(f)(g) = (g \mu_{\rms(g)})(f) = \delta_{\rmr(g)}(f) = (f \circ \rmr)(g)
    \eqstop    
  \end{gather*}
  In particular $\phi(\alpha_\gamma(a)) \in \linfty(\cG)_{\im \gamma}$ for all $a \in A_{\supp \gamma}$.  To see that  $\phi$ is $\cG$-equivariant, choose $\gamma \in \Gamma(\cG)$ and $a \in A_{\supp(\gamma)}$.  We must show that $\phi(\alpha_\gamma(a))(g) = \beta_{\gamma}(\phi(a))(g)$ for all $g \in \im \gamma$.
  
  Suppose $g \in \cG$ satisfies $\rmr(g) \in \im \gamma$.  Choose $\eta \in \Gamma(\cG)$ with $g \in \eta$.  Then $\rmr(g) \in \im \gamma  \cap \im \eta$, so in particular $\im \gamma \cap \im\eta \neq \emptyset$.  Choose $f \in \contc(\im \gamma \cap \im \eta)$ such that $f(\rmr(g)) = 1$. Then
  \begin{align*}
    \phi(\alpha_\gamma(a))(g)
    & = (g \mu_{\rms(g)})(\alpha_\gamma(a)) \\
    &= \mu_{\rms(g)} (\alpha_{\eta^*}( f \alpha_\gamma(a) f)) \\
    &= \mu_{\rms(g)} (\alpha_{(\gamma^* \eta)^*} ( \alpha_{\gamma^*}(f) a \alpha_{\gamma^*}(f))) \\
    &= \mu_{\rms(\gamma^* g)} (\alpha_{(\gamma^* \eta)^*} ( \alpha_{\gamma^*}(f) a \alpha_{\gamma^*}(f))) \\
    &= ((\gamma^* g) \mu_{\rms(\gamma^* g)})(a) \\
    &= \phi(a)(\gamma^*g) \\
    &= \beta_\gamma(\phi(a))(g)
      \eqstop
  \end{align*}
  Hence $\phi$ is equivariant.  Moreover, it is clear that $\phi$ satisfies $\delta_x \circ \phi = \mu_x$ for all $x \in \Gnaught$.  Therefore, the map from families of states as in \ref{it:poisson_maps:states} to $\cG$-ucp maps as in \ref{it:poisson_maps:morphism} is injective.
    
  It remains to show that the map is surjective. For this, let $\phi : A \to \linfty(\cG)$ be a $\cG$-ucp map as in \ref{it:poisson_maps:morphism}.  Define a family of states $\{\mu_x\}_{x \in \Gnaught}$ on $A$ by $\mu_x = \delta_x \circ \phi$ and let $\phi' : A \to \linfty(\cG)$ be the corresponding $\cG$-ucp map constructed as above.  For $g \in \cG$, let $\gamma \in \Gamma(\cG)$ be an open bisection with $g \in \gamma$ and let $f \in \contc(\im \gamma)$ be a function satisfying $f(\rmr(g)) = 1$. Then by the equivariance of $\phi$, we have
  \begin{align*}
    \phi'(a)(g) &= (g \mu_{\rms(g)})(a) \\
		&= \mu_{\rms(g)}(\alpha_{\gamma^*}(f a f)) \\
		&= (\delta_{\rms(g)} \circ \phi)(\alpha_{\gamma^*}(faf)) \\
		&= \delta_{\rms(g)}(\beta_{\gamma^*} \circ \phi(faf)) \\
		&= \phi(faf)(\gamma \rms(g)) \\
		&= \phi(faf)(g).
  \end{align*}
  Now using the fact that $\phi$ is a $\cont(\Gnaught)$-bimodule map gives
  \begin{gather*}
    \phi(faf)(g) = f(\rmr(g)) \phi(a)(g) f(\rmr(g)) = \phi(a)(g)
    \eqstop
  \end{gather*}
  Hence $\phi'(a)(g) = \phi(a)(g)$ for all $g \in \cG$, and we conclude that $\phi' = \phi$.  It follows that the map from families of states as in \ref{it:poisson_maps:states} to $\cG$-ucp maps as in \ref{it:poisson_maps:morphism} is surjective.
\end{proof}

For unital $\cG$-\Cstar-algebras $A$ and $B$, recall that a $\cG$-ucp map $\phi : A \to B$ is an embedding if it is a complete order embedding and note that if either of $A$ or $B$ is commutative, then $\phi$ is an embedding if and only if it is isometric.
\begin{definition}
  \label{def:injectivity}
  Let $\cG$ be an {\'e}tale groupoid with compact Hausdorff unit space. We say that a unital $\cG$-\Cstar-algebra $C$ is injective in the category of unital $\cG$-\Cstar-algebras if whenever $A$ and $B$ are unital $\cG$-\Cstar-algebras with an embedding $\iota : A \to B$ and a $\cG$-ucp map $\phi : A \to C$, there is a $\cG$-ucp map $\psi : B \to C$ extending $\phi$, that is the following diagram commutes.
  \begin{center}
    \begin{tikzcd}
      B \arrow[rd, "\psi", dashed]                 &   \\
      A \arrow[u, "\iota", hook] \arrow[r, "\phi"] & C
    \end{tikzcd}
  \end{center}
\end{definition}

With the correspondence from Proposition \ref{prop:poisson_maps}, we are now able to prove the injectivity of $\linfty(\cG)$ in the category of unital $\cG$-\Cstar-algebras.
\begin{proposition}
  \label{prop:ell_infty_injective}
  Let $\cG$ be an {\'e}tale groupoid with compact Hausdorff unit space. The \Cstar-algebra $\linfty(\cG)$ is injective in the category of unital $\cG$-\Cstar-algebras.
\end{proposition}
\begin{proof}
  Let $\iota : A \to B$ be a unital $\cG$-\Cstar-algebra embedding and let $\phi : A \to \linfty(\cG)$ be a $\cG$-ucp map.  Let $\psi : A \to \linfty(\Gnaught)$ be the corresponding unital completely positive map satisfying $\psi|_{\cont(\Gnaught)} = \id_{\cont(\Gnaught)}$ as in Proposition~\ref{prop:poisson_maps}.  Since $\linfty(\Gnaught)$ is a commutative von Neumann algebra, it is injective in the category of operator systems.  Hence there is a unital positive map $\tilde{\psi} : B \to \linfty(\Gnaught)$ extending $\psi$, that is the following diagram commutes.
  \begin{center}
    \begin{tikzcd}
      B \arrow[rd, "\widetilde{\psi}", dashed]                 &                            \\
      A \arrow[u, "\iota", hook] \arrow[r, "\psi"] & \linfty(\Gnaught)
    \end{tikzcd}
  \end{center}
  Applying Proposition \ref{prop:poisson_maps} again, we obtain a $\cG$-ucp map $\tilde{\phi} : B \to \linfty(\cG)$.  For $a \in A$, we have $\tilde{\phi}(a)|_{\Gnaught} = \tilde{\psi}(a) = \psi(a) = \phi(a)|_{\Gnaught}$, so it follows from the correspondence in Proposition~\ref{prop:poisson_maps} that $\tilde{\phi}|_A = \phi$.  We conclude that $\linfty(\cG)$ is injective.
\end{proof}

\begin{definition}
  \label{def:essential-rigid}
  Let $\cG$ be an {\'e}tale groupoid with compact Hausdorff unit space and let $A \subseteq B$ be a $\cG$-ucp embedding of unital $\cG$-\Cstar-algebras.
  \begin{enumerate}
  \item We say that $B$ is a rigid extension of $A$ if any $\cG$-ucp map $\phi : B \to B$ with $\phi|_A = \id_A$ satisfies $\phi = \id_B$.

  \item We say that $B$ is an essential extension of $A$ if whenever $C$ is a unital $\cG$-\Cstar-algebra and $\phi : B \to C$ is a $\cG$-ucp map such that the restriction $\phi|_A$ is an embedding, then $\phi$ is an embedding.

  \item $B$ is a $\cG$-injective envelope of $A$ if $B$ is an essential extension of $A$, and if, in addition, $B$ is injective in the category of unital $\cG$-\Cstar-algebras.
  \end{enumerate}
\end{definition}

The next result establishes the existence and uniqueness of the injective envelope of $\cont(\Gnaught)$ in the category of unital $\cG$-\Cstar-algebras, along with its rigidity property.  The first part of the proof closely follows Sinclair's proof from \cite{sinclair2015} of the existence of the injective envelope of an operator system. By utilising the convexity of the semigroup under consideration, we then deduce rigidity and essentiality. The proof will make use of basic facts about compact right topological semigroups as presented for example in \cite{hindmanstrauss12}.
\begin{theorem}
  \label{thm:injective_envelope}
  Let $\cG$ be an {\'e}tale groupoid with compact Hausdorff unit space.  The $\cG$-\Cstar-algebra $\cont(\Gnaught)$ admits an injective envelope in the category of unital $\cG$-\Cstar-algebras.  It is a rigid extension of $\cont(\Gnaught)$.  It is also a commutative \Cstar-algebra and it is unique up to *-isomorphism.
\end{theorem}
\begin{proof}
  By Proposition \ref{prop:ell_infty_injective}, the commutative von Neumann algebra $\linfty(\cG)$ is injective in the category of unital $\cG$-\Cstar-algebras.  Let $S$ denote the set of $\cG$-\Cstar-ucp maps $\phi : \linfty(\cG) \to \linfty(\cG)$ satisfying $\phi|_{\cont(\Gnaught)} = \id_{\cont(\Gnaught)}$, equipped with the relative point-weak* topology.  Then $S$ is a compact convex right topological semigroup under composition, meaning that for fixed $\psi \in S$ and a net $(\phi_i)$ in $S$ converging to $\phi \in S$, we have $\lim \phi_i \circ \psi = \phi \circ \psi$.  
  
  Since $S$ is a compact right topological semigroup, it contains a minimal left ideal $L \subseteq S$. Note that $L$ is necessarily closed. We claim that $L$ is also a left zero semigroup, in the sense that $\phi \circ \psi = \phi$ for all $\phi,\psi \in L$. To see this, fix $\psi \in L$ and observe that $S \psi$ is a left ideal of $S$ contained in $L$, so the minimality of $L$ implies $L = S \psi$. In particular, since $S$ is convex, this implies that $L$ is convex. The map $L \to L : \phi \mapsto \phi \circ \psi$ is thus a continuous affine map, so it admits a fixed point. Hence $L_0 = \{\phi \in L : \phi \circ \psi = \phi\}$ is a left ideal of $S$ contained in $L$. Applying the minimality of $L$ again implies that $L_0 = L$. Hence $L$ is a left zero semigroup.  This implies that every element of $L$, and in particular $\psi$, is idempotent.
  
  Let $A = \im \psi$.  Since $\psi$ is idempotent, $A$ is a \Cstar-algebra under the Choi-Effros product defined by $a \circ b = \psi(ab)$.  Since $\linfty(\cG)$ is commutative, $A$ is commutative.  Also, since $\psi|_{\cont(\Gnaught)} = \id_{\cont(\Gnaught)}$, it follows that $\cont(\Gnaught)$ is a \Cstar-subalgebra of $A$ and belongs to the multiplicative domain of $\psi$.  Thanks to $\cG$-equivariance of $\psi$, the $\cG$-\Cstar-algebra structure on $\linfty(\cG)$ restricts to a $\cG$-\Cstar-algebra structure on $A$.  Furthermore, since $A$ is the range of an idempotent $\cG$-ucp map from the injective $\cG$-\Cstar-algebra $\linfty(\cG)$, it follows that $A$ is injective in the category of $\cG$-\Cstar-algebras.
  
  Before proving that $A$ is actually the injective envelope of $\cont(\Gnaught)$, it will be convenient to first prove that it is a rigid extension of $\cont(\Gnaught)$. To see this, let $\phi : A \to A$ be a $\cG$-ucp map. Let $\iota : A \to \linfty(\cG)$ denote the $\cG$-embedding as a subset. Then $\iota \circ \phi \circ \psi \in S$ and even $\iota \circ \phi = \iota \circ \phi \circ \psi^2 \in L$. Since $L$ is a left zero semigroup, $\psi = \psi \circ \iota \circ \phi \circ \psi$, which implies that $\phi = \id_A$. Hence $A$ is a rigid extension of $\cont(\Gnaught)$.
  
  To see that $A$ is an essential extension of $\cont(\Gnaught)$, let $C$ be a unital $\cG$-\Cstar-algebra and let $\phi : A \to C$ be a $\cG$-ucp map.  Note that $\phi|_{\cont(\Gnaught)} = \id_{\cont(\Gnaught)}$, so in particular $\phi|_{\cont(\Gnaught)}$ is an embedding.  By the injectivity of $A$, there is a $\cG$-ucp map $\eta : C \to A$ satisfying $(\eta \circ \phi)|_{\cont(\Gnaught)} = \id_{\cont(\Gnaught)}$.  Hence by the rigidity of $A$, we have $\eta \circ \phi = \id_A$, implying that $\phi$ is an embedding.

  Finally, to see that $A$ is unique, let $D$ be an injective envelope of $\cont(\Gnaught)$ in the category of unital $\cG$-\Cstar-algebras.  By the injectivity of $D$, there is a $\cG$-ucp map $\phi : A \to D$ such that $\phi|_{\cont(\Gnaught)}$ is an embedding.  So the essentiality of $A$ implies that $\phi$ is an embedding.  Symmetrically, we obtain an embedding $\psi: D \to A$.  The composition $\psi \circ \phi$, must be the identity map, since $A$ is rigid, which implies that $\phi$ is surjective.  Hence $\phi$ is an isometric complete order isomorphism between \Cstar-algebras, and therefore is a *-isomorphism.
\end{proof}

\begin{definition}
  \label{def:hamana-boundary}
  Let $\cG$ be an {\'e}tale groupoid with compact Hausdorff unit space.  The Hamana boundary $\hb \cG$ of $\cG$ is the spectrum of the injective envelope of $\cont(\Gnaught)$ in the category of unital $\cG$-\Cstar-algebras.
\end{definition}

We observe that $\cont(\hb \cG)$ is also injective in the category of unital \Cstar-algebras. In particular, this implies that $\hb \cG$ is extremally disconnected. We will make use of this fact below.

\subsection{The Furstenberg boundary}
\label{sec:furstenberg-boundary}

Let $\cG$ be an {\'e}tale groupoid with compact Hausdorff unit space. In this subsection we will construct the Furstenberg boundary of $\cG$ by developing an analogue for groupoids of Furstenberg and Glasner's theory of topological dynamical boundaries for groups (see e.g. \cite{glasner1976}).

In the next definition, we make use of the notation $\cG \cdot A = \{g a \mid \rms(g) = p(a) \}$ for a subset $A \subseteq X$ of a $\cG$-space $p: X \to \Gnaught$.
\begin{definition}
  \label{def:boundary-actions}
  Let $\cG$ be an {\'e}tale groupoid with compact Hausdorff space of units and let $p : Y \to \Gnaught$ be a compact $\cG$-space.
  \begin{itemize}
  \item We will say that $Y$ is irreducible if whenever $Z \subseteq Y$ is a closed $\cG$-invariant subspace satisfying $p(Z) = \Gnaught$, then $Z = Y$.
    
  \item We will say that $Y$ is strongly proximal if whenever $(\mu_x)_{x \in \Gnaught}$  is a family of probability measures on $Y$ satisfying $p_* \mu_x = \delta_x$ for all $x \in \Gnaught$, then
    \begin{gather*}
      \Gnaught
      \subseteq
      p(Y \cap \ol{\cG \cdot \{\mu_x \mid x \in \Gnaught\}})
      \eqstop
    \end{gather*}
    
  \item We will say that $Y$ is a $\cG$-boundary if it is both irreducible and strongly proximal.
  \end{itemize}
\end{definition}

\begin{remark}
  \label{rem:G-boundary-irreducible}
  We will use the fact that a compact $\cG$-space $p: Y \to \Gnaught$ is a $\cG$-boundary if and only if for every family $(\mu_x)_{x \in \Gnaught}$ of probability measures $\mu_x \in \cP(Y)$ satisfying $p_* \mu_x = \delta_x$, we have $Y \subseteq \ol{\cG \cdot \{\mu_x \mid x \in \Gnaught\}}$.
\end{remark}

The next proposition provides a characterisation of strong proximality in the minimal setting that will be useful for the arguments in Section~\ref{sec:powers-averaging}.
\begin{proposition}
  \label{prop:strongly-proximal-minimal}
  Let $\cG$ be a minimal {\'e}tale groupoid with compact Hausdorff unit space.  Let $p : Y \to \Gnaught$ be an irreducible $\cG$-space.  Then $Y$ is strongly proximal if and only if the following condition holds: for every $x \in \Gnaught$ and every probability measure $\mu \in \cP(Y)$ satisfying $p_* \mu = \delta_x$, there is $y \in p^{-1}(x)$ and a net $(g_i)$ in $\cG_x$ such that $g_i \mu \to \delta_y$. 
\end{proposition}
\begin{proof}
  If $Y$ satisfies the condition of the proposition, then it is clear that $Y$ is strongly proximal. For the converse, suppose that $Y$ is strongly proximal and that $\mu \in \cP(Y)$ satisfies $p(\mu) = \delta_{x_0}$ for some $x_0 \in \Gnaught$.  By the minimality of $\cG$,
  \begin{gather*}
    p_*( \ol{\cG \cdot \mu} ) = \ol{\cG \cdot p_* \mu} = \Gnaught
    \eqstop
  \end{gather*}
  So we can find a family $(\mu_x)_{x \in \Gnaught}$ in $\ol{\cG \cdot \mu}$ satisfying $p(\mu_x) = \delta_x$ for all $x \in \Gnaught$.  From the strong proximality of $Y$, we infer that
  \begin{gather*}
    p(Y \cap \ol{\cG \cdot \mu})
    \supseteq
    p(Y \cap \ol{\cG \cdot \{\mu_x \mid x \in \Gnaught\}})
    =
    \Gnaught
    \eqcomma
  \end{gather*}
  which finishes the proof of the proposition.
\end{proof}

In Furstenberg and Glasner's theory of topological dynamical boundaries for groups, the affine flow of probability measures on a compact flow plays an important role. In the present setting, a complication arises from the fact that a $\cG$-flow $Z$ does not necessarily induce a $\cG$-flow structure on the entire space $\cP(Z)$ of probability measures on $Z$. Instead, it is necessary to work with the subspace $\cP_{\Gnaught}(Z) \subseteq \cP(Z)$ introduced in Section \ref{sec:groupoid-spaces}.

\begin{proposition}
  \label{prop:rigidity}
  Let $\cG$ be an {\'e}tale groupoid with compact Hausdorff unit space.  Let $p_Z : Z \to \Gnaught$ be a $\cG$-boundary and let $p_Y : Y \to \Gnaught$ be any compact $\cG$-space.
  \begin{enumerate}
  \item The image of every $\cG$-map $Y \to \cP_{\Gnaught}(Z)$ contains $Z$.
  \item If $Y$ is irreducible, then every $\cG$-map $Y \to \cP_{\Gnaught}(Z)$ maps onto $Z$, and there is at most one $\cG$-map $Y \to Z$.
  \end{enumerate}
\end{proposition}
\begin{proof}
  Let $\phi : Y \to \cP_{\Gnaught}(Z)$ be a $\cG$-map.  Then $(p_X)_* \phi(y) = \delta_{p_Y(y)}$ holds for all $y \in Y$. Since $Z$ is a $\cG$-boundary, we have $Z \subseteq \ol{\cG \cdot \phi(Y)}$ by Remark~\ref{rem:G-boundary-irreducible}.  Because $\phi$ is $\cG$-equivariant and $Y$ is compact, we infer that $Z \subseteq \phi(Y)$.
  
  Now assume in addition that that $Y$ is irreducible. The subset $\phi^{-1}(Z) \subseteq Y$ satisfies $p_Y(\phi^{-1}(Z)) = \Gnaught$, so by the irreducibility of $Y$, we have $\phi^{-1}(Z) = Y$. Combined with the previous paragraph, this implies that $\phi(Y) = Z$.  If $\psi: Y \to Z$ is another $\cG$-map, then $\frac{1}{2}(\phi + \psi): Y \to \cP_{\Gnaught}(Z)$ is also a $\cG$-map that, from above, must take values in $Z$. It follows that $\phi = \psi$.
\end{proof}

The next theorem establishes the existence of a Furstenberg boundary in analogy with the classical argument for groups.
\begin{theorem}
  \label{thm:furstenberg-boundary-exists}
  Let $\cG$ be an {\'e}tale groupoid with compact Hausdorff space of units.  There is a $\cG$-boundary $\fb \cG$ that is universal in the sense that for every $\cG$-boundary $Y$ there is a (necessarily surjective) $\cG$-map $\fb \cG \to Y$. Furthermore, $\fb \cG$ is the unique $\cG$-boundary with this property up to isomorphism of $\cG$-spaces.
\end{theorem}

\begin{proof}
  We will prove the existence of $\fb \cG$, whereupon uniqueness will follow from Proposition~\ref{prop:rigidity}. First observe that by irreducibility, the density character of every $\cG$-boundary does not exceed $|\cG|$.   So there is a family of representatives for isomorphism classes of $\cG$-boundaries $q_i: Y_i \to \Gnaught$, $i \in I$.  Consider the fibre product
  \begin{gather*}
    Y
    =
    \prod_I (Y_i, q_i)
    =
    \varprojlim_{i_1, \dotsc, i_n \in I} Y_{i_1} \times_{\Gnaught} Y_{i_2} \times_{\Gnaught} \dotsm \times_{\Gnaught} Y_{i_n}
  \end{gather*}
  and note that $Y$ is compact, being the projective limit of compact Hausdorff spaces.  Let $p: Y \to \Gnaught$ and $p_i: Y \to Y_i$ denote the natural projections.
  
  We now show that $Y$ is strongly proximal.  Let $(\mu_x)_{x \in \Gnaught}$ be a family of probability measures on $Y$ satisfying $p_* \mu_x = \delta_x$ for all $x \in \Gnaught$.  For a finite subset $F \subseteq I$, write $Y_F = \prod_F(Y_i, q_i)$ and denote by $p_F: Y \to Y_F$ and $q_F: Y_F \to \Gnaught$ the natural projections.  We will show by induction on the size of $F$ that
  \begin{gather*}
    \Gnaught
    \subseteq
    q_F(Y_F \cap \ol{\cG \cdot (p_F)_*\{\mu_x \mid x \in \Gnaught\}})
    \eqstop
  \end{gather*}
  The case $|F| = 1$ follows from the assumed strong proximality for all $Y_i$.  Fix a finite set $F \subseteq I$ and assume that the statement is proven for all strictly smaller sets than $F$.  Let $i \in F$.  Then by induction hypothesis, we have
  \begin{gather*}
    \Gnaught
    \subseteq
    q_{F \setminus \{i\}}(Y_{F \setminus \{i\}} \cap \ol{\cG \cdot (p_{F \setminus \{i\}})_*\{\mu_x \mid x \in \Gnaught\}})
    \eqstop
  \end{gather*}

  By the compactness of $\cP(Y)$ this means that for every $x \in \Gnaught$ there is $\nu_x \in \ol{ \cG \cdot \{\mu_x \mid x \in \Gnaught\}}$ such that $(p_{F \setminus \{i\}})_* \nu_x \in Y_{F \setminus \{i\}}$ and $p_* \nu_x = \delta_x$.  By the strong proximality of $Y_i$, we find that
  \begin{gather*}
    \Gnaught
    \subseteq
    q_i(Y_i \cap \ol{\cG \cdot (p_i)_* \{\nu_x \mid x \in \Gnaught\}})
    \eqstop
  \end{gather*}
  Choose probability measures $\sigma_x$, $x \in \Gnaught$ in $\ol{\cG \cdot \{\nu_x \mid x \in \Gnaught\}}$ such that $(p_i)_* \sigma_x \in Y_i$ and $p_* \sigma_x = \delta_x$.  We observe that
  \begin{gather*}
    (p_{F \setminus \{i\}})_* \sigma_x \in (p_{F \setminus \{i\}})_*(\ol{\cG \cdot \{\nu_x \mid x \in \Gnaught\}})
    \subseteq
    Y_{F \setminus \{i\}}
  \end{gather*}
  so that $(p_F)_* \sigma_x \in Y_F$ follows for all $x \in \Gnaught$.  This finishes the induction.
  
  In summary, we have found for every $x \in \Gnaught$ a net of probability measures $(\mu_{x, F})_{F \subseteq I \text{ finite}}$ in $\ol{\cG \cdot \{\mu_x \mid x \in \Gnaught\}}$ such that $p_*(\mu_{x,F}) = \delta_x$ and $(p_F)_*(\mu_{x,F}) \in Y_F$ for all $x \in \Gnaught$ and all finite subsets $F \subseteq I$.  By compactness, there are probability measures $\nu_x$, $x \in \Gnaught$ on $Y$ satisfying $p_*(\nu_x) = \delta_x$ and $(p_F)_*(\nu_x) \in Y_F$ for all $x \in \Gnaught$ and all finite subsets $F \subseteq I$.  Since cylinder sets generate the $\Sigma$-algebra of $Y$, this implies that $\nu_x \in Y$ for all $x \in \Gnaught$, which finishes the proof of strong proximality.
  
  Consider now the family of closed $\cG$-invariant subsets $A \subseteq Y$ that satisfy $p(A) = \Gnaught$.  This family is ordered by inclusion and by compactness satisfies the descending chain condition.  Hence it contains a minimal element, which necessarily will be an irreducible $\cG$-space.  Since it inherits strong proximality from $Y$, this proves the existence of $\fb \cG$.
\end{proof}

\begin{definition}
  \label{def:furstenberg-boundary}
  Let $\cG$ be an {\'e}tale groupoid with compact Hausdorff unit space.  The Furstenberg boundary $\fb \cG$ of $\cG$ is the $\cG$-boundary constructed in Theorem~\ref{thm:furstenberg-boundary-exists}.
\end{definition}

\subsection{Equivalence of boundaries}
\label{sec:equivalence-boundaries}

In this subsection we will prove that the Hamana boundary and the Furstenberg boundary coincide. The proof of the following lemma is based on Milman's partial converse to the Krein-Milman theorem, which states that if $Y$ is a closed subset of a compact convex set $K$ with the property that the closed convex hull of $Y$ is equal to $K$, then $Y$ contains all of the extreme points of $K$.
\begin{lemma}
  \label{lem:krein_milman_lemma}
  Let $\cG$ be an {\'e}tale groupoid with compact Hausdorff unit space.  Let $Y$ be a compact Hausdorff space identified with the closed subset of Dirac measures in $\cP(Y)$. Let $(\mu_i)_{i \in I}$ be a family of probability measures on $Y$. Then $Y \subseteq \overline{\{\mu_i\}_{i \in I}}^{\weakstar}$ if and only if the map $\bigoplus_{i \in I} \mu_i : \cont(Y) \to \linfty(I)$ is isometric.
\end{lemma}
\begin{proof}
  Suppose that the map $\bigoplus_{i \in I} \mu_i : \cont(Y) \to \linfty(I)$ is isometric but that there is $y \in Y$ such that $y \notin \overline{\{\mu_i\}_{i \in I}}^{\weakstar}$. Then letting $K \subseteq \cP(Y)$ denote the closed convex hull of the set $\{\mu_i\}_{i \in I}$, we infer that $K$ is a proper subset of $\cP(Y)$.  Hence by the Hahn-Banach separation theorem there is positive $f \in \cont(Y)$ and $\alpha \geq 0$ such that
  \begin{gather*}
    \sup_{i \in I} \mu_i(f) \leq \alpha < f(y)
    \eqcomma
  \end{gather*}
  contradicting the fact that the map $\bigoplus_{i \in I} \mu_i$ is isometric.
  
  Conversely, suppose that $Y \subseteq \overline{\{\mu_i\}_{i \in I}}^{\weakstar}$.  The map $\bigoplus_{i \in I} \mu_i$ induces a continuous map $\phi : \cP(\beta I) \to \cP(Y)$ satisfying $\phi(\delta_i) = \mu_i$ for $i \in I$. Hence $\overline{\phi(\beta I)}^{\weakstar} \supseteq Y$.  Since $\phi(\cP(\beta I))$ is compact and weak*-closed, the Krein-Milman theorem implies that $\phi(\cP(\beta I)) = \cP(Y)$, that is $\phi$ is surjective.  It follows that the map $\bigoplus_{i \in I} \mu_i$ is isometric.
\end{proof}

Recall the bijective correspondence between commutative unital $\cG$-\Cstar-algebras and compact $\cG$-spaces from Section \ref{sec:groupoid-spaces}.  The next result implies that this correspondence restricts to a bijective correspondence between essential commutative unital $\cG$-\Cstar-algebra extensions of $\cont(\Gnaught)$ and $\cG$-boundaries.
\begin{proposition}
  \label{prop:essential_iff_boundary}  
  Let $\cG$ be an {\'e}tale groupoid with compact Hausdorff unit space.  Let $\cont(Y)$ be a commutative unital $\cG$-\Cstar-algebra.  Then $\cont(Y)$ is an essential extension of $\cont(\Gnaught)$ if and only if $Y$ is a $\cG$-boundary.
\end{proposition}
\begin{proof}
  First recall from Remark~\ref{rem:G-boundary-irreducible} that $Y$ is a $\cG$-boundary if and only if for any family of probability measures $(\mu_x)_{x \in \Gnaught}$ on $Y$ with the property that $\mu_x|_{\cont(\Gnaught)} = \delta_x$ for each $x \in \Gnaught$, we have
  \begin{gather*}
    Y
    \subseteq
    \overline{\{g \mu_{\rms(g)} \mid g \in \cG\}}^{\weakstar}
    \eqstop
  \end{gather*}
  Suppose that $Y$ is a $\cG$-boundary.  Let $B$ be a $\cG$-\Cstar-algebra and let $\phi : \cont(Y) \to B$ be a $\cG$-ucp map.  We must show that $\phi$ is an embedding.  Equivalently, since $\cont(Y)$ is commutative, we must show that $\phi$ is isometric.  By Proposition~\ref{prop:ell_infty_injective}, we know that $\linfty(\cG)$ is an injective $\cG$-\Cstar-algebra.  So there is a $\cG$-ucp map $B \to \linfty(\cG)$.  It suffices to show that the composition $\psi \circ \phi$ is isometric.  Hence without loss of generality, we can assume that $B = \linfty(\cG)$.  For $x \in \Gnaught$, let $\mu_x = \delta_x \circ \phi$.  Then by Proposition~\ref{prop:ell_infty_injective}, we have $\phi = \bigoplus_{g \in \cG} g \mu_{\rms(g)}$. It now follows from Lemma~\ref{lem:krein_milman_lemma} and the characterisation of $\cG$-boundaries from the beginning of the proof that $\phi$ is isometric.
  
  Conversely, suppose that $\cont(Y)$ is an essential extension of $\cont(\Gnaught)$. Let $(\mu_x)_{x \in \Gnaught}$ be a family of probability measures on $Y$ satisfying $p_*\mu_x = \delta_x$ for all $x \in \Gnaught$.  By Proposition~\ref{prop:ell_infty_injective}, we obtain a $\cG$-ucp map $\phi = \bigoplus_{g \in \cG} g \mu_{s(g)}$.  By essentiality, $\phi$ is isometric.  Hence by Lemma~\ref{lem:krein_milman_lemma}, we have $Y \subseteq \overline{\{g\mu_{\rms(g)} \mid g \in \cG \}}^{\weakstar}$, so by the characterisation of $\cG$-boundaries from Remark~\ref{rem:G-boundary-irreducible}, $Y$ is a $\cG$-boundary.
\end{proof}

\begin{theorem}
  \label{thm:hamana-equals-furstenberg}
  Let $\cG$ be an {\'e}tale groupoid with compact Hausdorff unit space. The Hamana boundary $\hb \cG$ and the Furstenberg boundary $\fb \cG$ are isomorphic as $\cG$-spaces.
\end{theorem}
\begin{proof}
  We will prove that $\cont(\hb \cG)$ and $\cont(\fb \cG)$ are isomorphic as $\cG$-\Cstar-algebras. The result will then follow from the discussion in Section \ref{sec:groupoid-spaces}.
  
  Theorem~\ref{thm:injective_envelope} and Proposition~\ref{prop:essential_iff_boundary} imply that $\hb \cG$ is a $\cG$-boundary. Hence by the universal property of $\fb \cG$ there is a (necessarily surjective) $\cG$-map $\pi : \fb \cG \to \hb \cG$.  This map corresponds to a $\cG$-ucp embedding $\phi : \cont(\hb \cG) \to \cont(\fb \cG)$.  By the injectivity of $\cont(\hb \cG)$, there is a $\cG$-ucp map $\psi : \cont(\fb \cG) \to \cont(\hb \cG)$ such that the following diagram commutes:
  \begin{center}
    \begin{tikzcd}
      \cont(\fb \cG) \arrow[rd, "\psi", dashed]                 &                    \\
      \cont(\hb \cG) \arrow[u, "\phi", hook] \arrow[r, "\id"] & \cont(\hb \cG)
    \end{tikzcd}
  \end{center}
  Applying Proposition~\ref{prop:essential_iff_boundary} again, $\cont(\fb \cG)$ is an essential extension of $\cont(\Gnaught)$, so $\psi$ must be an embedding, forcing all of the maps in the above diagram to be *-isomorphisms. In particular, $\phi$ is a *-isomorphism.
\end{proof}

Finally, in the Hausdorff setting, we deduce the equality of the Furstenberg boundary as constructed in this section with the Furstenberg boundary constructed by Borys, which we denote by $\bb \cG$.  Recall that $\cont(\bb \cG)$ is the injective envelope of $\cont(\Gnaught)$ in the category of concrete $\cG$-operator systems in the terminology of \cite{borys2019-boundary,borys2020-thesis}.  In particular, it is the injective envelope of $\cont(\Gnaught)$ in the category of $\cG$-\Cstar-bundles. 
\begin{theorem}
  Suppose that $\cG$ is an {\'e}tale Hausdorff groupoid with compact space of units.  Then the Furstenberg boundary $\fb \cG$ and the boundary $\bb \cG$ constructed by Borys are isomorphic as $\cG$-spaces.
\end{theorem}
\begin{proof}
  We will prove that $\cont(\fb \cG)$ and $\cont(\bb \cG)$ are isomorphic as $\cG$-\Cstar-algebras.  The result will then follow from the discussion in Section~\ref{sec:groupoid-spaces}.

  Since $\cont(\fb \cG)$ is a commutative unital $\cG$-\Cstar-algebra, Proposition~\ref{prop:bundle-G-cstar-algebra-integration-desintegration} says that it has a unique compatible structure of a unital $\cG$-\Cstar-bundle.  By Proposition~\ref{prop:bundle-G-cstar-algebra-integration-desintegration} and injectivity of $\cont(\bb \cG)$ in this category, there is a $\cG$-ucp map $\phi: \cont(\fb \cG) \to \cont(\bb \cG)$.  Since $\cont(\bb \cG)$ is a unital $\cG$-\Cstar-algebra, there is also a $\cG$-ucp map $\psi: \cont(\bb \cG) \to \cont(\fb \cG)$.

  By the rigidity of $\cont(\fb \cG)$, we infer that $\psi \circ \phi$ is the identity map.  Similarly, by the rigidity of $\cont(\bb \cG)$, we have that $\phi \circ \psi$ is the identity map.  Hence $\phi$ and $\psi$ are both *-isomorphisms.
\end{proof}

\section{Essential groupoid C$^*$-algebras}
\label{sec:non-hausdorff-groupoids}

In this section, we develop some understanding of essential groupoid \Cstar-algebras, which will be necessary to adapt methods from the theory of \Cstar-simplicity to non-Hausdorff groupoids.  While there is always an inclusion $\Cstarred(\cG) \subseteq \Cstarred(\cG \ltimes \fb \cG)$, it is {\`a} priori not clear that there is a similar inclusion on the level of essential groupoid \Cstar-algebras.  In Section~\ref{sec:inclusion-essential-groupoids} we will show that this is the case for minimal groupoids and $\sigma$-compact groupoids.  For the proof of this fact we reformulate the local conditional expectation of Kwa{\'s}niewski-Meyer in terms of continuous extensions of functions on extremally disconnected spaces, such as the Furstenberg boundary $\fb \cG$.  This will be done in Section~\ref{sec:local-conditional-expectation}.  Finally, the point of view developed here naturally leads to the question whether the local conditional expectation of the Furstenberg groupoid is actually related to a reduced groupoid \Cstar-algebra, which appears in the background. This is indeed the case, as we will show in Section~\ref{sec:hausdorffification}, when introducing the Hausdorffification of a groupoid with extremally disconnected unit space.  This concept will also be useful in our further discussion on the intersection property in Sections~\ref{sec:characterisations} and \ref{sec:confined-subgroupoids}.

\subsection{The local conditional expectation for groupoid C$^*$-algebras via continuous extensions on extremally disconnected spaces}
\label{sec:local-conditional-expectation}

The next lemma is an adaption of \cite[Lemma 3.2]{breuillardkalantarkennedyozawa14} (see also \cite[Lemma 3.3]{ursu2021-relative-simplicity}) to the setting of groupoid dynamical systems that are not necessarily minimal.
\begin{lemma}
  \label{lem:G-space-projections-dense-subsets}
  Let $\cG$ be an {\'e}tale groupoid with compact Hausdorff space of units and assume that $\cG$ is minimal or $\sigma$-compact.  Let $p: X \to \Gnaught$ be a totally disconnected, irreducible, compact $\cG$-space.  Then
  \begin{itemize}
  \item For every open subset $U \subseteq X$ its image $p(U) \subseteq \Gnaught$ has non-empty interior.
  \item For every dense subset $D \subseteq \Gnaught$, the inverse image $p^{-1}(D) \subseteq X$ is dense.
  \end{itemize}
\end{lemma}
\begin{proof}
  Denote by $\pi: \cG \ltimes X \to \cG$ the natural extension of $p$.  Let $U \subseteq X$ be an open subset.  Denote by $\vphi_\gamma$ the partial homeomorphism of $X$ associated to an open bisection $\gamma$ of $\cG \ltimes X$ and $\psi_\gamma$ the partial homeomorphism of $\Gnaught$ associated with an open bisection of $\cG$.

  If $\cG$ is minimal, then $\cG \grpaction{} X$ is minimal too.  By compactness of $X$ we find finitely many compact open bisections $(\gamma_n)_n$ of $\cG \ltimes X$ such that $(\im \vphi_{\gamma_n})_n$ covers $X$ and for every $n$, we have $\supp \vphi_{\gamma_n} \subseteq U$ and there is an open bisection $\beta_n$ of $\cG$ such that $\gamma_n \subseteq \pi^{-1}(\beta_n)$.  Since each of the sets $p(\im \vphi_{\gamma_n})$ is closed and $X$ is a Baire space, it follows that there is some $n$ such that $p(\im \vphi_{\gamma_n})$ has non-empty interior.  Then also $\psi_{\beta_n^*} \circ p(\im \vphi_{\gamma_n}) = p(\supp \vphi_{\gamma_n})$ has non-empty interior.

  If $\cG$ is $\sigma$-compact, observe that $\pi: \cG \ltimes X \to \cG$ is proper, so that $\cG \ltimes X$ is $\sigma$-compact and as such covered by countably many compact open bisections $(\gamma_n)_{n \in \NN}$, each contained in $\pi^{-1}(\beta_n)$ for a some open bisection $\beta_n$ of $\cG$.  Consider the open subset $O = \bigcup_{n \in \NN} \vphi_{\gamma_n}(U) \subseteq X$.  Then $X \setminus O \subseteq X$ is a proper closed $\cG$-invariant subset.  Thus $p(X \setminus O) \subseteq \Gnaught$ is a proper subset by irreducibility of $X$.  Further, since $X$ is compact and $\Gnaught$ is Hausdorff, $p(X \setminus O)$ is also closed.  It follows that $p(O)$ has non-empty interior.  Since it is the countable union of the closed subsets $(p(\vphi_{\gamma_n}(U)))_{n \in \NN}$ and $X$ is a Baire space, there is $n \in \NN$ such that $p(\vphi_{\gamma_n}(U))$ has non-empty interior.  As before, one concludes that $\psi_{\beta_n^*} \circ p (\vphi_{\gamma_n}(U)) = p(U \cap \supp \vphi_{\gamma_n})$ has non-empty interior.

  Let now $D \subseteq \Gnaught$ be a dense subset.  Given an open subset $U \subseteq X$, we know by the first part that $p(U)$ has non-empty interior and thus intersects $D$ non-trivially.  Thus $U \cap p^{-1}(D)$ is non-empty either, which proves density of $p^{-1}(D) \subseteq X$.
\end{proof}

\begin{remark}
  \label{rec:extension-theorem}
  A topological space $X$ is extremally disconnected if and only if for every open subset $U \subseteq X$ and every continuous function $f \in \contb(U)$ there is a continuous function $g \in \contb(X)$ satisfying $g|_U = f$.  We refer to \cite[Exercise 1.H.6, p.23]{gillmanjerison1960}, which uses a suitable version of Urysohn's extension theorem (see \cite[1.17, p.18]{gillmanjerison1960}).  One may choose $g$ to be supported in the clopen subset $\ol{U} \subseteq X$.
\end{remark}

Let us now describe our perspective on the local conditional expectation.  We adopt the original perspective of the local multiplier algebra as explained in Section~\ref{sec:groupoids}.  Recall the terminology of \cite[Section 3.1]{kwasniewskimeyer2019-essential} that for an inclusion $A \subseteq B$ of \Cstar-algebras, a generalised conditional expectation is given by another inclusion $A \subseteq C$ and a completely positive, contractive map $\rF: B \to C$ that restricts to the identity on $A$.  In the next proposition, the role of $A$ is played by $\cont(\Gnaught)$.  
\begin{proposition}
  \label{prop:conditional-expectation-extremally-disconnected}
  Let $\cG$ be an {\'e}tale groupoid with compact Hausdorff space of units and assume that $\cG$ is minimal or $\sigma$-compact.  Let $p: X \to \Gnaught$ be an extremally disconnected, irreducible, compact $\cG$-space and denote by $\pi: \cG \ltimes X \to \cG$ its natural extension.  Then there is a unique generalised conditional expectation $\rF: \Cstarred(\cG) \to \cont(X)$ such that $\supp(\rF(f)) \subseteq \ol{p^{-1}(U \cap \Gnaught)}$ and
  \begin{gather*}
    \rF(f)|_{p^{-1}(U \cap \Gnaught)} = (f \circ \pi)|_{p^{-1}(U \cap \Gnaught)}
  \end{gather*}
  for all open bisections $U \subseteq \cG$ and all $f \in \contc(U)$.  If $\Ered: \Cstarred(\cG) \to \Mloc(\cont(\Gnaught))$ denotes the local conditional expectation, then there is a unique $\cG$-equivariant embedding ${\Mloc(\cont(\Gnaught)) \to \cont(X)}$ such that the following diagram commutes
  \begin{equation*}
    \begin{tikzcd}
      \Cstarred(\cG) \arrow{rd}{\Ered} \arrow{r}{\rF} & \cont(X) \\
      & \Mloc(\cont(\Gnaught)) \arrow{u}
    \end{tikzcd}
  \end{equation*}
\end{proposition}
\begin{proof}
  Let $U \subseteq \cG$ be an open bisection and $f \in \contc(U)$.  Write $V = p^{-1}(U \cap \Gnaught) \subseteq X$ and denote by $g \in \cont(X)$ the unique continuous extension of $(f \circ \pi)|_V$ having support in $\ol{V}$.  This defines a map $\rF: \bigoplus_U \contc(U) \to \cont(X)$.  Let $f = \sum_{i = 1}^n f_i \in \cC(\cG)$ for open bisections $U_i \subseteq \cG$ and $f_i \in \contc(U_i)$.  Write $V_i = p^{-1}(U_i \cap \Gnaught)$ and observe that $| \sum_i \rF(f_i) (x)| = | \sum_i f_i \circ \pi (x)| \leq \|\sum_i f_i\|_{\Cstarred}$ for all $x \in X \setminus \bigcup_i \partial V_i$, where $f_i$ is considered as a Borel function defined on all of $\cG$ and having support in $U_i$. So continuity of each $\rF(f_i)$ implies that
  \begin{align*}
    \| \sum_i F(f_i) \|
    & =
      \sup \{ |\sum_i F(f_i)(x)| \mid x \in X\} \\
    & =
      \sup \{ |\sum_i F(f_i)(x)| \mid x \in X \setminus \bigcup_i \partial V_i\} \\
    & \leq
      \|f\|_{\Cstarred}
  \end{align*}
  So $\rF$ factors through $\cC(\cG)$ and extends to a well-defined map $\Cstarred(\cG) \to \cont(X)$.

  Given a dense open subset $O \subseteq \Gnaught$ and $f \in \contb(O)$,  Lemma \ref{lem:G-space-projections-dense-subsets} says that $p^{-1}(O) \subseteq X$ is dense, so that there is a unique continuous extension of $f \circ \pi$ to $X$, which defines an element in $\cont(X)$.  We thus obtain a well-defined map $\contb(O) \to \cont(X)$.  Varying $O$, these embeddings define an injective *-homomorphism $\iota: \Mloc(\cont(\Gnaught)) \to \cont(X)$.

  Given an open bisection $U \subseteq \cG$ and $f \in \contc(U)$, fix $O = \Gnaught \setminus \partial U$.  Then $\Ered(f) \in \contb(O) \subseteq \Mloc(\cont(\Gnaught))$ by definition.  Further, $\iota(\Ered(f))(x) = \rF(f)(x)$ for all $x \in X \setminus p^{-1}(\Gnaught \cap \partial U)$.  By Lemma~\ref{lem:G-space-projections-dense-subsets} the latter set is dense in $X$ and we can infer $\iota(\Ered(f)) = \rF(f)$ from continuity.
\end{proof}

\subsection{Inclusions of essential groupoid C$^*$-algebras}
\label{sec:inclusion-essential-groupoids}

Using the skyscraper groupoid as in \cite[Example 2.5]{khoshkamskandalis2002}, it was observed in \cite[Remark 4.8]{kwasniewskimeyer2019-essential} that the essential groupoid construction is not functorial. The present section demonstrates that in the situations arising from the study of \Cstar-simplicity, we do see a natural inclusion of essential groupoid \Cstar-algebras arising from suitable surjections of groupoids.  We make use of the Dixmier algebra picture of the local conditional expectation, as explained in Section~\ref{sec:groupoids}.
\begin{theorem}
  \label{thm:inclusion-essential-algebras}
  Let $\cG$ and $\cH$ be {\'e}tale groupoids with compact Hausdorff space of units and that $\cH$ is minimal or $\sigma$-compact.  Assume that $\pi: \cG \to \cH$ is a fibrewise bijective, proper surjection of groupoids.  Then precomposition with $\pi$ defines a *-homomorphism $\cC(\cH) \to \cC(\cG)$ giving rise to the following commuting diagram.
  \begin{equation*}
    \begin{tikzcd}
      \cC(\cH) \arrow[hookrightarrow]{d}  \arrow[hookrightarrow]{r} & \Cstarred(\cH) \arrow[hookrightarrow]{d} \arrow{r} & \Cstaress(\cH) \arrow[hookrightarrow]{d} \\
      \cC(\cG) \arrow[hookrightarrow]{r} & \Cstarred(\cG) \arrow{r} & \Cstaress(\cG)
    \end{tikzcd}
  \end{equation*}
\end{theorem}
\begin{proof}
  Whenever $U \subseteq \cH$ is an open bisection, then $\pi^{-1}(U) \subseteq \cG$ is an open bisection, since $\pi$ is fibrewise bijective.  Since $\pi$ is proper, precomposition with $\pi$ defines a *-isomorphism $\contc(U) \cong \contc(\pi^{-1}(U))$.  This proves that $\pi_*: \cC(\cH) \to \cC(\cG)$ is well-defined.  It is injective, since $\pi$ is surjective.

  For $x \in \Gnaught$, we have $\lambda^\cG_x \circ \pi_* = \lambda^\cH_{\pi(x)}$ as a straightforward calculation on the subspaces $\contc(U) \subseteq \cC(\cH)$ for open bisections $U \subseteq \cH$ shows.  It follows that $\pi_*$ extends continuously to an embedding $\Cstarred(\cH) \hra \Cstarred(\cG)$.

  Write now $X = \mathrm{spec}(\Dix(\Gnaught))$ and observe that $X$ is an extremally disconnected $\cG$-space.  Denote its unit space projection by $p: X \to \Gnaught$.  Since $\pi$ is fibrewise bijective, $X$ also becomes an $\cH$-space.  For an element $h \in \cH$ and an element $x \in X$ satisfying $\pi \circ p(x) = \rms(h)$, there is a unique $g \in \pi^{-1}(h)$ satisfying $\rms(g) = p(x)$.  Then $hx = gx$ holds by definition.  Denote by $\Ered^\cG: \Cstarred(\cG) \to \cont(X)$ the local conditional expectation.  Then
  \begin{gather*}
    \Ered^{\cG} \circ \pi_*: \Cstarred(\cH) \to \cont(X)
  \end{gather*}
  falls in the scope of Proposition~\ref{prop:conditional-expectation-extremally-disconnected}.  Hence we find that $\Ered^{\cG} \circ \pi_* = \iota \circ \Ered^\cH$, where we denote by $\iota: \Dix(\Hnaught) \to \cont(X)$ the natural inclusion and by $\Ered^\cH: \Cstarred(\cH) \to \Dix(\Hnaught)$ the local conditional expectation of $\cH$. It follows that  
  \begin{align*}
    \ker(\Cstarred(\cH) \to \Cstaress(\cH))
    & =
      \{a \in \Cstarred(\cH) \mid  \Ered^\cH(a^*a) = 0\} \\
    & =
      \{a \in \Cstarred(\cH) \mid \Ered^\cG \circ \pi_* (a^*a) = 0\} \\
    & =
      \ker(\Cstarred(\cH) \to \Cstaress(\cG))
      \eqstop
  \end{align*}
\end{proof}

\begin{remark}
  \label{rem:inclusion-essential-algebras-alternative-proof}
  Our proof of Theorem~\ref{thm:inclusion-essential-algebras} makes direct use of the definition of the essential groupoid.  It should be remarked that Lemma~\ref{lem:G-space-projections-dense-subsets} applied to the map $\cG \ltimes X \to \cG$ also implies that preimages of meager sets remain meager.  So functions with meager support are pulled back to functions with meager support.   Recall from Section~\ref{sec:groupoids} that for $a \in \Cstarred(\cG)$, one defines a function on $\cG$ by $\hat a (g) = \langle \lambda_{\rms(g)}(a) \delta_{\rms(g)}, \delta_g \rangle$.  It follows from \cite[Proposition 7.18]{kwasniewskimeyer2019-essential} and its proof that for {\'e}tale groupoids $\cG$, the kernel of $\Cstarred(\cG) \to \Cstaress(\cG)$ consists exactly of elements $a \in \Cstarred(\cG)$ such that $\widehat{a^*a}$ has meager support.  This provides an alternative approach to Theorem~\ref{thm:inclusion-essential-algebras} for $\sigma$-compact groupoids and minimal groupoids.
\end{remark}

\begin{remark}
  \label{rem:possible-generalisation-inclusion-essential-algebras}
  We currently do not know whether the assumption of minimality or $\sigma$-compactness is necessary in Theorem~\ref{thm:inclusion-essential-algebras}.
\end{remark}

\subsection{Hausdorffification of groupoids with extremally disconnected space of units}
\label{sec:hausdorffification}

Recall from Proposition~\ref{prop:isotropy-clopen} that the isotropy groupoid of an {\'e}tale groupoid with an extremally disconnected space of units is clopen.  The analogue statement holds true for the closure of its space of units, which will be the starting point of defining a Hausdorffification.
\begin{lemma}
  \label{lem:extended-unit-space-clopen}
  Let $\cG$ be an {\'e}tale groupoid whose space of units is an extremally disconnected, locally compact Hausdorff space.  Then $\cG$ is extremally disconnected.  In particular, $\ol{\Gnaught} \subseteq \cG$ is a clopen normal subgroupoid.
\end{lemma}
\begin{proof}
  Let $O \subseteq \cG$ be an open subset and $g \in \ol{O}$.  Let $U \subseteq \cG$ be an open bisection containing $g$.  Then $U \cap O$ contains a net converging to $g$, so that $g$ lies in the relative closure $\ol{U \cap O}^U \subseteq U$.  Since $U$ is an open bisection and as such homeomorphic to an open subset of $\Gnaught$, it is extremally disconnected.  So $\ol{U \cap O}^U \subseteq \ol{O}$ is an open neighbourhood of $g$.
\end{proof}

Recall that quotients of {\'e}tale groupoids by open normal subgroupoids remain {\'e}tale.  In the setting of Lemma~\ref{lem:extended-unit-space-clopen}, the extended unit space is even clopen, so that we obtain a Hausdorff quotient.
\begin{definition}
  \label{def:hausdorffification}
  Let $\cG$ be an {\'e}tale groupoid whose space of units is an extremally disconnected, locally compact Hausdorff space.  The \emph{extended unit space} of $\cG$ is $\ol{\Gnaught} \subseteq \cG$.  The \emph{Hausdorffification} of $\cG$ is $\cG_\Haus = \cG / \ol{\Gnaught}$.
\end{definition}

Before we proceed to the next theorem, let us remark that if $\cG$ is an {\'e}tale groupoid with extremally disconnected compact Hausdorff space of units, then Proposition~\ref{prop:conditional-expectation-extremally-disconnected}, applied to the trivial $\cG$-space $\Gnaught$, produces a conditional expectation $\Ered: \Cstarred(\cG) \to \cont(\Gnaught)$ which induces a faithful conditional expectation $\Eess: \Cstaress(\cG) \to \cont(\Gnaught)$.
\begin{theorem}
  \label{thm:essential-is-reduced-cstar-algebra}
  Let $\cG$ be an {\'e}tale groupoid whose space of units is an extremally disconnected, compact Hausdorff space with the property that there exists a dense set of points in $\Gnaught$ where the isotropy group of the extended unit space $\ol{\Gnaught}$ is amenable.  Denote by $p: \cG \to \cG_\Haus$ the quotient map to its Hausdorffification.
  \begin{itemize}
  \item There is a unique *-isomorphism $\pi: \Cstaress(\cG) \to \Cstarred(\cG_\Haus)$ which restricts to the natural isometry $p_*: \contc(U) \to \contc(p(U))$ for every open bisection $U \subseteq \cG$.
  \item Denote by $\rE: \Cstarred(\cG_\Haus) \to \cont(\Gnaught)$ the natural conditional expectation.  Then $\rE \circ \pi = \Eess$.
  \end{itemize}
\end{theorem}
\begin{proof}
  By assumption we have a dense set of points where the isotropy group of $\ol{\Gnaught}$ is amenable. Hence Proposition~\ref{prop:quotient-reduced-groupoid-cstar-algebras} implies that there is a *-homomorphism $\pi_{\mathrm{red}}: \Cstarred(\cG) \to \Cstarred(\cG_\Haus)$ restricting to the natural isometry $p_*: \contc(U) \to \contc(p(U))$ for every open bisection $U \subseteq \cG$.  We will show that $\pi_{\mathrm{red}}$ factors through the essential groupoid \Cstar-algebra of $\cG$ to a *-homomorphism $\pi: \Cstaress(\cG) \to \Cstarred(\cG_\Haus)$ satisfying $\rE \circ \pi = \Eess$.  Since $\rE$ is faithful, it suffices to show that $\rE \circ \pi_{\mathrm{red}}$ is the local conditional expectation of $\Cstarred(\cG)$.  To this end, we verify the conditions of Proposition~\ref{prop:conditional-expectation-extremally-disconnected} for $X = \Gnaught$.  For every open bisection $U \subseteq \cG$ and every $f \in \contc(U)$ the function satisfying
  \begin{gather*}
    \rE \circ \pi_{\mathrm{red}}(f)(x)
    =
    \sum_{g \in \ol{\Gnaught}_x} f(g)
  \end{gather*}
  is continuous, supported in $\ol{U \cap \Gnaught}$ and satisfies $\rE \circ \pi_{\mathrm{red}}(f)|_{U \cap \Gnaught} = f|_{U \cap \Gnaught}$.  So $\rE \circ \pi_{\mathrm{red}}(f) = \Eess(f)$ follows.
\end{proof}

\section{Fundamental characterisations of the ideal intersection property}
\label{sec:characterisations}

In this section we establish characterisations of the ideal intersection property that will be of fundamental importance for further reasoning both about the existence of essentially confined amenable sections of the isotropy group in Section~\ref{sec:confined-subgroupoids}, and about Powers averaging property in Section~\ref{sec:powers-averaging}.

The characterisations we will consider here come in three different flavours.  First, the ideal intersection property of $\cont(\Gnaught) \subseteq \Cstaress(\cG)$ is equivalent to the ideal intersection property of $\cont(\fb \cG) \subseteq \Cstaress(\cG \ltimes \fb \cG)$, making it possible to exploit the extremal disconnectedness of the Furstenberg boundary.  Second, on the level of $\cG \ltimes \fb \cG$, in the presence of a dense set of points in $\fb \cG$ with amenable stabilizers, it is possible to express the ideal intersection property as a simple statement about principality, which must, however, take into account the possible non-Hausdorffness of the groupoid.  This is expressed in the definition of essentially principal groupoids in Definition~\ref{def:essentially-effective}.  Third, again in the presence of a dense set of points in $\fb \cG$ with amenable stabilizers, the ideal intersection property can be expressed in terms of a uniqueness statement for $\cG$-pseudo expectations.  It is this last characterisation that will be most frequently applied in further results.  Recall that, following Pitts \cite{pitts2017}, a $\cG$-pseudo expectation on $\Cstaress(\cG)$ is a $\cG$-equivariant ucp map $\Cstaress(\cG) \to \cont(\fb \cG)$ which restricts on $\cont(\Gnaught)$ the the canonical inclusion $\cont(\Gnaught) \hra \cont(\fb \cG)$.

\begin{theorem}
  \label{thm:intersection-property-essential-algebras}

\begin{enumerate}
  \item[(I)] Let $\cG$ be an {\'e}tale groupoid with compact Hausdorff space of units.  Assume that $\cG$ is Hausdorff, that $\cG$ is minimal or that $\cG$ is $\sigma$-compact.  Then the following statements are equivalent.
  \begin{enumerate}
  \item[(i)] $\cont(\fb \cG) \subseteq \Cstaress(\cG \ltimes \fb \cG)$ has the ideal intersection property.
  \item[(ii)] Every $\cG$-pseudo expectation $\Cstaress(\cG) \to \cont(\fb \cG)$ is faithful.
  \item[(iii)] $\cont(\Gnaught) \subseteq \Cstaress(\cG)$ has the ideal intersection property.
  \end{enumerate}

  \item[(II)] Assume, in addition to the assumptions in (I), that $\mathfrak{Y} \mathrel{:=} \{ y \in \fb \cG \mid (\cG \ltimes \fb \cG)_y^y \text{ is amenable} \}$ is dense in $\fb \cG$. Then (i) to (iii) are equivalent to
  \begin{enumerate}
  \item[(iv)] $\Iso(\cG \ltimes \fb \cG) = \ol{\fb \cG}^{\cG \ltimes \fb \cG}$.
  \item[(v)] There is a unique $\cG$-pseudo expectation $\Cstaress(\cG) \to \cont(\fb \cG)$.
  \end{enumerate}

  \item[(III)] Let $\mathfrak{X}$ be the set of those points $x \in \Gnaught$ whose $\cG$-orbit $\cG.x$ is contained in the interior of its closure in $\Gnaught$. Suppose that $\mathfrak{X}$ is dense in $\Gnaught$. Let $\pi : \fb \cG \to \Gnaught$ be the canonical projection. Then $\pi^{-1}(\mathfrak{X}) \subseteq \mathfrak{Y}$ and $\mathfrak{Y}$ is dense in $\fb \cG$. In that case, (i) to (v) are equivalent.  
\end{enumerate}
\end{theorem}
Note that the condition in (III) is satisfied if $\cG$ is topologically transitive, in the sense that there exists a point in $\Gnaught$ with dense $\cG$-orbit.

\begin{remark}
  \label{rem:countability-assumption}
  It is worth commenting on the necessity of the assumption of minimality or $\sigma$-compactness for non-Hausdorff groupoids in Theorem~\ref{thm:intersection-property-essential-algebras}~(I).  Most \Cstar-simplicity results for groups hold without any assumption on countability of the group, and this is also reflected in the Hausdorff case for {\'e}tale groupoids.  However, the essential groupoid \Cstar-algebra, which replaces the reduced groupoid \Cstar-algebra in the non-Hausdorff case, is by its very definition governed by the interaction between open dense subsets of the groupoid.  As a consequence, it is not even clear that there is an inclusion of $\Cstaress(\cG)$ into $\Cstaress(\cG \ltimes \fb \cG)$.  Only under the additional assumption of Theorem~\ref{thm:intersection-property-essential-algebras} have we been able to prove the existence of such an inclusion in Section~\ref{sec:inclusion-essential-groupoids}.  This inclusion is therefore vital for arguments pertaining to the theory of \Cstar-simplicity.
\end{remark}

\begin{remark}
  \label{rem:NotAllStabAmen}
  Note that, in general, contrary to what we had in an earlier version of our paper and what is also claimed in \cite[Proposition~4.2.16]{borys2020-thesis} and \cite[Proposition~5.1]{borys2019-boundary}, the stabilizer groups $(\cG \ltimes \fb \cG)_y^y$, for $y \in \fb \cG$, need not be amenable. Indeed, consider the following concrete example: Let $\mathbb{F}$ be a non-abelian free group and $N_i$ a sequence of finite index normal subgroups of $\mathbb{F}$ such that $\bigcap_i N_i$ is the trivial subgroup. Consider the bundle of groups $\cG \mathrel{:=} (\coprod_{i \in \mathbb{N}} \mathbb{F}/N_i) \amalg \mathbb{F}$, where each $G_i \mathrel{:=} \mathbb{F} / N_i$ is equipped with the discrete topology and a basic open set around $g \in \mathbb{F}$ is given by $\{gN_i \mid i \in I\}$, for some co-finite subset $I$ of $\mathbb{N}$. Then $\cG$ becomes a Hausdorff groupoid with unit space given by the one-point compactification of $\mathbb{N}$, where we identify the point $\{ \infty \}$ at infinity with the identity of $\mathbb{F}$. Let $\pi : \fb \cG \to \Gnaught$ be the canonical projection. For each $i \in \mathbb{N}$, $\pi^{-1}(i)$ consists of just a single point with isotropy group $G_i$ in $\cG \ltimes \fb \cG$. It follows that $\mathbb{F}$ acts trivially on $\pi^{-1}(\infty)$. Hence, for each $y \in \pi^{-1}(\infty)$, $(\cG \ltimes \fb \cG)_y^y$ is not amenable.
\end{remark}

\begin{remark}
  \label{rem:AllStabAmenSit}
  Related to Remark~\ref{rem:NotAllStabAmen}, there are special situations where all stabilizer groups in $\cG \ltimes \fb \cG$ are amenable. For example, this is trivially true if all stabilizer groups of $\cG$ are amenable. This is also true for transformation groupoids with compact unit spaces, as shown in \cite{kawabe17}. Moreover, we will show below (see Corollary~\ref{cor:MinAllStabAmen}) that this is also the case for minimal groupoids.
\end{remark}

We will adopt the following terminology in the remainder of this article.
\begin{definition}
  \label{def:essentially-effective}
  Let $\cG$ be an {\'e}tale groupoid with locally compact Hausdorff space of units.  We will say that $\cG$ is \emph{essentially principal} if $\Iso(\cG) = \ol{\Gnaught}$.  It is \emph{essentially effective} if $\Iso(\cG)^\circ \subseteq \ol{\Gnaught}$.
\end{definition}
We remark that, by Proposition~\ref{prop:isotropy-clopen}, an essentially effective groupoid with extremally disconnected unit space is automatically essentially principal.

For the proof of Theorem~\ref{thm:intersection-property-essential-algebras}, it is convenient to use the notation $\cH \mathrel{:=} \cG \ltimes \fb \cG$. Moreover, the conditions in Theorem~\ref{thm:intersection-property-essential-algebras}~(I) imply that there is a canonical inclusion $\Cstaress(\cG) \subseteq \Cstaress(\cH)$. Indeed, if $\cG$ is Hausdorff, then $\Cstaress(\cG) = \Cstarred(\cG) \subseteq \Cstarred(\cH) = \Cstaress(\cH)$, and if $\cG$ is minimal or $\sigma$-compact, then the desired inclusion is provided by Theorem~\ref{thm:inclusion-essential-algebras}.

Let us begin with the proof of (I).
\begin{proof}[Proof of Theorem~\ref{thm:intersection-property-essential-algebras}~(I)]
(i) $\Rightarrow$ (ii): Let $\vphi : \Cstaress(\cG) \to \cont(\fb G)$ be a $\cG$-pseudo expectation. Since $\Cstaress(\cG) \subseteq \Cstaress(\cH)$, applying $\cG$-injectivity of $\cont(\fb \cG)$ we can extend $\vphi$ to a $\cG$-equivariant ucp map $\psi : \Cstaress(\cH) \to \cont(\fb \cG)$. By $\cG$-rigidity we find that $\psi$ restricted to $\cont(\Hnaught) = \cont(\fb \cG)$ is the identity. Now consider the ideal $I \mathrel{:=} \{ a \in \Cstaress(\cH) \mid \psi(bac) = 0 \text{ for all } b, c \in \Cstaress(\cH) \}$. Clearly $\cont(\fb \cG) \cap I = \{0\}$, so that $I = \{0\}$ because $\cont(\fb \cG) \subseteq \Cstaress(\cH)$ has the ideal intersection property. We claim that if $a \in \Cstaress(\cG)$ satisfies $\vphi(aa^*) = 0$, then $a \in I$. Indeed, $\psi(bac) \psi(bac)^* \leq \psi(bac c^* a^* b^*) \leq \Vert c c^* \Vert \psi(b a a^* b^*)$, so that it suffices to show that $\psi(b a a^* b^*) = 0$ for all $b \in \Cstaress(\cH)$. Since $\Cstaress(\cH) = \cspan \cont(\fb \cG) \Cstaress(\cG)$ and $\Cstaress(\cG)$ is generated by normalizers of $\cont(\Gnaught)$, it actually suffices to show that $\psi(b a a^* b^*) = 0$ for normalizers $b \in \Cstaress(\cG)$ of $\cont(\Gnaught)$. But $\cG$-equivariance of $\psi$ implies that $\psi(b a a^* b^*) = b \psi(a a^*) b^*$. Hence we conclude that $\psi(b a a^* b^*) = b \psi(a a^*) b^* = b \vphi(a a^*) b^* = 0$, as desired. Since $I = \{0\}$, it follows that $a = 0$. In other words, we have shown that $\vphi$ is faithful.

(ii) $\Rightarrow$ (iii): Let $\rho: \Cstaress(\cG) \thra A$ be a surjective *-homomorphism that is faithful on $\cont(\Gnaught)$.  Proposition~\ref{prop:groupoid-cstar-algebras-carries-action} shows that $A$ carries a unital $\cG$-\Cstar-algebra structure such that $\rho$ is $\cG$-equivariant.  Then $(\rho|_{\cont(\Gnaught )})^{-1}: \rho(\cont(\Gnaught)) \to \cont(\Gnaught)$ extends to a $\cG$-ucp map $\vphi:A \to \cont(\fb \cG)$.   The composition $\vphi \circ \rho: \Cstaress(\cG) \to \cont(\fb \cG)$ is a $\cG$-pseudo expectation of $\cG$. By (ii), it is faithful. So also $\rho$ must be faithful.

(iii) $\Rightarrow$ (i): Let $\rho: \Cstaress(\cH) \thra A$ be a surjective *-homomorphism that is faithful on $\cont(\fb \cG)$. Since $\cont(\Gnaught) \subseteq \Cstaress(\cG)$ has the ideal intersection property, it follows that $\rho$ is faithful on $\Cstaress(\cG) \subseteq \Cstaress(\cH)$. Denote by $\Eess: \Cstaress(\cG) \to \cont(\fb \cG)$ the natural $\cG$-pseudo expectation.  We observe that $A$ becomes a unital $\cG$-\Cstar-algebra by Proposition~\ref{prop:groupoid-cstar-algebras-carries-action}.  So $\Eess \circ (\rho|_{\Cstaress(\cG)})^{-1}: \rho(\Cstaress(\cG)) \to \cont(\fb \cG)$ extends to a $\cG$-ucp map $\vphi: A \to \cont(\fb \cG)$.  Write $\psi = \vphi \circ \rho: \Cstaress(\cH) \to \cont(\fb \cG)$.  Then $\cG$-rigidity implies that $\psi|_{\cont(\fb \cG)}$ is the identity map.  In particular, $\cont(\fb \cG)$ lies in the multiplicative domain of $\psi$.  Since $\Cstaress(\cH) = \cspan  \Cstaress(\cG) \cont(\fb \cG)$ and $\psi|_{\Cstaress(\cG)} = \Eess$, we find that $\psi$ is the natural conditional expectation of $\Cstaress(\cH)$.  In particular $\psi$, and thus also $\rho$, is faithful.
\end{proof}

Let us continue with the proof of (II).
\begin{proof}[Proof of Theorem~\ref{thm:intersection-property-essential-algebras}~(II)]
(i) $\Rightarrow$ (iv): First assume that $\cG$ is Hausdorff, in which case $\cH$ is Hausdorff, too. Since the unit space of $\cH$ is extremally disconnected, Proposition~\ref{prop:isotropy-clopen} says that the isotropy groupoid of $\cH$ is clopen.  So there is a well-defined conditional expectation $\Phi : \Cstarred(\cH) \to \Cstarred(\Iso(\cH))$. Let $F$ be the composition 
\[
  \Cstarred(\cH) \to \Cstarred(\Iso(\cH)) \to \prod_{y \in \mathfrak{Y}} \Cstarred(\cH_y^y) \to \prod_{y \in \mathfrak{Y}} \mathbb{C},
\]
where the first map is given by $\Phi$, the second map is the canonical projection map and the third map is given by the trivial representations, which exist because of amenability of $\cH_y^y$ for $y \in \mathfrak{Y}$. Observe that $F$ is tracial and that $\Cstarred(\Iso(\cH))$ lies in the multiplicative domain of $F$.  Denote by $H$ the Hilbert-$\cont(\fb \cG)$-module obtained from separation-completion coming with a map $\Lambda: \Cstarred(\cH) \to H$ satisfying
  \begin{gather*}
    \langle \Lambda(a), \Lambda(b) \rangle = F(ab^*)
  \end{gather*}
  for all $a,b \in \Cstarred(\cH)$.  Denote by $\rho$ the associated KSGNS-representation of $\Cstarred(\cH)$ \cite{lance1995-toolkit}, which is faithful on $\cont(\fb \cG)$ because $\mathfrak{Y}$ is dense in $\fb \cG$.  Since $\cont(\fb \cG) \subseteq \Cstarred(\cH)$ has the ideal intersection property, this implies that $\rho$ is faithful on $\Cstarred(\cH)$.

  Let $U \subseteq \Iso(\cH)$ be an open bisection and $f \in \contc(U)$.  Write $g = f \circ (\rms|_U)^{-1} \in \contc(\rms(U))$.  Then $F(f - g) = 0$.  Combined with the fact that $\Cstarred(\Iso(\cH))$ lies in the multiplicative domain of $F$, this leads to the calculation
  \begin{align*}
    \langle \Lambda(a) , \rho(f - g) \Lambda(b) \rangle
    & =
      F(a^* (f - g) b) \\
    & =
      F(b a^* (f - g)) \\
    & =
      F(ba^*) F(f - g) \\
    & = 0
      \eqcomma
  \end{align*}
  for all $a,b \in \Cstarred(\cH)$.  This shows that $f - g$ is zero.  Since $f \in \contc(U)$ was arbitrary, this shows that $U = \rms(U) \subseteq \fb \cG$.  Since $\Iso(\cH)$ is open by Proposition~\ref{prop:isotropy-clopen}, this proves the proposition for Hausdorff groupoids.

  If $\cG$ (and hence also $\cH$) is not Hausdorff, we observe that $\cont(\fb \cG) \subseteq \Cstaress(\cH) \cong \Cstarred(\cH_\Haus)$ has the ideal intersection property.  So the Hausdorffification $\cH_\Haus$ introduced in Section~\ref{sec:hausdorffification} is principal.  This implies that $\Iso(\cH) = \ol{\fb \cG}$.

(iv) $\Rightarrow$ (v): Existence of a $\cG$-pseudo expectation $\Cstaress(\cG) \to \cont(\fb G)$ follows from Proposition~\ref{prop:conditional-expectation-extremally-disconnected} and the fact that the local conditional expectation $\Ered: \Cstarred(\cG) \to \Dix(\Gnaught)$ factors through $\Cstaress(\cG)$.  We will prove uniqueness.  Let $\vphi: \Cstaress(\cG) \to \cont(\fb \cG)$ be a $\cG$-pseudo expectation. Applying $\cG$-injectivity of $\cont(\fb \cG)$ we can extend $\vphi$ to $\Cstaress(\cH)$.  Theorem~\ref{thm:essential-is-reduced-cstar-algebra} provides the natural identification $\Cstaress(\cH) \cong \Cstarred(\cH_\Haus)$, and we denote the resulting $\cG$-pseudo expectation by $\psi: \Cstarred(\cH_\Haus) \to \cont(\fb \cG)$.  By $\cG$-rigidity we find that $\psi$ restricted to $\cont(\Hnaught_\Haus) = \cont(\fb \cG)$ is the identity.  Consequently, $\cont(\fb \cG)$ lies in the multiplicative domain of $\psi$.  Let now $\gamma \subseteq \cH_\Haus \setminus \fb \cG$ be some open bisection and $f \in \contc(\gamma)$.  Since $\cH$ is essentially principal, we have $\Iso(\cH_\Haus) = \fb \cG$.  So if $f \neq 0$, then for every non-empty open subset $V_0 \subseteq \supp \psi(f)$ there is some nonempty open subset $V \subseteq V_0$ such that $\gamma V \cap V = \emptyset$.  This leads to
  \begin{gather*}
    \psi(f) \mathbb{1}_V
    =
    \psi(f \mathbb{1}_V)
    =
    \psi(\mathbb{1}_{\gamma V} f)
    =
    \mathbb{1}_{\gamma V} \psi(f)
    \eqstop
  \end{gather*}
  This is a contradiction, showing that $\psi(f) = 0$.  We showed that $\psi$ is the natural conditional expectation of $\Cstarred(\cH_\Haus)$.  So the composition 
  \begin{gather*}
    \vphi: \Cstaress(\cG) \hra \Cstaress(\cH) \stackrel{\cong}{\to} \Cstarred(\cH_\Haus) \to \cont(\fb \cG)    
  \end{gather*}
  equals the natural $\cG$-pseudo expectation by Theorem~\ref{thm:essential-is-reduced-cstar-algebra}.

(v) $\Rightarrow$ (ii): This is clear because the canonical $\cG$-pseudo expectation $\Cstaress(\cG) \to \cont(\fb G)$ is faithful.
\end{proof}

To prove Theorem~\ref{thm:intersection-property-essential-algebras}~(III), we prove the following result.
\begin{proposition}
\label{prop:SomeStabAmen}
Let $\cG$ be an {\'e}tale groupoid with compact Hausdorff space of units and $x \in \Gnaught$ be a point such that its $\cG$-orbit $\cG.x$ is contained in the interior of its closure $\ol{\cG.x}$ in $\Gnaught$. Take an arbitrary $y \in \pi^{-1}(x)$. Then $\cH_y^y$ is amenable.
\end{proposition}
\begin{proof}
Let $K$ be the closure of the complement $\Gnaught \setminus \ol{\cG.x}$ of the orbit closure of $x$. By construction, we have $\Gnaught = \ol{\cG.x} \cup K$. Moreover, by assumption, we have $(\cG.x) \cap K = \emptyset$. Both $\ell^{\infty} (\cG_x)$ and $\cont(K)$ are $\cG$-\Cstar-algebras: 

We have a homomorphism $C(\Gnaught) \to \ell^{\infty} (\cG_x), \, f \mapsto (f \circ \rmr) \vert_{\cG_x}$, and given an open bisection $\gamma$ with $\rms(\gamma) = U$ and $\rmr(\gamma) = V$, its action is given by
\begin{equation*}
  \ol{C_0(U) \ell^{\infty} (\cG_x)} \to \overline{C_0(V) \ell^{\infty} (\cG_x)}, \, f \mapsto f(\gamma^{-1} \cdot \sqcup).
\end{equation*}

Moreover, since $K$ is a closed $\cG$-invariant subspace of $\Gnaught$, the canonical $\cG$-action on $\cont(\Gnaught)$ restricts to a $\cG$-action on $\cont(K)$.

By combining these two $\cG$-actions, we obtain the diagonal $\cG$-action on $\ell^{\infty} (\cG_x) \oplus \cont(K)$.

We now claim that the corresponding $\cG$-equivariant homomorphism $C(\Gnaught) \to \ell^{\infty} (\cG_x) \oplus \cont(K)$ is injective. Indeed, given $f \in \cont(\Gnaught)$, $(f \circ \rmr) \vert_{\cG_x} = 0$ implies that $f \vert_{\cG.x} = 0$, and if in addition $f \vert_K = 0$, then $f = 0$ because $\Gnaught = \ol{\cG.x} \cup K$.

Now $\cG$-injectivity of $\cont(\fb \cG)$ implies that we obtain a $\cG$-ucp map $\psi : \ell^{\infty} (\cG_x) \oplus \cont(K) \to \cont(\fb \cG)$. Let $y \in \pi^{-1}(x)$ be arbitrary and let $\ev_y$ be the character on $\cont(\fb \cG)$ given by evaluation at $y$. Then we claim that $\ev_y \circ \psi$ vanishes on the direct summand $\cont(K)$ of $\ell^{\infty} (\cG_x) \oplus \cont(K)$. Indeed, $(\cG.x) \cap K = \emptyset$ means that $x \in \Gnaught \setminus K$. Hence we can find $f \in \cont(\Gnaught)$ with $f(x) = 1$ and $f \vert_K = 0$. Then
\begin{equation*}
  \ev_y(\psi(\cont(K)) = f(x) \ev_y(\psi(\cont(K))) = \ev_y(f \cdot \psi(\cont(K))) = \ev_y(\psi( (f \vert_K) \cdot \cont(K) )) = 0.
\end{equation*}

Observe that $\cG_x \cong \cH_y$. Therefore the restriction of $\ev_y \circ \psi$ gives rise to a ucp map $\ell^{\infty} (\cH_y) \to \mathbb{C}$ which is $\cH_y^y$-equivariant. It follows that $\cH_y^y$ is amenable.
\end{proof}
Clearly, Proposition~\ref{prop:SomeStabAmen} and Theorem~\ref{thm:intersection-property-essential-algebras}~(II) imply Theorem~\ref{thm:intersection-property-essential-algebras}~(III).

We record the following direct consequences.

\begin{corollary}
\label{cor:MinAllStabAmen}
Let $\cG$ be an {\'e}tale groupoid with compact Hausdorff space of units. If $\cG$ is minimal, then all stabilizer subgroups in $\cG \ltimes \fb \cG$ are amenable.
\end{corollary}

\begin{corollary}
\label{cor:IntProp_TopTrans}
Let $\cG$ be an {\'e}tale groupoid with compact Hausdorff space of units.  Assume that $\cG$ satisfies one of the following properties:
\begin{itemize}
    \item $\cG$ is minimal, or
    \item $\cG$ is topologically transitive and Hausdorff, or
    \item $\cG$ is topologically transitive and $\sigma$-compact.
\end{itemize}
Then (i) to (v) from Theorem~\ref{thm:intersection-property-essential-algebras} are all equivalent.
\end{corollary}

\begin{remark}
\label{rem:AllEquivAllStabAmen}
Because of Remark~\ref{rem:AllStabAmenSit}, if all stabilizer groups of our groupoid $\cG$ are amenable or if $\cG$ is a transformation groupoid, then we again conclude that (i) to (v) are equivalent if $\cG$ satisfies the assumptions of Theorem~\ref{thm:intersection-property-essential-algebras}~(I).
\end{remark}

\begin{remark}
  \label{rem:reduced-groupoid-cstar-algebra-intersection-property}
  We made it clear that it is necessary to consider the essential groupoid \Cstar-algebra in order to obtain a relation between the ideal structure and the algebraic-dynamical structure of a groupoid.  This can be underpinned by the fact the ideal intersection property for $\Cstarred(\cG \ltimes \fb \cG)$ implies that $\cG$ is Hausdorff.  Indeed, assume that $\Cstarred(\cG \ltimes \fb \cG)$ has the ideal intersection property. Since the map $\Cstarred(\cG \ltimes \fb \cG) \to \Cstaress(\cG \ltimes \fb \cG)$ is injective on $\cont(\fb \cG)$, it follows that the ideal of singular functions in $\Cstarred(\cG \ltimes \fb \cG)$ is trivial.  Since $\ol{\fb \cG} \subseteq \cG \ltimes \fb \cG$ is clopen by Lemma \ref{lem:extended-unit-space-clopen}, this implies that $\cG \ltimes \fb \cG$ is Hausdorff.  Then also $\cG$ is Hausdorff, since the quotient map $\cG \ltimes \fb \cG \to \cG$ maps $\ol{\fb \cG}$ onto $\ol{\Gnaught}$.
\end{remark}

\section{Essentially confined subgroupoids}
\label{sec:confined-subgroupoids}

An important point in our understanding of \Cstar-simplicity of discrete groups has been the characterisation in terms of confined subgroups from \cite{kennedy15-cstarsimplicity}, since this provides criteria that can be checked in terms of the group itself.  Note that the terminology of ``confined'' subgroups does not appear in \cite{kennedy15-cstarsimplicity}, although it does appear in prior work on group algebras of locally finite groups \cite{hartleyyalesskij1997}. 

In this section we obtain results about the ideal intersection property for groupoids.  The key achievement is Theorem~\ref{thm:seciso}, which establishes a characterisation in terms of confined amenable sections of isotropy groups for groupoids with a compact Hausdorff space of units under the same hypotheses as in Theorem~\ref{thm:intersection-property-essential-algebras}~(II). Subsequently, in \S~\ref{sec:alexandrov-groupoid}, we will be able to remove the compactness assumption on the unit space by considering a notion of Alexandrov groupoid, whose unit space is the one-point compactification of the unit space of the original groupoid.

Given an {\'e}tale groupoid $\cG$ with locally compact Hausdorff space of units, as in \cite[\S~4.2.4, p. 78]{borys2020-thesis}, we define $\Sub(\cG)$ as the space of all subgroups of the isotropy groups of $\cG$, viewed as a subspace of the space $\mathfrak{C}(\cG)$ of all closed subsets of $\cG$ equipped with the Chabauty topology (see for instance \cite{fell1962}). We will denote elements of $\Sub(\cG)$ by $(x,H)$, where $x \in \Gnaught$ and $H$ is a subgroup of $\cG_x^x$.  Observe that the first factor projection $\Sub(\cG) \to \Gnaught$ together with the conjugation action of $\cG$ turns $\Sub(\cG)$ into a $\cG$-space.

The next definition is a modification of \cite[Definition~4.2.24]{borys2020-thesis} and at the same time a generalization to the non-Hausdorff case.
\begin{definition}
  \label{def:section-isotropy}
  Let $\cG$ be an {\'e}tale groupoid with locally compact Hausdorff space of units.  A section of isotropy subgroups is a collection $\{ (x,H_x) \mid x \in \mathfrak{X} \}$ of isotropy subgroups $H_x \subseteq \cG_x^x$ for all $x$ in a dense subset $\mathfrak{X}$ of $\Gnaught$.  Such a section is called amenable if all $H_x$ are amenable.  A section of isotropy subgroups  $\Lambda = \{ (x,H_x) \mid x \in \mathfrak{X} \}$ is called essentially confined if there exists $x \in \Gnaught$ such that $H \not\subseteq \ol{\Gnaught}$ for all $(x,H) \in \ol{\cG \cdot \Lambda}$.
\end{definition}
We recall that the terminology of confinedness first arose in the study of groups, where it qualifies subgroups that are isolated from the trivial group in a strong sense.

\subsection{The ideal intersection property for groupoids with compact space of units}
\label{sec:compact-unit-space}

Combining arguments similar to those used by Kawabe \cite{kawabe17} and Borys \cite{borys2020-thesis} with the techniques developed in the present article yields the following characterisation of the intersection property for {\'e}tale groupoids with compact space of units, which is new in the non-Hausdorff case and in the Hausdorff case completes Borys's sufficient criterion and at the same time takes into account that isotropy groups of the Furstenberg groupoid are not always amenable (see Remark~\ref{rem:NotAllStabAmen}).
\begin{theorem}
  \label{thm:seciso}
  Let $\cG$ be an {\'e}tale groupoid with compact Hausdorff space of units.  Assume that $\cG$ is Hausdorff, that $\cG$ is minimal or that $\cG$ is $\sigma$-compact. Further assume that $\mathfrak{Y} = \{ y \in \fb \cG \mid (\cG \ltimes \fb \cG)_y^y \text{ is amenable} \}$ is dense in $\fb \cG$. Then $\cont(\Gnaught) \subseteq \Cstaress(\cG)$ has the intersection property if and only if $\cG$ has no essentially confined amenable sections of isotropy subgroups.
\end{theorem}

\begin{remark}
  \label{rem:recovering-previous-work-confined}
  Previous work on crossed product \Cstar-algebras and groupoid \Cstar-algebras obtained partial results that can be recovered from Theorem~\ref{thm:seciso}.

  Kawabe considered in \cite{kawabe17} crossed product \Cstar-algebras $\cont(X) \redtimes \Gamma$ for a compact Hausdorff space $X$ and a discrete group $\Gamma$.  Note that $\cont(X) \redtimes \Gamma \cong \Cstarred(\Gamma \ltimes X)$ is the groupoid \mbox{\Cstar-algebra} associated with the transformation groupoid $\Gamma \ltimes X$.  Kawabe's main result \cite[Theorem 1.6]{kawabe17} characterised the ideal intersection property for $\cont(Y) \subseteq \cont(Y) \redtimes \Gamma$ where $Y \subseteq X$ runs through all closed $\Gamma$-invariant subsets by the following property: For every point $x \in X$ and every amenable subgroup $\Lambda \leq \Gamma_x$, there is a net $(g_i)$ in $\Gamma$ such that $(g_i x)$ converges to $x$ and $(g_i \Lambda g_i^{-1})$ converges to $\{e\}$ in the Chabauty topology.  We observe that $\Sub(X \rtimes \Gamma) = \{(x, \Lambda) \in X \times \Sub(\Gamma) \mid \Lambda \leq \Gamma_x\}$ as topological spaces.  This shows how Kawabe's result is generalised by our Theorem~\ref{thm:seciso} (see also Remark~\ref{rem:AllEquivAllStabAmen}).

  Borys considered in \cite{borys2020-thesis,borys2019-boundary} {\'e}tale Hausdorff groupoids with compact space of units and proved that the absence of confined amenable sections in the isotropy is a sufficient criterion for simplicity of $\Cstarred(\cG)$. However, his argument used amenability of all isotropy subgroups of the Furstenberg groupoid, which is not true in general (see Remark~\ref{rem:NotAllStabAmen}). Hence we had to adjust the notion of confined sections of isotropy groups, and with this new notion, our Theorem~\ref{thm:seciso} provides the correct analogue of Borys' sufficient criterion and at the same time shows its necessity.

We will comment on Kwa{\'s}niewski-Meyer's result \cite[Theorem 7.29]{kwasniewskimeyer2019-essential} in Remark~\ref{rem:previous-work-locally-compact-case}, since groupoids with locally compact space of units are considered in this work.
\end{remark}

We begin with the following observation, which is probably well known.
\begin{lemma}
  \label{lem:space-of-subgroupoids-closed}
  Let $\cG$ be an {\'e}tale groupoid with locally compact Hausdorff space of units.  $\Sub(\cG)$ is closed in $\mathfrak{C}(\cG)$.
\end{lemma}
\begin{proof}
Suppose that $(x_i,H_i) \in \Sub(\cG)$ converges to $C$ in $\mathfrak{C}(\cG)$. As explained in \cite{fell1962}, the subset $C$ consists precisely of the elements of $g \in \cG$ with the following property: For every open subset $U \subseteq \cG$ with $g \in U$, we have $H_i \cap U \neq \emptyset$ eventually. First of all, we claim that $C \subseteq \cG_x^x$ for some $x \in \Gnaught$. Indeed, take $g, h \in C$ with $\rms(g) = x$ and $\rms(h) = y$. If $x \neq y$, then we can find disjoint open subsets $U$ and $V$ of $\Gnaught$ containing $x$ and $y$, respectively. Then we must have $H_i \cap \rms^{-1}(U) \neq \emptyset$ eventually, which implies $x_i \in U$ eventually. Similarly, $x_i \in V$ eventually. This shows that $x = y$. The same argument, applied to the range map, shows that $C \subseteq \cG_x^x$. This also shows that we must have $\lim_i x_i = x$. It remains to show that $C$ is a subgroup of $\cG_x^x$. Take $g \in C$ and an open subset $V \subseteq \cG$ with $g^{-1} \in V$. Then $g \in V^{-1}$, so that $H_i \cap V^{-1} \neq \emptyset$ eventually. Applying the inverse map, we deduce that $H_i \cap V \neq \emptyset$ eventually. Hence $g^{-1} \in C$. Now take $g, h \in C$ and an open subset $W \in \cG$ with $gh \in W$. By continuity of multiplication, we find open subsets $U, V \subseteq \cG$ with $g \in U$, $h \in V$ such that $UV \subseteq W$. Then $g, h \in C$ implies $H_i \cap U \neq \emptyset$ and $H_i \cap V \neq \emptyset$ eventually. Thus $H_i \cap UV \neq \emptyset$ eventually. We conclude that $gh \in C$. So $C \in \Sub(\cG)$, as desired. 
\end{proof}

The next proposition establishes the reverse implication of Theorem~\ref{thm:seciso}.  Portions of the proof will closely follow the presentation in \cite{borys2020-thesis} for Hausdorff groupoids.
\begin{proposition}
  \label{prop:norec->intprop}
  Let $\cG$ be an {\'e}tale groupoid with compact Hausdorff space of units.  Suppose that $\cG$ is Hausdorff, $\cG$ is minimal or $\cG$ is covered by countably many bisections. Further assume that $\mathfrak{Y} = \{ y \in \fb \cG \mid (\cG \ltimes \fb \cG)_y^y \text{ is amenable} \}$ is dense in $\fb \cG$. If $\cG$ has no essentially confined amenable sections of isotropy subgroups, then $\cG \ltimes \fb \cG$ is essentially principal.
\end{proposition}
\begin{proof}
Let $\pi: \: \cG \ltimes \fb \cG \to \cG$ be the canonical projection. Choose a subset $\mathfrak{Y}' \subseteq \mathfrak{Y}$ which maps bijectively onto $\pi(\mathfrak{Y})$ under $\pi$ and consider the section of isotropy subgroups $\Lambda = \{ \pi((\cG \ltimes \fb \cG)_y^y) \mid y \in \mathfrak{Y}' \}$, which is amenable by assumption.  The same argument as in \cite[Proof of Theorem~4.2.25]{borys2020-thesis}, which does not use the assumption that $\cG$ is Hausdorff, shows that $\{ \pi((\cG \ltimes \fb \cG)_y^y) \mid y \in \fb \cG \}$ is a closed invariant subspace of $\Sub(\cG)$.  So our assumption that $\cG$ has no essentially confined amenable section of isotropy subgroups implies that for all $x \in \Gnaught$, there exists $y \in \pi^{-1}(x)$ such that $\pi((\cG \ltimes \fb \cG)_y^y) \subseteq \ol{\Gnaught}_x$.

  Let us consider the case when $\cG$ is minimal.  Then we define
  \begin{gather*}
    Z = \{y \in \fb \cG \mid \pi((\cG \ltimes \fb \cG)_y^y) \subseteq \ol{\Gnaught}_{\pi(y)} \}
  \end{gather*}
  and observe that $\pi(Z) = \Gnaught$ and $Z$ is $\cG$-invariant.  If $y \in \ol{Z}$ and $g \in (\cG \ltimes \fb \cG)_y^y$, we take an open bisection $U \subseteq \Iso(\cG \ltimes \fb \cG)$ that contains $g$.  Let $(y_i)_i$ be a net in $\rms(U) \cap Z$ converging to $y$.  Then $(s|_U)^{-1}(y_i) \to g$ and thus $\pi(g) = \lim_i \pi((s|_U)^{-1}(y_i)) \in \ol{\Gnaught}$.  This proves that $y \in Z$.  So $Z$ is closed, which implies $Z = \fb \cG$ by irreducibility.  Equivalently, $\pi(\Iso(\cG \ltimes \fb \cG)) \subseteq \ol{\Gnaught}$.  Put $A = \Iso(\cG \ltimes \fb \cG) \setminus \ol{\fb \cG}$.  We will conclude the proof in the minimal case by assuming that $A \neq \emptyset$ and deducing a contradiction.   If $A \neq \emptyset$, we can find an open bisection $U \subseteq \cG$ such that $U \cap \pi(A) \neq \emptyset$.  Then $V = \pi^{-1}(U) \cap A \subseteq \cG \ltimes \fb \cG$ is a non-empty open bisection.  Since $\cG$ is minimal, Lemma~\ref{lem:G-space-projections-dense-subsets} says that $\rms(\pi(V)) = \pi(\rms(V))$ contains a non-empty open subset.  Since $\pi(V) \subseteq U$ is contained in an open bisection, this implies that $\emptyset \neq \pi(V)^\circ \subseteq \pi(A) \subseteq \ol{\Gnaught} \setminus \Gnaught$.  This is the desired contradiction.  Hence $\cG \ltimes \fb \cG$ is essentially principal.

  Now assume that $\cG$ is Hausdorff or covered by countably many open bisections, and define
  \begin{gather*}
    Z = \{ y \in \fb \cG \mid (\cG \ltimes \fb \cG)_y^y \subseteq \ol{\fb \cG}_y \}
    \eqstop
  \end{gather*}
  As the extended unit space $\ol{\fb \cG}$ is normal in $\cG \ltimes \fb \cG$, it follows that $Z$ is $\cG$-invariant.  Moreover, the identity $\fb \cG \setminus Z = \rms\bigl( {\rm Iso}(\cG \ltimes \fb \cG) \setminus \ol{\fb \cG} \bigr )$ implies that $Z$ is closed.  Hence $\pi(Z)$ is a closed subset of $\Gnaught$.  We want to show that $\pi(Z) = \Gnaught$. If $\cG$ is Hausdorff, then our statement that for all $x \in \Gnaught$, there exists $y \in \pi^{-1}(x)$ such that $\pi((\cG \ltimes \fb \cG)_y^y) \subseteq \ol{\Gnaught}_x = \{ x \}$ implies that for all $x \in \Gnaught$, there exists $y \in \pi^{-1}(x)$ such that $(\cG \ltimes \fb \cG)_y^y = \{ y \}$. This shows $\pi(Z) = \Gnaught$, as desired. If $\cG$ is covered by countably many open bisections, then \cite[Lemma 7.15]{kwasniewskimeyer2019-essential} implies that the set $\{ x \in \Gnaught \mid \ol{\Gnaught}_x = \{x\} \}$ is dense in $\Gnaught$.  Now for all $x$ in this set, there exists $y \in \pi^{-1}(x)$ with $\pi((\cG \ltimes \fb \cG)_y^y) \subseteq \ol{\Gnaught}_x = \{x\}$, that is $\pi((\cG \ltimes \fb \cG)_y^y) = \{x\}$ and thus $(\cG \ltimes \fb \cG)_y^y = \{y\}$.  This shows that $\pi(Z)$ is dense in $\Gnaught$.  Since $\pi(Z)$ is also closed, we conclude that $\pi(Z) = \Gnaught$, as desired.  Since $\fb \cG$ is $\cG$-irreducible,  it follows that $Z = \fb \cG$, that is  ${\rm Iso}(\cG \ltimes \fb \cG) \subseteq \ol{\fb \cG}$.
\end{proof}

Let us now explain the proof of the forward implication of Theorem~\ref{thm:seciso}.  The following proof is an adaptation of Kawabe's argument for dynamical systems of groups \cite{kawabe17}, as presented in \cite[Proof of Theorem~3.3.6]{borys2020-thesis}.  Note that this proposition is new even in the Hausdorff case.
\begin{proposition}
  \label{prop:unique-pseudo-expectation-implies-no-essentially-confined-amenable-sections}
  Let $\cG$ be an {\'e}tale groupoid with compact Hausdorff space of units.  Assume that $\cG$ is Hausdorff, that $\cG$ is minimal or that $\cG$ is $\sigma$-compact. If there is a unique $\cG$-pseudo-expectation $\Cstaress(\cG) \to C(\fb \cG)$, then $\cG$ has no essentially confined amenable sections of isotropy subgroups.
\end{proposition}
\begin{proof}
  Let $\Lambda = \{ (x,H_x) \mid x \in \mathfrak{X} \}$ be an amenable section of isotropy subgroups and $Y$ its orbit closure, that is $Y = \ol{\cG \cdot \Lambda} \subseteq \Sub(\cG)$. Define for all $a \in \cC(\cG)$ a function $\theta(a) : Y \to \mathbb{C}$ by $\theta(a)(x,H) \mathrel{:=} \sum_{\rms(g) = x} a(g) \mathbb{1}_H(g)$. 
  
  We claim that $\theta(a)$ is continuous on $Y$. It suffices to prove this for $a \in \contc(U)$, where $U \subseteq \cG$ is an open bisection. For such $a$, our construction yields
  \begin{gather*}
    \theta(a)(x,H) =
    \begin{cases}
      a(Ux) \mathbb{1}_H(Ux) & \text{if } x \in \rms(U) \\
      0 & \text{otherwise.}
    \end{cases}
  \end{gather*}
  Now let $O = \{ g \in U \mid a(g) \neq 0 \}$, which is open, and let $K \subseteq U$ be compact such that $O \subseteq K$. The sets $\cO_O = \{(x,H) \in \Sub(\cG) \mid H \cap O \neq \emptyset\}$ and $\cO'_K = \{(x,H) \in \Sub(\cG) \mid H \cap K = \emptyset\}$ are open in $\Sub(\cG)$. We have $\theta(a) \vert_{(Y \cap \cO_O)^{\rm c}} \equiv 0$, and $\theta(a) \vert_{(Y \cap \cO'_K)^{\rm c}}$ is given as the composition
  \begin{gather*}
    (Y \cap \cO'_K)^{\rmc} \to \rms(K) \to \CC:  (x,H) \mapsto x \mapsto a(Ux)
    \eqstop
  \end{gather*}
  It follows that both restrictions $\theta(a) \vert_{(Y \cap \cO_O)^{\rm c}}$ and $\theta(a) \vert_{(Y \cap \cO'_K)^{\rm c}}$ are continuous. Moreover, $O \subseteq K$ implies that $(Y \cap \cO_O) \cap (Y \cap \cO'_K) = \emptyset$, so that we obtain a decomposition into two closed sets $Y = (Y \cap \cO_O)^{\rm c} \cup (Y \cap \cO'_K)^{\rm c}$. Since $\theta(a)$ is continuous on both of these closed subsets, it follows that $\theta(a)$ is continuous on $Y$.

  For $a \in \cC(\cG)$ and $(x,H) \in \cG \cdot \Lambda$, $\theta(a)(x,H)$ is given by the image of $a$ under the composition
  \begin{equation*}
    \begin{tikzcd}
      \Cstarred(\cG) \arrow{r}{\rE_x} &
      \Cstarred(\cG_x^x) \arrow{r}{\rE_H} &
      \Cstarred(H) \arrow{r}{\chi} &
      \CC
    \end{tikzcd}
  \end{equation*}
  where $\rE_x$ and $\rE_H$ are the natural conditional expectations and $\chi$ is the character corresponding to the trivial representation. Here we are using that $H$ is amenable. It follows that $\vert \theta(a)(x,H) \vert \leq \Vert a \Vert_{\Cstarred(\cG)}$. By continuity of $\theta(a)$ and density of $\cG \cdot \Lambda$ in $Y$, we conclude that $\Vert \theta(a) \Vert_{C(Y)} \leq \Vert a \Vert_{\Cstarred(\cG)}$. Hence $\theta$ extends to a contractive map $\theta : \Cstarred(\cG) \to C(Y)$, which is clearly unital and completely positive.

  Now we show that $\theta$ is $\cG$-equivariant.  Take an open bisection $\gamma$, an open bisection $U$ with $\rmr(U), \rms(U) \subseteq \rms(\gamma)$ and $a \in \contc(U)$.  Then $\alpha_\gamma(a) \in \contc(\gamma U \gamma^*)$.  For all $(x,H) \in Y$, we have $\theta(\alpha_\gamma(a))(x,H) = 0$ if $x \notin \rms(\gamma U \gamma^*)$ and $\theta(a)(\psi_{\gamma^*}(x), \gamma^* H \gamma) = 0$ if $\psi_{\gamma^*}(x) \notin \rms(U)$.  The conditions $x \notin \rms(\gamma U \gamma^*) = \psi_\gamma \circ \rms(U)$ and $\psi_{\gamma^*}(x) \notin \rms(U)$ are equivalent. Using the fact that $\psi_{\gamma^*}(x) = \gamma^* x \gamma$, for $x \in \rms(\gamma U \gamma^*)$, we have
  \begin{align*}
    \theta(\alpha_\gamma(a))(x,H)
    & = \alpha_\gamma(a)(\gamma U \gamma^* x) \mathbb{1}_H(\gamma U \gamma^* x) \\
    & = a(U \gamma^* x \gamma) \mathbb{1}_H(\gamma U \gamma^* x) \\
    & = a(U \gamma^* x \gamma) \mathbb{1}_{\gamma^* H \gamma}(U \gamma^* x \gamma) \\
    & = \theta(a)(\psi_{\gamma^*} (x), \gamma^* H \gamma)
      \eqstop
  \end{align*}

  The projection onto the first coordinate $Y \to \Gnaught$ induces a $\cG$-equivariant embedding $\cont(\Gnaught) \to \cont(Y)$.  Hence $\cG$-injectivity of $\cont(\fb \cG)$ provides us with a $\cG$-equivariant ucp map $\vphi: \cont(Y) \to \cont(\fb \cG)$.  It follows that $\vphi \circ \theta:  \Cstarred(\cG) \to C(\fb \cG)$ is a $\cG$-ucp map with $(\varphi \circ \theta) \vert_{\cont(\Gnaught)} = \id$. We claim that $\varphi \circ \theta$ factors as
  \begin{equation*}
    \begin{tikzcd}
      \Cstarred(\cG) \arrow{r}{\mathfrak{q}} & \Cstaress(\cG) \arrow{r}{\Psi} & C(\fb \cG),
    \end{tikzcd}
  \end{equation*}
  or equivalently, that $(\varphi \circ \theta)(J_{\rm sing}) = 0$.  If $\cG$ is Hausdorff, this is tautological.  Otherwise, given $a \in J_{\mathrm{sing}}$ we have $\Ered(a^*a) = 0$.  So by Proposition~\ref{prop:singular-elements-vanishing-generically}, there is a dense subset $U \subseteq \Gnaught$ such that $\widehat{a^*a}|_U = 0$.  Since
  \begin{gather*}
    \widehat{a^*a}(x) = \sum_{g \in \cG_x} |\hat a(g)|^2
  \end{gather*}
  for all $x \in \Gnaught$, it follows that $\hat a|_{\cG_x} = 0$ for all $x \in U$.  Thus $\theta(a)|_{\{x\} \times \Sub(\cG_x^x)} = 0$.  By $\cont(\Gnaught)$-modularity, it follows that $(\vphi \circ \theta)(a) \vert_{\pi^{-1}(x)} = 0$ for all $x \in U$, where $\pi:  \cG \ltimes \fb \cG \to \cG$ is the natural projection.  By Lemma~\ref{lem:G-space-projections-dense-subsets} the preimage of $U$ in $\fb \cG$ is dense.  So $(\varphi \circ \theta)(a) = 0$ follows.   Hence, indeed, $(\vphi \circ \theta)(J_{\rm sing}) = 0$.

  Now it follows that $\Psi$ is a $\cG$-pseudo-expectation $\Cstaress(\cG) \to \cont(\fb \cG)$.  By assumption, there is only one such $\cG$-pseudo-expectation.  It follows that $\Psi \circ \mathfrak{q} = \rE \circ \pi_*$, where $\rE: \Cstarred(\cG \ltimes \fb \cG) \to \cont(\fb \cG)$ is the (unique) conditional expectation described in Proposition~\ref{prop:conditional-expectation-extremally-disconnected}.  Thus $\vphi \circ \theta = \rE \circ \pi_*$.

  Now take $y \in \fb \cG$ and set $x = \pi(y)$.  The composition $\ev_y \circ \vphi$ defines a state on $\cont(Y)$, hence it corresponds to a probability measure $\mu_y = \vphi_*(\delta_y)$ on $Y$.  The commutative diagram
  \begin{equation*}
    \begin{tikzcd}
      \fb \cG \arrow{rr}{\vphi_*} \arrow{dr} && \cP(Y) \arrow{dl} \\
      & \cP(\Gnaught)
    \end{tikzcd}
  \end{equation*}
  shows that $\supp(\mu_y) \subseteq \{x\} \times \Sub(\cG_x^x)$.  Now assume, for the sake of contradiction, that $\Lambda$ is essentially confined. Then there exists $x \in \Gnaught$ such that $H \not\subseteq \ol{\Gnaught}_x$ for all $(x,H) \in Y = \ol{\cG \cdot \Lambda}$.  Choose $y \in \fb \cG$ with $\pi(y) = x$. We claim that $\mu_y(\{(x,H) \in Y \mid g \in H\}) = 0$ for all $g \in \cG_x^x$ with $g \notin \ol{\Gnaught}$.  Indeed, take an open subset $U \subseteq \cG \setminus \ol{\Gnaught}$ with $g \in U$.  Choose $a \in \contc(U)$ satisfying $a(g) = 1$.  Then $U \cap \ol{\Gnaught} = \emptyset$ implies that $\pi^{-1}(U) \cap \ol{\fb \cG} = \emptyset$.  Hence it follows that $\rE(a \circ \pi) = 0$.  Now, using ${\rm supp}(\mu_y) \subseteq \{x\} \times \Sub(\cG_x^x)$, we obtain
  \begin{align*}
    \mu_y(\{(x,H) \in Y \mid g \in H\})
    & = \int_Y \mathbb{1}_H(g) \rmd \mu_y(x,H) \\
    & = \int_Y \theta(a)(x,H) \rmd \mu_y(x,H) \\
    & = (\ev_y \circ \vphi \circ \theta)(a) \\
    & = \ev_y ( (\rE \circ \pi^*) (a)) \\
    & = 0.
  \end{align*}
  To finish the proof, let us observe that $H \not \subseteq \ol{\Gnaught}_x$ for all $(x, H) \in Y$ implies that
  \begin{gather*}
    (\{x\} \times \Sub(\cG_x^x)) \cap Y = \bigcup_{g \in \cG_x^x \setminus \ol{\Gnaught}_x} \{(x,H) \mid g \in H\}
    \eqstop
  \end{gather*}
  The latter set is exhausted by compact subsets of the form $\{x\} \times K_F$, where $K_F = \{H \in \Sub(\cG_x^x) \mid F \cap H \neq \emptyset\}$ and $F \subseteq \cG_x^x \setminus \ol{\Gnaught}$ is finite.  Since $\mu_y$ is supported on $\{x\} \times \Sub(\cG_x^x)$, inner regularity implies that there is a finite subset $F \subseteq \cG_x^x \setminus \ol{\Gnaught}_x$ such that $0 < \mu_y(\{x\} \times K_F)$.  We find that
  \begin{gather*}
    0
    <
    \mu_y(\{x\} \times K_F) 
    \leq
    \sum_{g \in F} \mu_y( \{ (x,H) \mid g \in H \} )
    =
    0
    \eqstop
  \end{gather*}
  This is a contradiction.
\end{proof}

Let us finish by proving the main theorem in this section.
\begin{proof}[Proof of Theorem \ref{thm:seciso}]
  By Theorem~\ref{thm:intersection-property-essential-algebras}~(II), we know that the ideal intersection property for $\Cstaress(\cG)$ is equivalent to essential principality of $\cG \ltimes \fb \cG$ and to the uniqueness of a $\cG$-pseudo expectation $\Cstaress(\cG) \to \cont(\fb \cG)$.   Also observe that every $\sigma$-compact {\'e}tale groupoid is covered by countably many open bisections.  So the present theorem follows from Propositions~\ref{prop:norec->intprop} and~\ref{prop:unique-pseudo-expectation-implies-no-essentially-confined-amenable-sections}.
\end{proof}

We record the following immediate consequence of Theorem~\ref{thm:intersection-property-essential-algebras}~(III) (see also Corollary~\ref{cor:IntProp_TopTrans}) and Theorem~\ref{thm:seciso}.
\begin{corollary}
\label{cor:seciso}
Let $\cG$ be an {\'e}tale groupoid with compact Hausdorff space of units.  Assume that $\cG$ is Hausdorff, that $\cG$ is minimal or that $\cG$ is $\sigma$-compact. In addition, suppose that the set of points $x \in \Gnaught$ whose $\cG$-orbit $\cG.x$ is contained in the interior of its closure $\ol{\cG.x}$ in $\Gnaught$ is dense in $\Gnaught$. In particular, these assumptions are satisfied if $\cG$ is an {\'e}tale groupoid with compact Hausdorff space of units satisfying one of the following
\begin{itemize}
    \item $\cG$ is minimal, or
    \item $\cG$ is topologically transitive and Hausdorff, or
    \item $\cG$ is topologically transitive and $\sigma$-compact.
\end{itemize}
Then $\cont(\Gnaught) \subseteq \Cstaress(\cG)$ has the intersection property if and only if $\cG$ has no essentially confined amenable sections of isotropy subgroups.
\end{corollary}

\subsection{The Alexandrov groupoid and the case of non-compact unit spaces}
\label{sec:alexandrov-groupoid}

In this section we provide a full characterisation of the ideal intersection property for {\'e}tale groupoids satisfying similar assumptions as in \S~\ref{sec:compact-unit-space}.  To this end, we will combine Theorem~\ref{thm:seciso} with the study of a suitable notion of Alexandrov groupoid.

Recall that if $X$ is a locally compact Hausdorff space, its Alexandrov or one-point compactification is the compact Hausdorff space $X^+ = X \sqcup \{\infty\}$ whose topology is determined by specifying that the inclusion $X \hra X^+$ is a homeomorphism onto its image and a neighbourhood basis of $\infty$ is provided by the set $\{\infty\} \cup (X \setminus K)$, where $K \subseteq X$ runs through compact subsets of $X$.  In particular, $X \subseteq X^+$ is dense if $X$ is non-compact and $\infty$ is an isolated point in $X^+$ if $X$ is compact. 

The next definition provides a suitable notion of Alexandrov compactification for groupoids whose unit space is not necessarily compact.
\begin{definition}
  \label{def:alexandrov-compactification}
  Let $\cG$ be an {\'e}tale groupoid with locally compact Hausdorff space of units.  Then the Alexandrov groupoid $\cG^+$ is the set $\cG \cup (\Gnaught)^+$ with the topology determined by specifying that the inclusions $\cG, (\Gnaught)^+ \subseteq \cG^+$ are open, and the groupoid structure extending the groupoid structure of $\cG$ and making $\infty$ a unit.
\end{definition}
We directly observe that $\cG^+$ is an {\'e}tale groupoid whose unit space is, by construction, the compact Hausdorff space $(\Gnaught)^+$.  Note also that Anantharaman-Delaroche in the appendix of \cite{anantharamandelaroche2016v2-exact} considered a fibrewise Alexandrov compactification of groupoids, which is different from the present construction.

We identify the essential groupoid \Cstar-algebra of the Alexandrov groupoid.
\begin{proposition}
  \label{prop:essential-groupoid-cstar-algebra-alexandrov-groupoid-unitisation}
  Let $\cG$ be an {\'e}tale groupoid with locally compact Hausdorff space of units.  Then the inclusion $\cC(\cG) \subseteq \cC(\cG^+)$ induces an inclusion $\Cstaress(\cG) \unlhd \Cstaress(\cG^+)$ isomorphic with the unitisation $\Cstaress(\cG) \unlhd \Cstaress(\cG)^+$.
\end{proposition}
\begin{proof}
  Let us first observe that there is indeed an inclusion $\cC(\cG) \subseteq \cC(\cG^+)$, since $\cG \subseteq \cG^+$ is open.  It moreover defines a *-isomorphism $\cC(\cG)^+ \cong \cC(\cG^+)$ mapping $1 \in \cC(\cG)^+$ to $\mathbb{1}_{(\Gnaught)^+}$.  Let us identify the essential \Cstar-algebra norm on $\cC(\cG)$ and $\cC(\cG^+)$.  To this end, we show that for all $f \in \cC(\cG^+)$ and $h_1, \dotsc, h_n \in \cC(\cG)$ there is $\tilde f \in \cC(\cG)$ such that $f * h_i = \tilde f * h_i$ for all $i \in \{1, \dotsc, n\}$.  For each $i$, we know that $\supp h_i \subseteq \cG$ is compact.  So also $K = \bigcup_{i = 1}^n \rmr(\supp h_i) \subseteq \Gnaught$ is compact.  Let $g \in \contc(\Gnaught)$ be a function satisfying $0 \leq g \leq 1$ and $g|_K \equiv 1$.  Put $\tilde f = f \cdot (g \circ \rms)$ and observe that $\tilde f \in \cC(\cG)$.  Then for $x \in \cG$ and $i \in \{1, \dotsc, n\}$, we have
  \begin{gather*}
    \tilde f * h_i(x)
    = \sum_{x = y z} \tilde f(y) h_i(z)
    = \sum_{x = y z} f(y) g(\rmr(z)) h_i(z)
    = \sum_{x = y z} f(y) h_i(z)
    = f*h_i (x)
    \eqstop
  \end{gather*}

  Denote by $\Eess: \Cstaress(\cG) \to \Dix(\Gnaught)$ the local condition expectation and let $f \in \cC(\cG)$.  Recall that $\Gnaught \subseteq (\Gnaught)^+$ is dense if $\Gnaught$ is non-compact and $\infty$ is an isolated point in $(\Gnaught)^+$ otherwise.  Combining this with \cite[Proposition 4.10]{kwasniewskimeyer2019-essential} and the previous paragraph, we find that for $f \in \contc(\cG)$,
  \begin{align*}
    \|f\|^2_{\Cstaress(\cG)}
    & = \sup \{ \|\rE(g^* * f^* * f * g) \| \mid g \in \cC(\cG) \} \\
    & = \sup_{g \in \cC(\cG)} \inf_{U \subseteq \Gnaught \text{ dense open}} \sup_{x \in U} |g^* * f^* * f * g|(x) \\
    & = \sup_{g \in \cC(\cH)} \inf_{U \subseteq (\Gnaught)^+ \text{ dense open}} \sup_{x \in U} |g^* * f^* * f * g|(x) \\
    & = \|f\|^2_{\Cstaress(\cG^+)}
      \eqstop
  \end{align*}
  This shows that $\cC(\cG) \subseteq \cC(\cG^+)$ extends to an inclusion $\Cstaress(\cG) \subseteq \Cstaress(\cG^+)$.  The universal property of the unitisation provides an injective extension to a *-homomorphism $\Cstaress(\cG)^+ \hra \Cstaress(\cG^+)$, which must also be surjective since it restricts to the *-isomorphism $\cC(\cG)^+ \cong \cC(\cG^+)$.
\end{proof}

Let us now translate the condition about the absence of essentially confined amenable sections of isotropy.
\begin{lemma}
  \label{lem:transfer-essentially-confined-sections}
  Let $\cG$ be an {\'e}tale groupoid with locally compact Hausdorff space of units.  Then $\cG$ has no essentially confined amenable sections of isotropy groups if and only if $\cG^+$ has no essentially confined amenable sections of isotropy groups.
\end{lemma}
\begin{proof}
  Assume that $\cG^+$ has no essentially confined amenable sections of isotropy groups and let $(\Lambda_x)_{x \in \mathfrak{X}}$ be an amenable section of isotropy groups of $\cG$. Then this section is also an amenable section of isotropy groups of $\cG^+$.  Put $Y = \ol{\cG^+ \cdot \{(x, \Lambda_x) \mid x \in \mathfrak{X} \}} \subseteq \Sub(\cG^+)$.  By assumption, for every $x \in (\Gnaught)^+$ there is $(x,H) \in Y$ such that $H \subseteq \ol{\Gnaught^+}_x$.  Since $\infty$ is $\cG^+$ invariant, this shows that $(\Lambda_x)_{x \in \mathfrak{X}}$ is not essentially confined.

  Assume that $\cG$ has no essentially confined amenable sections of isotropy groups and let $(\Lambda_x)_{x \in \mathfrak{X}^+}$ be an amenable section of isotropy groups of $\cG^+$.  Put $Y = \ol{\cG \cdot \{(x,\Lambda_x) \mid x \in \mathfrak{X}^+ \cap \Gnaught \}} \subseteq \Sub(\cG)$. By assumption, for every $x \in \Gnaught$, there is $(x,H) \in Y$ such that $H \subseteq \ol{\Gnaught}_x$.  Since $\cG^+_\infty = \{\infty\}$, this implies that $(\Lambda_x)_{x \in \mathfrak{X}^+}$ is not confined.
\end{proof}

We are now able to combine the discussion in this section with our results from Section~\ref{sec:compact-unit-space}.  This completes the proof of Theorem~\ref{thmintro:characterisation-confined-subgroups}.
\begin{theorem}
  \label{thm:characterisation-locally-compact-case}
      Let $\cG$ be an {\'e}tale groupoid with locally compact Hausdorff space of units.  Assume that $\cG$ is Hausdorff or $\cG^+$ is $\sigma$-compact. 
      In addition, let $\mathfrak{X}$ be the set of those points $x \in \Gnaught$ whose $\cG$-orbit $\cG.x$ is contained in the interior of its closure in $\Gnaught$, and suppose that $\mathfrak{X}$ is dense in $\Gnaught$. The last condition is in particular satisfied if $\cG$ is topologically transitive.
      
      Then $\Cstaress(\cG)$ has the ideal intersection property if and only if $\cG$ has no essentially confined amenable sections of isotropy groups.  Further, $\cG^+$ is $\sigma$-compact if $\cG$ is $\sigma$-compact.
\end{theorem}
\begin{proof}
 If $\cG$ is Hausdorff, then $\Gnaught \subseteq \cG$ is closed, so that also $(\cG^+)^{(0)} = \Gnaught \cup \{\infty\} \subseteq \cG^+$ is closed.  So also $\cG^+$ is Hausdorff.  If $\cG$ is $\sigma$-compact and $(K_n)_n$ is a sequence of compact subsets exhausting $\cG$, then $(K_n \cup \{\infty\})_n$ exhausts $\cG^+$.  So the Alexandrov groupoid is also $\sigma$-compact.

  Since $\Cstaress(\cG) \subseteq \Cstaress(\cG^+)$ is isomorphic with the inclusion $\Cstaress(\cG) \subseteq \Cstaress(\cG)^+$ by Proposition~\ref{prop:essential-groupoid-cstar-algebra-alexandrov-groupoid-unitisation}, it follows that $\conto(\Gnaught) \subseteq \Cstarred(\cG)$ has the ideal intersection property if and only if $\cont((\Gnaught)^+) \subseteq \Cstarred(\cG^+)$ has the ideal intersection property.  We can now combine Corollary~\ref{cor:seciso} and Lemma~\ref{lem:transfer-essentially-confined-sections} to conclude the proof.
\end{proof}

\begin{remark}
  \label{rem:translating-sigma-compactness}
  As stated in Theorem~\ref{thm:characterisation-locally-compact-case}, the Alexandrov groupoid $\cG^+$ is always $\sigma$-compact if $\cG$ is so.  The converse holds if $\{x \in \Gnaught \mid \cG_x = \{x\}\}$ is $\sigma$-compact, in particular if $\cG_x \neq \{x\}$ holds for every $x \in \Gnaught$.  Indeed, if $(K_n)_n$ is a sequence of compact subsets exhausting $\cG^+$, then the subsets $C_n = K_n \setminus (\Gnaught)^+ \subset \cG$ are compact and exhaust $\cG \setminus \Gnaught$.  Then $\Gnaught = \{x \in \Gnaught \mid \cG_x = \{x\}\} \cup \bigcup_n \rms(C_n)$, showing that $\cG$ is $\sigma$-compact.
\end{remark}

\begin{remark}
Theorem~\ref{thm:characterisation-locally-compact-case} leads to a characterisation of the ideal separation property (see for instance \cite{boenickeli2020}) for {\'e}tale Hausdorff groupoids $\cG$ with locally compact Hausdorff space of units in terms of sections of isotropy subgroups of restricted groupoids if the restriction of $\cG$ to every closed invariant subspace of $\Gnaught$ satisfies the conditions in Theorem~\ref{thm:seciso}, Corollary~\ref{cor:seciso} or Theorem~\ref{thm:characterisation-locally-compact-case}.
\end{remark}

Let us now give a proof of Theorem~\ref{thmintro:characterisation-confined-subgroups-minimal-case}, combining the results of this section.
\begin{theorem}
  \label{thm:characterisation-confined-subgroups-minimal-case}
  Let $\cG$ be an {\'e}tale groupoid with locally compact Hausdorff space of units.   Assume that $\cG$ is Hausdorff, $\cG$ is $\sigma$-compact or $\cG$ has a compact space of units.  Then the essential groupoid \Cstar-algebra $\Cstaress(\cG)$ is simple if and only if $\cG$ is minimal and has no essentially confined amenable sections of isotropy groups.
\end{theorem}
\begin{proof}
  In view of Theorems~\ref{thm:seciso} (see also Corollary~\ref{cor:seciso}) and \ref{thm:characterisation-locally-compact-case}, we only have to show that $\cG$ is minimal if $\Cstaress(\cG)$ is simple.  Assume that $\cG$ is not minimal and let $U \subseteq \Gnaught$ be a non-empty, proper, open $\cG$-invariant subset of its unit space.  Consider the ideal $I = \ol{\Cstaress(\cG) \conto(U) \Cstaress(\cG)}$ generated by $\conto(U)$ inside the essential groupoid \Cstar-algebra.  It is non-zero since $U$ is non-empty and it is proper since its preimage in $\Cstarred(G)$ is the proper ideal $\ol{\Cstarred(\cG) \conto(U) \Cstarred(\cG)}$. So $\Cstaress(\cG)$ is not simple.
\end{proof}

Let us next compare our Theorem~\ref{thm:characterisation-locally-compact-case} with the work of Kwa{\'s}niewski-Meyer \cite{kwasniewskimeyer2019-essential}.  In order to do so, we need to introduce their notion of topologically free groupoids, which is a strengthening of essentially effective groupoids as introduced in Definition~\ref{def:essentially-effective}, as we will observe below.
\begin{definition}[See {\cite[Definition 2.20]{kwasniewskimeyer2019-essential}}]
  \label{def:topologically-free}
  Let $\cG$ be an {\'e}tale groupoid with locally compact Hausdorff space of units.  Then $\cG$ is topologically free, if for every bisection $U \subset \cG \setminus \Gnaught$ the set $\{x \in \Gnaught \mid \cG_x^x \cap U \neq \emptyset\}$ has empty interior.
\end{definition}

\begin{remark}
  \label{rem:previous-work-locally-compact-case}

  Kwa{\'s}niewski-Meyer \cite[Theorem 7.29]{kwasniewskimeyer2019-essential} characterised {\'e}tale groupoids with locally compact Hausdorff space of units for which the kernel of $\Cstar(\cG) \to \Cstaress(\cG)$ is the unique maximal ideal of $\Cstar(\cG)$ that intersects $\conto(\Gnaught)$ trivially.  These are precisely the topologically free groupoids in the sense of Definition~\ref{def:topologically-free}. Let us explain why topologically free groupoids have no essentially confined sections of isotropy subgroups. Assume that $\cG$ is topologically free and observe that the bisection $U$ in Definition~\ref{def:topologically-free} may be assumed to be open and a subset of $\cG \setminus \ol{\Gnaught}$.  Fix a compact subset $C \subset \Gnaught$ with non-empty interior.  Given a section of isotropy groups $(\Lambda_x)_{x \in \Gnaught}$ (amenable or not), we will exhibit a net $(x_K)_K$ in $C$ indexed by compact subsets $K$ of $\cG \setminus \ol{\Gnaught}$ such that $\Lambda_{x_K} \cap K = \emptyset$.  Compactness of the Chabauty space $\Sub(\cG)$, which is a non-trivial fact for non-Hausdorff groupoids \cite{fell1962}, then implies that $(\Lambda_{x_K})_K$ has a cluster point, which is a subgroup of $\ol{\Gnaught}_x$ for some $x \in C$.  Let $K \subseteq \cG \setminus \ol{\Gnaught}$ be a compact subset and let $U_1, \dotsc, U_n$ be open bisections of $\cG$ covering $K$.  Then $\{x \in \Gnaught \mid \bigcup_{i = 1}^n U_i \cap \cG_x^x \neq \emptyset\}$ has empty interior.  In particular, there is $x \in C$ such that $\Lambda_x \cap K = \emptyset$.  We can put $x_K = x$ and finish the argument, showing that $(\Lambda_x)_x$ is not essentially confined.
\end{remark}

As explained in the introduction, characterisations of simplicity and a suitable generalised ideal intersection property of groupoid \Cstar-algebras were limited to the maximal groupoid \Cstar-algebra, as in \cite[Theorem 5.1]{brownclarkfarthingsims2014} and \cite[Theorem 7.29]{kwasniewskimeyer2019-essential}.  For amenable groupoids, this leads to characterisations of simplicity of the essential groupoid \Cstar-algebra.  The next corollary clarifies the relation of our Theorem~\ref{thm:characterisation-locally-compact-case} to such characterisations for amenable groupoids.
\begin{corollary}
  \label{cor:characterisation-locally-compact-case-amenable-isotropy}
  Let $\cG$ be a $\sigma$-compact {\'e}tale groupoid with locally compact Hausdorff space of units and amenable isotropy groups.  Then $\Cstaress(\cG)$ has the ideal intersection property if and only if $\cG$ is topologically free.
\end{corollary}
\begin{proof}
  If $\cG$ is topologically free, it follows from \cite[Theorem 7.29]{kwasniewskimeyer2019-essential} (see also Remark~\ref{rem:previous-work-locally-compact-case}) that $\Cstaress(\cG)$ has the ideal intersection property. Assume that $\cG$ is not topologically free.  Then there is a bisection $U \subseteq \cG \setminus \Gnaught$ such that $\{x \in \Gnaught \mid \cG_x^x \cap U \neq \emptyset\}$ has non-empty interior.  We may assume that $U$ is open and that $U \subset \Iso(\cG)$.  Putting $\Lambda_x = \cG_x^x$, we obtain an amenable section of isotropy groups of $\cG$.  We will show that it is confined.  For every $g \in \cG_x$, the equality $g \Lambda_x g^{-1} = \Lambda_{\rmr(g)}$ holds, so that $\ol{\cG \cdot \{\Lambda_x\}_x} = \ol{\{\Lambda_x\}_x}$. Let us denote this set by $Y$. Since for all $x \in \rms(U)$, we have $\Lambda_x \cap U \neq \emptyset$, also every element $(H,x) \in Y$ with $x \in \rms(U)$ satisfies $H \cap U \neq \emptyset$.  So $H \not \subseteq \ol{\Gnaught}$.  We can now apply Theorem~\ref{thm:characterisation-locally-compact-case} (noting that the condition on density of $\mathfrak{X}$ is not needed if $\cG$ has amenable isotropy groups, see also Remark~\ref{rem:AllStabAmenSit}) and conclude that $\Cstaress(\cG)$ does not have the ideal intersection property.  
\end{proof}

\section{Powers averaging for minimal groupoids}
\label{sec:powers-averaging}

In this section, we generalise the work done in \cite{amrutamursu2021}, and derive the Powers averaging property for simple essential groupoid \Cstar-algebras based on the notion of generalised probability measures.  We also prove the relative Powers averaging property for certain semigroups of generalised probability measures, leading to natural \Cstar-irreducible inclusions into essential groupoid \Cstar-algebras and applications to unitary representations in Section~\ref{sec:examples}.

\subsection{Generalised probability measures}
\label{sec:generalised-probability-measures}

In \cite{amrutamursu2021}, a notion of generalised probability measure on a compact space $X$ was introduced.  It combines the action of a group $G \grpaction{} X$ with the action of positive elements in $\cont(X)$ by elementary operators.  This idea naturally fits with our approach to the ideal intersection property for groupoid \Cstar-algebras by means of the pseudogroup of open bisections.

\begin{definition}
  \label{def:groupoid-convolver}
  Let $\cG$ be an {\'e}tale groupoid with compact Hausdorff space of units.  A (finite) generalised $\cG$-probability measure on $\Gnaught$ is a finite formal sum of pairs $\sum_{i \in I} (\gamma_i, f_i)$ such that $\gamma_i \in \Gamma(\cG)$ and $f_i \in \contc(\supp (\gamma_i), \RR_{\geq 0})$ is a positive, continuous, compactly supported function on the support of $\gamma_i$ satisfying $\sum_i f_i \circ \psi_{\gamma_i^*} = \mathbb{1}_{\Gnaught}$.  We denote the set of (finite) generalised $\cG$-probability measures on $\Gnaught$ by $\CP_\cG(\Gnaught)$.
\end{definition}
We point out that the concept of generalised probability measures allows for the repetition of bisections. Therefore, it is not necessarily true that the formal sum in the definition can be indexed by a subset of $\Gamma(\cG)$.

\begin{remark}
  \label{rem:rearranging-convolvers}
  In some situations it is preferable to think of a pair $(\gamma, f)$ as in Definition~\ref{def:groupoid-convolver} and being represented as $(f \circ \psi_{\gamma^*}, \gamma)$.  This is for example the case when interpreting the unitality condition $\sum_i f_i \circ \psi_{\gamma_i^*} = \mathbb{1}_{\Gnaught}$ and later in Definition~\ref{def:contractive-semigroup}, where covering semigroups of generalised probability measures are introduced.
\end{remark}

Generalised probability measures naturally act both on groupoid \Cstar-algebras and on their state spaces.  To prove this, we will first exhibit their natural semigroup structure.
\begin{proposition}
  Let $\cG$ be an {\'e}tale groupoid with compact Hausdorff space of units.  Given two generalised probability measures $\mu_1 = \sum_i (\gamma_{1,i}, f_{1,i})$ and $\mu_2 = \sum_j (\gamma_{2,j}, f_{2,j})$ from $\CP_\cG(\Gnaught)$, we define their product by the following formula.
  \begin{gather*}
    \mu_1 \mu_2
    =
    \sum_{i,j} (\gamma_{1,i} \gamma_{2,j}, (f_{1,i} \circ \psi_{\gamma_{2,j}}) f_{2,j})
    \eqstop
  \end{gather*}
  Then $\CP_\cG(\Gnaught)$ becomes a semigroup with this product.
\end{proposition}
\begin{proof}
  Once the product is well-defined, its associativity follows from associativity of the product in $\Gamma(\cG)$ and the fact that $\Gamma(\cG)$ acts on $\cont(\Gnaught)$.  Thus, we only have to prove that $\mu_1 \mu_2$ as defined above is again an element of $\CP_\cG(\Gnaught)$.

  First observe that for all $i,j$ the product $(f_{1,i} \circ \psi_{\gamma_{2,j}}) f_{2,j}$ is a continuous function whose support lies in $\psi_{\gamma_{2,j}^*}(\supp \gamma_{1,i} \cap \im \gamma_{2,j}) = \supp \gamma_{1,i}\gamma_{2,j}$.  Further,
  \begin{gather*}
    ( f_{1,i} \circ \psi_{\gamma_{2,j}}) f_{2,j}
    =
    ( f_{1,i} (f_{2,j} \circ \psi_{\gamma_{2,j}^*})) \circ \psi_{\gamma_{2,j}}
    \in
    \contc(\supp \gamma_{1,i} \gamma_{2,j})
    \eqstop
  \end{gather*}
  Evaluating the following sum in the space of Borel functions on $\Gnaught$, we find that
  \begin{align*}
    \sum_{i,j} ((f_{1,i} \circ \psi_{\gamma_{2,j}}) f_{2,j} ) \circ \psi_{(\gamma_{1,i}\gamma_{2,j})^*}
    & =
      \sum_i   f_{1, i} \circ \psi_{\gamma_{1,i}^*} \cdot \left ( \sum_j f_{2,j} \circ \psi_{\gamma_{2,j}^*} \right) \circ \psi_{\gamma_{1,i}^*} \\
    & =
      \sum_i (f_{1,i} \circ \psi_{\gamma_{1,i}^*}) \cdot \mathbb{1}_{\im \gamma_{1, i}} \\
    & =
      \mathbb{1}_{\Gnaught}
    \eqstop
  \end{align*}
  This finishes the proof of the proposition.
\end{proof}

Generalised probability measures were introduced in \cite{amrutamursu2021} to provide a notion of contractibility of measures on spaces admitting a minimal group action.  In order to generalise this to groupoids in the context of \Cstar-simplicity, we require suitable actions of $\CP_{\cG}(\Gnaught)$ on $\cont(\Gnaught)$ and on $\Cstaress(\cG)$.
\begin{proposition}
  \label{prop:action-generalised-probability-measures}
  Let $\cG$ be an {\'e}tale groupoid with compact Hausdorff space of units and let $(A, \alpha)$ be a unital $\cG$-\Cstar-algebra.  There is a semigroup action of $\CP_\cG(\Gnaught)$ on $A$ by completely positive, completely contractive and non-degenerate maps defined by
  \begin{gather*}
    \mu a = \sum_i \alpha_{\gamma_i}(f_i^{1/2} a f_i^{1/2})
  \end{gather*}
  for $a \in A$ and $\mu = \sum_i (\gamma_i, f_i)$. Consequently, there is a right semigroup action on the state space $\cS(A)$ of $A$ defined by
  \begin{gather*}
    (\vphi \mu)(a) = \vphi(\mu a)
  \end{gather*}
  for $\vphi \in \cS(A)$, $\mu \in \CP_\cG(\Gnaught)$ and $a \in A$.
\end{proposition}
\begin{proof}
  Let us first check that the formula in the statement defines an action.  Given two generalised $\cG$-probability measures $\mu_1 = \sum_i (\gamma_{1,i}, f_{1,i})$ and $\mu_2 = \sum_j (\gamma_{2,j}, f_{2,j})$ and an element $a \in A$, we find
  \begin{align*}
    \mu_1 (\mu_2 a) & =
                      \mu_1 (\sum_j \alpha_{\gamma_{2,j}} (f_{2,j}^{\frac{1}{2}} a f_{2,j}^{\frac{1}{2}})) \\
                    & =
                      \sum_i \alpha_{\gamma_{1,i}} ( f_{1,i}^{\frac{1}{2}} \sum_j \alpha_{\gamma_{2,j}} (f_{2,j}^{\frac{1}{2}} a f_{2,j}^{\frac{1}{2}}) f_{1,i}^{\frac{1}{2}}) \\
                    & =
                      \sum_i \alpha_{\gamma_{1,i}} (  \sum_j \alpha_{\gamma_{2,j}} ( (f_{1,i} \circ \psi_{\gamma_{2,j}})^{\frac{1}{2}} f_{2,j}^{\frac{1}{2}} a f_{2,j}^{\frac{1}{2}} (f_{1,i} \circ \psi_{\gamma_{2,j}})^{\frac{1}{2}})) \\
                    & =
                      (\mu_1\mu_2)(a)
                      \eqstop
  \end{align*}
For every $\mu \in \CP_\cG(\Gnaught)$ the map $a \mapsto \mu a$  is completely positive as a sum of a composition of completely positive maps.  Further, it is unital on $A$, since
  \begin{gather*}
    \mu(1_A)
    =
    \sum_i \alpha_{\gamma_i} (f_i^{\frac{1}{2}} 1_A f_i^{\frac{1}{2}})
    =
    \sum_i \alpha_{\gamma_i} (f_i^{\frac{1}{2}} f_i^{\frac{1}{2}})
    =
    \sum_i \alpha_{\gamma_i}(f_i)
    =
    \mathbb{1}_{\Gnaught}
    =
    1_A
    \eqstop
  \end{gather*}
  This shows that $a \mapsto \mu a$ is completely contractive and non-degenerate on $A$.  It follows directly that $\vphi \mu(a) = \vphi(\mu a)$ defines an action on the state space of $\cS(A)$.
\end{proof}

\subsection{Contractive and covering semigroups}
\label{sec:contractive-covering-semigroups}

When we consider Powers averaging property, we will require the following definition, which provides a suitable generalisation of boundary actions of groups.

\begin{definition}
  \label{def:contractive-semigroup}
  Let $\cG$ be an {\'e}tale groupoid with compact Hausdorff space of units.  We will say that a subsemigroup $S \subseteq \CP_\cG(\Gnaught)$ is covering if for every $g \in \cG$ there is $\mu = \sum_i (\gamma_i, f_i) \in S$ such that $f_i \circ \psi_{\gamma_i^*}(\rmr(g)) \neq 0$ implies $g \in \gamma_i$.  If $\cG$ is minimal, we will say that $S$ is contractive if for any $\nu \in \cP(\Gnaught)$ and any $x \in \Gnaught$ there is a net $(\mu_i)_i$ in $S$ such that $\nu \mu_i \xrightarrow{\weakstar} \delta_x$.
\end{definition}

The following result was established for the special case of crossed products by minimal group actions in \cite[Lemma 3.6]{amrutamursu2021}.
\begin{proposition}
  \label{prop:generalised-probabilities-covering-and-contractive}
  Let $\cG$ be an {\'e}tale groupoid with compact Hausdorff space of units.  Then $\CP_\cG(\Gnaught)$ is a covering semigroup.  If $\cG$ is minimal, then $\CP_\cG(\Gnaught)$ is contractive.
\end{proposition}
\begin{proof}
  Let us first prove that $\CP_\cG(\Gnaught)$ is covering.  To this end take $g \in \cG$, let $\gamma$ be an open bisection of $\cG$ containing $g$ and let $f \in \contc(\im \gamma)$ with $0 \leq f \leq 1$ and $f(\rmr(g)) = 1$.  Denoting by $e \in G$ the neutral element, define
  \begin{gather*}
    \mu_g
    =
    (\gamma, f \circ \psi_\gamma) + (e, 1- f)
    \eqstop
  \end{gather*}
  Then $\mu_g$ satisfies the conditions of Definition~\ref{def:contractive-semigroup} and hence witnesses that $\CP_\cG(\Gnaught)$ is covering.

  Assume that $\cG$ is minimal.  Fix any $x \in \Gnaught$ and an open neighbourhood $V$ of $x$.  We will find $\mu_V \in \CP_\cG(\Gnaught)$ such that $\nu \mu_V$ is supported in $\ol{V}$ for all $\nu \in \cP(\Gnaught)$.  Given any $y \in \Gnaught$, by minimality there exists some $\gamma_y \in \Gamma(\cG)$ with $y \in \supp \gamma_y$ and $\psi_{\gamma_y}(y) \in V$.  In other words, $y \in \psi_{\gamma_y^*}(V) $.  It follows that the family $(\psi_{\gamma^*}(V))_{\gamma \in \Gamma(\cG)}$ is an open cover of $\Gnaught$.  So by compactness, there is a finite number of open bisections $\gamma_1, \dotsc, \gamma_n$ such that $(\psi_{\gamma_i^*}(V))_{i = 1}^n$ covers $\Gnaught$. Let $g_i$ be a partition of unity subordinate to this open cover and put $f_i = g_i \circ \psi_{\gamma_i}$ as well as $\mu_V = \sum_{i=1}^n (\gamma_i, f_i)$.  Then $\mu_V$ is a generalised probability measure.
  
  Let $\nu \in \cP(\Gnaught)$ and let $h \in \cont(\Gnaught)$ be such that $h \geq 0$ and $h|_{\ol{V}} = 0$. We have
  \begin{gather*}
    (\nu \mu_V)(h)
    =
    \nu(\mu_V h)
    =
    \sum_i \nu(\alpha_{\gamma_i}( f_i^{\frac{1}{2}} h f_i^{\frac{1}{2}}))
    =
    0
    \eqcomma
  \end{gather*}
  since all functions $f_i$ are supported in $V$.
\end{proof}

Recall from Section \ref{sec:groupoid-spaces} the action of $\cG$ on states in $\cS_{\Gnaught}(A)$ for a $\cG$-\Cstar-algebra $A$.  That is, if $\vphi \in \cS(A)$ satisfies $\vphi|_{\cont(\Gnaught)} = \ev_x$, then $g \vphi = \vphi(\alpha_{\gamma^*}(f \cdot f))$ for any $g \in \cG_x$, any open bisection $\gamma$ containing $g$ and any $f \in \contc(\im \gamma)$ satisfying $f(\rmr(g)) = 1$.  The next lemmas shows how covering semigroups can be used to implement this action.
\begin{lemma}
  \label{lem:semigroup_translation_is_generalized_measure}
  Let $\cG$ be an {\'e}tale groupoid with compact Hausdorff space of units and let $(A, \alpha)$ be a unital $\cG$-\Cstar-algebra.  Assume that $S \subseteq \CP_\cG(\Gnaught)$ is a semigroup covering $\cG$ and let $\vphi \in \cS(A)$ be a state satisfying $\vphi|_{\cont(\Gnaught)} = \ev_x$ for some $x \in \Gnaught$.  Then for every $g \in \cG_x$ there is $\mu_g \in S$ such that $\vphi \mu_g = g \vphi$.
\end{lemma}
\begin{proof}
  Since $S$ is covering, we can choose $\mu_g = \sum_i (\gamma_i^*, f_i) \in S$ such that $f_i \circ \psi_{\gamma_i}(\rms(g)) \neq 0$ implies $g \in \gamma_i$.  The intersection of all $\gamma_i$ containing $g$ is an open bisection $\gamma$ containing $g$.  Since all functions $f_i$ have compact support and $\Gnaught$ is Hausdorff, we can choose $h \in \contc(\im \gamma)$ with $0 \leq h \leq 1$ and $h(x) = 1$ such that $\supp h \cap \supp f_i \circ \psi_{\gamma_i} = \emptyset$ whenever $g \notin \gamma_i$.  Then we can calculate
  \begin{align*}
    \vphi \mu_g
    & =
    \sum_i \vphi(\alpha_{\gamma_i^*}(f_i^{\frac{1}{2}} \cdot f_i^{\frac{1}{2}})) \\
    & =
      \sum_i h(x) \vphi( \alpha_{\gamma_i^*}(f_i^{\frac{1}{2}} \cdot f_i^{\frac{1}{2}})) \\
    & =
      \sum_i \vphi( h \alpha_{\gamma_i^*}(f_i^{\frac{1}{2}} \cdot f_i^{\frac{1}{2}})) \\
    & =
      \sum_i \vphi(\alpha_{\gamma^*}((\alpha_{\gamma}(h) f_i)^{\frac{1}{2}} \cdot (\alpha_{\gamma^*}(h) f_i)^{\frac{1}{2}})) \\
    & =
      g \vphi \cdot \left ( \sum_i (\alpha_{\gamma}(h) f_i)(\rmr(g)) \right )  \\
    & =
      g \vphi \cdot \left ( \sum_i (h(x) f_i \circ \psi_{\gamma_i}(\rms(g)) \right )  \\
    & =
      g \vphi
      \eqstop
  \end{align*}
  This finishes the proof.
\end{proof}

We will need to know that the contractivity of certain semigroups of generalised probability measures is preserved by passing to the Furstenberg groupoid.
\begin{lemma}
  \label{lem:contractivity-furstenberg-boundary}
  Let $\cG$ be a minimal {\'e}tale groupoid with compact Hausdorff space of units and let $S \subseteq \CP_\cG(\Gnaught)$ be a contractive and covering semigroup.  Denote by $\pi: \fb \cG \to \Gnaught$ the projection map.  For $\mu = \sum_i (\gamma_i, f_i) \in \CP_\cG(\Gnaught)$ write $\pi^*(\mu) = \sum_i (\pi^{-1}(\gamma_i), f_i \circ \pi) \in \CP_{\cG \ltimes \fb \cG}(\fb \cG)$.  Then $\pi^*(S) \subseteq \CP_{\cG \ltimes \fb \cG}(\fb \cG)$ is contractive and covering.
\end{lemma}
\begin{proof}
  It is clear that $\pi^*(S)$ is covering, so we have to show that it is contractive.  Note that considering $\cont(\fb \cG)$ as a $\cG$-\Cstar-algebra, we have $\nu \pi^*(\mu) = \nu \cdot \mu$, for all $\nu \in \cP(\fb \cG)$ and all $\mu \in \CP_{\cG}(\Gnaught)$.  We will use the simple notation $\nu \mu$.  Let $\nu \in \cP(\fb \cG)$ and $y \in \fb \cG$.  Put $x = \pi(y) \in \Gnaught$.  Since $S$ is contractive, we can find a net $(\mu_i)_i$ in $S$ such that $\pi_*(\nu)\mu_i \to \delta_x$.  Passing to a subnet, we may assume that $\nu \mu_i \to \tilde \nu$ for some $\tilde \nu \in \cP(\fb \cG)$.  Then $\pi_*(\tilde \nu) = \delta_x$ holds.  Since $\cG$ is minimal, Proposition~\ref{prop:strongly-proximal-minimal} and strong proximality of $\fb \cG$ imply that there is a net $(g_i)_i$ in $\cG_x$ such that $g_i \tilde \nu \to \delta_y$ for some $y \in \pi^{-1}(x)$.  Lemma~\ref{lem:semigroup_translation_is_generalized_measure} shows that there is a net $(\tilde \mu_i)_i$ in $S$ such that $\tilde \nu \tilde \mu_i = g_i \tilde \nu  \to \delta_y$.  Summarising, we found that $\delta_y \in \ol{ \nu \cdot S} = \ol{\nu \cdot \pi^*(S) }$, which proves contractivity of $\pi^*(S)$.
\end{proof}

\subsection{Powers averaging}
\label{sec:powers-averaging-subsection}

We are ready to establish Powers averaging property for simple essential groupoid \Cstar-algebras.  We begin by considering the state space of $\Cstaress(\cG)$.  To put the statement of the next proposition into context, let us recall that the natural conditional expectation $\Eess: \Cstaress(\cG) \to \Dix(\Gnaught)$ does not take values in $\cont(\Gnaught)$ if $\cG$ is not Hausdorff.  Also recall from Theorem~\ref{thm:characterisation-confined-subgroups-minimal-case} that if $\Cstaress(\cG)$ is simple, then $\cG$ is necessarily minimal, so that Theorem~\ref{thm:inclusion-essential-algebras} implies there is a natural inclusion $\Cstaress(\cG) \subseteq \Cstaress(\cG \ltimes \fb \cG)$.  In the remainder of this section, we let $\rE: \Cstaress(\cG \ltimes \fb \cG) \to \cont(\fb \cG)$ denote the conditional expectation described in Proposition~\ref{prop:conditional-expectation-extremally-disconnected}.

In order to formulate averaging results conveniently, it is useful to introduce a notion of convex combinations of generalised probability measures.

\begin{definition}
  \label{def:convex-generalised-probabilities}
  Let $\cG$ be an {\'e}tale groupoid with compact Hausdorff space of units.  For generalised probability measures indexed over disjoint sets $I$ and $J$, a formal convex combination is an expression of the form
  \begin{gather*}
    c \sum_{i \in I} (\gamma_i, f_i) + (1 - c)\sum_{j \in J} (\gamma_j, f_j)
    =
    \sum_{i \in I \sqcup J} (\gamma_i, \mathbb{1}_I(i) c f_i + \mathbb{1}_J(i) (1 - c) f_i).
    \eqcomma
  \end{gather*}
  We call a subset of $\CP_\cG(\Gnaught)$ convex if it is closed under convex combinations.
\end{definition}

\begin{proposition}
  \label{prop:average_state_to_trivial_boundary}
  Let $\cG$ be an {\'e}tale groupoid with compact Hausdorff space of units.  Let $S \subseteq \CP_\cG(\Gnaught)$ be a contractive and covering convex semigroup.  Assume that $\Cstaress(\cG)$ is simple.  Then given any $\vphi \in \cS(\Cstaress(\cG \ltimes \fb \cG))$, we have
  \begin{gather*}
    \{ \nu \circ \rE \mid \nu \in \cP(\fb \cG) \}
    \subseteq
    \ol{\{\vphi \mu \mid \mu \in S\}}^{\weakstar}
  \end{gather*}
\end{proposition}
\begin{proof}
  Write $K = \ol{\{\vphi \mu \mid \mu \in S\}}^{\weakstar}$ and observe that $K$ is weak-* closed and convex.  Hence, it suffices to prove that $\ev_y \circ \rE \in K$ for all $y \in \fb \cG$.  Fix $y \in \fb \cG$ and let $x = \pi(y) \in \Gnaught$.  By contractivity of $S$, there is $\sigma \in K$ satisfying $\sigma|_{\cont(\Gnaught)} = \ev_x$.  Since $S$ is both contractive and covering, Lemma~\ref{lem:contractivity-furstenberg-boundary} shows that there is a net $(\mu_i)_i$ in $S$ such that $\sigma|_{\cont(\fb \cG)} \mu_i \to \ev_y$.  Dropping to a subnet, we may assume that $\sigma \mu_i \to \omega \in \cS(\Cstaress(\cG \ltimes \fb \cG))$.   Observe that $\omega|_{\Cstaress(\cG)} \in K$.

 As observed in the proof of Theorem~\ref{thm:characterisation-confined-subgroups-minimal-case}, simplicity of $\Cstaress(\cG)$ implies that $\cG$ is minimal. Hence all stabilizer subgroups in $\cG \ltimes \fb \cG$ are amenable by Corollary~\ref{cor:MinAllStabAmen}. Thus Theorem~\ref{thm:essential-is-reduced-cstar-algebra} applies and allows us to identify $\Cstaress(\cG \ltimes \fb \cG)$ with $\Cstarred(\cH)$ for the Hausdorffification $\cH = (\cG \ltimes \fb \cG)_\Haus$.  Then $\rE$ is identified with the natural conditional expectation of $\Cstarred(\cH)$.  We use the fact that $\Cstarred(\cH)$ is simple to infer by Theorem~\ref{thm:intersection-property-essential-algebras} (see also Corollary~\ref{cor:IntProp_TopTrans}) that $\cH$ is principal.  So it suffices to show that $\omega(u_\gamma f) = 0$ for all open bisections $\gamma \subseteq \cH \setminus \fb \cG$ and all positive functions $f \in \contc(\supp \gamma)$.  If $y \notin \supp \gamma$, then we see that
  \begin{gather*}
    \omega(u_\gamma f)
    =
    \omega(u_\gamma f^{\frac{1}{2}}) \omega(f^{\frac{1}{2}})
    =
    \omega(u_\gamma f^{\frac{1}{2}}) f^{\frac{1}{2}}(y)
    =
    0
    \eqstop
  \end{gather*}
  If $y \in \supp \gamma$, the fact that $\psi_{\gamma}(y) \neq y$ allows us to choose a function $h \in \contc(\supp \gamma)$ such that $h(y) = 1$ and $h(\psi_\gamma(y)) = 0$. It follows that
  \begin{gather*}
    \omega(u_\gamma f)
    =
    h(y) \omega(u_\gamma f)
    =
    \omega(h u_\gamma f)
    =
    \omega(u_\gamma f (\alpha_{\gamma^*}(h)))
    =
    \omega(u_\gamma f) h(\psi_{\gamma}(y))
    =
    0
    \eqstop
  \end{gather*}
  This finishes the proof.
\end{proof}

We now extend the previous result to the entire dual space of $\Cstaress(\cG)$.
\begin{corollary}
  \label{cor:average_functional_to_trivial_boundary}
  Let $\cG$ be an {\'e}tale groupoid with compact Hausdorff space of units. Assume that $\Cstaress(\cG)$ is simple, let $S \subseteq \CP_\cG(\Gnaught)$ be a contractive and covering convex semigroup and let $\omega \in \Cstaress(\cG \ltimes \fb \cG)^*$. Then
  \begin{gather*}
    \{ \omega(1) \nu \circ E \mid \nu \in P(\fb \cG)\}
    \subseteq
    \ol{\{\omega \mu \mid \mu \in S\}}^{\weakstar}
    \eqstop
  \end{gather*}
\end{corollary}
\begin{proof}
  We write $K = \ol{\{\omega \mu \mid \mu \in S\}}^{\weakstar}$.  Since $K$ is convex and weak-* closed, it suffices to shows that $\omega(1) \ev_y \circ \rE \in K$ for any $y \in \fb \cG$.  We decompose $\omega = \sum_{k=1}^4 c_k \vphi_k$ as a convex combination with four states $\vphi_i \in \cS(\Cstaress(\cG \ltimes \fb \cG))$.  By Proposition~\ref{prop:average_state_to_trivial_boundary}, we may find a net $(\mu_i)$ in $S$ with $\vphi_1 \mu_i \xrightarrow{\weakstar} \nu_1 \circ \rE$ for some $\nu_1 \in \cP(\fb \cG)$.  Dropping to a subset, we may assume that $\omega \mu_i \xrightarrow{\weakstar} \nu_1 \circ \rE + \sum_{k=2}^4 c_k \vphi_k'$ for some new states $\vphi_k' \in \cS(\Cstaress(\cG))$.  Repeating this process three more times, and noting that the set $\{ \nu \circ \rE  \mid  \nu \in \cP(\fb \cG)\}$ is invariant under the right action of $\CP_\cG(\Gnaught)$, we see that there is some element of $K$ of the form $\omega(1) \nu \circ \rE$ for some finite complex measure $\nu$ on $\fb \cG$.  Thanks to Lemma~\ref{lem:contractivity-furstenberg-boundary} we find that $\omega(1) \ev_y \circ \rE \in K$ for every $y \in \fb \cG$.
\end{proof}

We are now able to dualise Corollary~\ref{cor:average_functional_to_trivial_boundary} in order to obtain a version of Powers averaging property for essential groupoid \Cstar-algebras.  The next definition provides a notion of Powers averaging that subsumes the classical notion for groups and dynamical systems.
\begin{definition}
  \label{def:powers-averaging}
  Let $\cG$ be an {\'e}tale groupoid with compact Hausdorff space of units and let $S \subset \CP_\cG(\Gnaught)$ be a subsemigroup.  We say that $\Cstaress(\cG)$ satisfies relative Powers averaging property with respect to $S$ if $0 \in \cconv\{\mu a \mid \mu \in S\}$ for all $a \in \Cstaress(\cG)$ satisfying $\Eess(a) = 0$.
\end{definition}

Recall that in \cite{rordam2021} a unital inclusion of \Cstar-algebras $A \subseteq B$ is said to be \Cstar-irreducible if every intermediate \Cstar-algebra is simple.  This notion will be linked to the relative Powers averaging property by means the following definition, relating generalised probability measures to \Cstar-subalgebras of $\Cstaress(\cG)$.
\begin{definition}
  \label{def:algebra-supporting-generalised-probability-measures}
  Let $\cG$ be an {\'e}tale groupoid with compact Hausdorff space of units and let $\mu = \sum_{i = 1}^n (\gamma_i, f_i) \in \CP_\cG(\Gnaught)$ be a generalised probability measure.  Let $A \subseteq \Cstaress(\cG)$ be a \mbox{\Cstar-subalgebra}. Then $A$ is said to \emph{support} $\mu$, if $u_{\gamma_i} f_i \in A$ for all $i \in \{1, \dotsc, n\}$.  Further, if $S \subseteq \CP_\cG(\Gnaught)$ is a subsemigroup, then $A$ is said to support $S$ if it supports every element of $S$.
\end{definition}
We observe that $\Cstaress(\cG)$ supports $\CP_\cG(\Gnaught)$.

We are now ready to state our main result on Powers averaging for essential groupoid \Cstar-algebras.    
\begin{theorem}
  \label{thm:powers-averaging}
  Let $\cG$ be a minimal {\'e}tale groupoid with compact Hausdorff space of units.  Then the following statements are equivalent.
  \begin{enumerate}
  \item \label{item:simplicity}
    $\Cstaress(\cG)$ is simple.
  \item \label{item:zero-expectation}
    $\Cstaress(\cG \ltimes \fb \cG)$ satisfies relative Powers averaging property with respect to any covering and contractive semigroup of generalised probability measures on $\cG$.
  \item \label{item:integral}
    Given any $a \in \Cstaress(\cG \ltimes \fb \cG)$ and any $\nu \in \cP(\fb \cG)$, we have $\nu \circ \rE(a) \in \cconv\{ \mu a \mid \mu \in S\}$ for any covering and contractive semigroup $S$ of generalised probability measures on $\cG$.
  \item \label{itm:irreducibility-groupoid-algebra}
    $A \subseteq \Cstaress(\cG)$ is \Cstar-irreducible for every \Cstar-subalgebra $A$ supporting a covering and contractive semigroup of generalised probability measures on $\cG$.
  \item \label{itm:irreducibility}
    $A \subseteq \Cstaress(\cG \ltimes \fb \cG)$ is \Cstar-irreducible for every \Cstar-subalgebra $A$ supporting a covering and contractive semigroup of generalised probability measures on $\cG$.
  \end{enumerate}
  If $\cG$ is Hausdorff, then all these conditions are equivalent to the following statement.
  \begin{enumerate}
    \setcounter{enumi}{5}
  \item \label{item:expectation}
    Given any $a \in \Cstarred(\cG)$, we have $\rE(a) \in \cconv\{ \mu a \mid \mu \in S\}$ for every covering and contractive semigroup $S$ of generalised probability measures on $\cG$.
  \end{enumerate}
\end{theorem}

\begin{remark}
  The assumptions on $S$ being covering and contractive are satisfied for $S = \CP_\cG(\Gnaught)$ by Proposition~\ref{prop:generalised-probabilities-covering-and-contractive}, since $\cG$ is assumed to be minimal.
\end{remark}

\begin{proof}[Proof of \ref{thm:powers-averaging}]
  The implication from~\ref{itm:irreducibility} to~\ref{itm:irreducibility-groupoid-algebra} is clear.  It is also clear that \ref{itm:irreducibility-groupoid-algebra} implies~\ref{item:simplicity}.

For the rest of the proof we fix a covering and contractive semigroup $S$ of generalised probability probability measures on $\cG$.  Without loss of generality, we may assume that $S$ is convex.
  
  Let us show that~\ref{item:simplicity} implies~\ref{item:integral}.  For a contradiction, assume there were some $\nu \in \cP(\fb \cG)$ and some $a \in \Cstaress(\cG \ltimes \fb \cG)$ for which the conclusion of~\ref{item:integral} does not hold.  By the Hahn-Banach separation theorem, there is some functional $\omega \in \Cstaress(\cG \ltimes \fb \cG)^*$ and some $\alpha \in \RR$ with
  \begin{gather*}
    \Re \omega(1) \nu(\rE(a)) < \alpha \leq \Re \omega(\mu a)
  \end{gather*}
  for all $\mu \in S$. This contradicts Corollary~\ref{cor:average_functional_to_trivial_boundary}.

  Let us next assume that~\ref{item:integral} holds.  Given $a \in \Cstaress(\cG \ltimes \fb \cG)$ such that $\rE(a) = 0$, we have $0 = \nu(\rE(a)) \in \ol{\{\mu a \mid \mu \in S\}}$ for any auxiliary $\nu \in \cP(\fb \cG)$.  This proves~\ref{item:zero-expectation}.

  We will now show that~\ref{item:zero-expectation} implies~\ref{itm:irreducibility}.  Assume that $I \subseteq \Cstaress(\cG \ltimes \fb \cG)$ is a nonzero \Cstar-subalgebra which is invariant under multiplication with elements from $A$.  Let $a \in I$ be nonzero. Replacing $a$ by $a^*a$, we may assume without loss of generality that $a$ is positive. Consider the nonzero positive function $f = \rE(a) \in \cont(\fb \cG)$.  Fix $x \in \fb \cG$ such that $f(x) \neq 0$.  Denote by $\pi: \cG \ltimes \fb \cG \to \cG$ the natural projection.  Since $S$ is covering $\cG$, it also covers $\cG \ltimes \fb \cG$ by Lemma~\ref{lem:contractivity-furstenberg-boundary}.  So for every $g \in (\cG \ltimes \fb \cG)_x$ there is some $\mu = \sum_i (\gamma_i, f_i)$ such that $f_i \circ \psi_{\gamma_i^*}(\rmr(g)) \neq 0$ implies $g \in \gamma_i$.  Hence,
  \begin{gather*}
    \mu f (\rmr(g))
    =
    \sum_i \alpha_{\gamma_i}(f_i^{\frac{1}{2}} f f_i^{\frac{1}{2}})(\rmr(g))
    =
    \sum_i (f_i f )(\psi_{\gamma_i^*}(\rmr(g)))
    =
    \sum_i (f_i f )(x)
    =
    f(x)
    \neq
    0
    \eqstop
  \end{gather*}
  So $gx \in \supp \mu f$ follows.  By minimality of $\cG \ltimes \fb \cG$ and compactness of $\fb \cG$ there are finitely many elements $\mu_1, \dotsc, \mu_n \in S$ such that $\frac{1}{n} \sum_i \mu_i f$ is nowhere zero.  By compactness of $\fb \cG$, there is $\delta > 0$ such that $\frac{1}{n} \sum_i \mu_i f \geq \delta$.    Since $\rE$ is a $\cG$-expectation, it follows that $\rE(\frac{1}{n} \sum_i \mu_i a) =  \frac{1}{n} \sum_i \mu_i \rE(a) \geq \delta$.  So we found a positive element $b \in I$ such that $\rE(b) \geq \delta > 0$.  Then for arbitrary $\mu \in S$, we have
  \begin{gather*}
    \mu b = \mu(b - E(b)) + \mu(E(b)) \geq \mu(b - E(b)) + \delta
    \eqstop
  \end{gather*}
  Since $S$ is convex, our assumption allows to choose $\mu \in S$ such that $\|\mu(b - E(b))\| \leq \frac{\delta}{2}$.  Then we infer that $\mu b \geq \frac{\delta}{2}$, and hence $\mu b \in I$ is invertible.  This shows that $I = \Cstaress(\cG \ltimes \fb \cG)$, and thus \ref{itm:irreducibility}.

  Assuming additionally that $\cG$ is Hausdorff, we observe that the proof that \ref{item:expectation} implies~\ref{item:simplicity} follows from an obvious simplification of the argument that~\ref{item:zero-expectation} implies~\ref{itm:irreducibility}, making use of the fact that $\Ered(a) \in \cont(\Gnaught)$ holds for all $a \in \Cstarred(\cG)$.  The implication from~\ref{item:integral} to~\ref{item:expectation} follows from contractivity of $S$, since every function $f \in \cont(\Gnaught)$ lies in the $S$ closure the functions $f(x) \mathbb{1}_{\Gnaught}$.
\end{proof}

\section{From boundary actions to Powers averaging}
\label{sec:examples}

In this section, we will apply Theorem~\ref{thm:powers-averaging} in order to obtain concrete examples of unitary group representations into \Cstar-algebras satisfying relative Powers averaging property.  Let us recall that for a discrete group $G$, a $G$-boundary is a compact Hausdorff space with a minimal and strongly proximal action of $G$.  Also recall that the topological full group $\mathbf{F}(\cG)$ of a groupoid is the group of its global bisections.  We can now formulate the following corollary of Theorem~\ref{thm:powers-averaging}.
\begin{corollary}
  \label{cor:contracting-groups-implies-irreducibility}
  Let $\cG$ be an {\'e}tale groupoid with compact Hausdorff space of units.  Assume that there is a subgroup of the topological full group $G \leq \mathbf{F}(\cG)$ that covers $\cG$ and such that $G \grpaction{} \Gnaught$ is a $G$-boundary.  Denote by $\pi: G \to \Cstaress(\cG)$ the unitary representation of $G$ in the essential groupoid \Cstar-algebra of $\cG$.  If $\Cstaress(\cG)$ is simple, then $\Cstaress(\cG)$ satisfies Powers averaging relative to $\pi(G)$.
\end{corollary}
\begin{proof}
   We define $S$ to be the convex subsemigroup generated by $G$ inside $\CP_\cG(\Gnaught)$.  Then it suffices to note that $S$ is contractive and covering to apply Theorem~\ref{thm:powers-averaging}.
\end{proof}

Recall that the groupoid of germs associated with an action of a discrete group $G \grpaction{} X$ is the quotient $(G \ltimes X) / \Iso(G \ltimes X)^\circ$.  We will apply Corollary~\ref{cor:contracting-groups-implies-irreducibility} to groupoids of germs, thereby obtaining concrete examples of groupoid \Cstar-algebras satisfying relative Power averaging with respect to a natural group of unitaries.  \Cstar-irreducibility of the associated inclusions has been studied in \cite{kalantarscarparo2021}.
\begin{theorem}
  \label{thm:relative-powers-averaging-boundary-action}
  Let $G$ be a countable discrete group and $G \grpaction{} X$ a boundary action.  Denote by $\cG$ its groupoid of germs and by $\pi: G \to \Cstaress(\cG)$ the associated unitary representation.  Then $\pi(G) \subseteq \Cstaress(\cG)$ satisfies the relative Powers averaging property.
\end{theorem}
\begin{proof}
  Let us first show that $\Cstaress(\cG)$ is simple.  To this end, we observe that $\cG$ is minimal, since it has the same orbits as $G \grpaction{} X$.  Further, since $G$ is countable, the set $\bigcup_{g \in G} \partial \Fix(g)$ is meager in $X$. Hence, its complement is dense in $X$, and it follows that $\cG$ is topologically principal.  It follows from Theorem~\ref{thmintro:characterisation-confined-subgroups} (and already from \cite[Theorem 7.26]{kwasniewskimeyer2019-essential}) that $\Cstaress(\cG)$ is simple.   By Corollary~\ref{cor:contracting-groups-implies-irreducibility} applied to $G/ \ker(G \grpaction{} X)$, it now follows that $\pi(G) \subseteq \Cstaress(\cG)$ satisfies the relative Powers averaging property.
\end{proof}

\begin{remark}
  \label{rem:identifying-representation}
  The representation of $\pi$ appearing in Theorem~\ref{thm:relative-powers-averaging-boundary-action} can be identified with a quasi-regular representation as employed in \cite{kalantarscarparo2020,kalantarscarparo2021}.  To this end recall the following notation, given an action of a discrete group $G \grpaction{} X$.  The open stabiliser at $x \in X$ is defined as
  \begin{gather*}
    G_x^\circ = \{g \in G \mid \text{ there is a neighbourhood } x \in U \text{ such that } g|_U = \id\}
    \eqcomma
  \end{gather*}
  and, for $g \in G$, we denote by $\Fix(g) = \{x \in X \mid gx = x\}$ the fixed point set of $g$.

  Now if $x \in X$ is such that
  \begin{itemize}
  \item the subquotient $G_x/G_x^0$ is amenable, and
  \item $x \notin \partial (\Fix(g)^\circ)$ for every $g \in G$,
  \end{itemize}
  then the inclusion $\pi(G) \subseteq \Cstaress(\cG)$ is isomorphic with an inclusion $\lambda_{G/G_x}(G) \subseteq \Cstaress(\cG)$ arising from the regular representation of $\cG$ associated with $x$.

  Indeed, consider the regular representation $\lambda_x$ of $\Cstarred(\cG)$ on $\ltwo(\cG_x)$.  We recall from Section~\ref{sec:groupoids} that if $a \in \Cstarred(\cG)$ is singular, then $\rms(\supp \hat a) \subseteq \rms(\ol{\Gnaught} \setminus \Gnaught) = \bigcup_{g \in G} \partial (\Fix(g)^\circ)$. So by the assumption on $x$, it follows that $\lambda_x$ factors through a *-homomorphism $\Cstaress(\cG) \to \bo(\ltwo(\cG_x))$.  We observe that $\cG_x = G/G_x^0$ and the associated representation of $G$ on $\ltwo(\cG_x) \cong \ltwo(G / G_x^0)$ is the quasi-regular representation $\lambda_{G, G_x^0}$.  Further, since $\cG_x^x = G_x / G_x^0$ is amenable, the *-homomorphism above factors to a *-homomorphism $\Cstaress(\cG) \to \bo(\ltwo(\cG_x / \cG_x^x))$.  Observe that this map is injective, because $\Cstaress(\cG)$ is simple. The associated representation of $G$ on $\ltwo(\cG_x / \cG_x^x) \cong \ltwo(G / G_x)$ is the quasi-regular representation $\lambda_{G, G_x}$.
\end{remark}

\begin{remark}
  \label{rem:recovering-previous-work-irreducible}
  Previous results about relative Powers averaging arising from group actions were obtained by Amrutam-Kalantar \cite{amrutamkalantar2020}, who showed that if $G$ is a \Cstar-simple discrete group and $X$ is any minimal, compact $G$-space, then the inclusion $G \subseteq \cont(X) \redtimes G = \Cstarred(X \rtimes G)$ satisfies relative Powers averaging.  In these examples, \Cstar-simplicity of the group is the source of Powers averaging.  In contrast, our Theorem~\ref{thm:relative-powers-averaging-boundary-action} makes no assumption on the acting group $G$, but relies instead on the assumption that $X$ is a $G$-boundary.  Nevertheless, there are some situations where the two results overlap such as hyperbolic groups $G$ with trivial finite radical acting on their Gromov boundary $\partial G$.  Since $\partial G$ is a topologically free boundary action of a \Cstar-simple group, both \cite[Theorem 1.3]{amrutamkalantar2020} and our Theorem~\ref{thm:relative-powers-averaging-boundary-action} imply the relative Powers averaging property for the inclusion $G \subseteq \Cstarred(\partial G \rtimes G)$.
\end{remark}

\begin{remark}
  \label{rem:relation-kalantar-scarparo}
  In their work \cite{kalantarscarparo2020}, Kalantar-Scarparo considered boundary actions of a discrete group $G \grpaction{} X$, and obtained a characterisation of \Cstar-simplicity for quasi-regular representations arising from neighbourhood stabilisers $G_x^0$ with $x \in X \setminus \bigcup_{g \in G} \partial(\Fix(g)^\circ)$ \cite[Corollary 5.3]{kalantarscarparo2020}.  Our Theorem~\ref{thm:relative-powers-averaging-boundary-action} and Remark \ref{rem:identifying-representation} consider the associated groupoid of germs $\cG$ and say that there the inclusion $\lambda_{G, G_x^0}(G) \subseteq \Cstaress(\cG)$ satisfies relative Powers averaging.  This is a considerably stronger conclusion.  No prediction about \Cstar-simplicity of the associated quasi-regular representations for $G_x^0$ is made for $x \in \bigcup_{g \in G} \partial(\Fix(g)^\circ)$, and \cite[Example 6.5]{kalantarscarparo2020} even shows that in general $\lambda_{G, G_x^0}$ it will not necessarily be \Cstar-simple in this case.  Theorem~\ref{thm:relative-powers-averaging-boundary-action} provides a conceptual explanation of this phenomenon, by drawing attention to the difference between the reduced and the essential groupoid \Cstar-algebra.
  
  More recently in \cite[Theorem 5.5]{kalantarscarparo2021} simplicity results for the groupoid of germs associated with a minimal group action on a locally compact space $G \grpaction{} X$ were obtained.  This result is implied by Theorem~\ref{thm:characterisation-locally-compact-case}.  However, in this situation there is no natural map from $G$ to the essential groupoid \Cstar-algebra of the groupoid of germs, so that Powers averaging and thus an analogue of Theorem~\ref{thm:relative-powers-averaging-boundary-action} needs further care to be even formulated.
\end{remark}

We now consider the concrete case of Thompson's group $\rT$ acting on the circle as well as Thompson's group $\rV$ acting on a totally disconnected cover of the circle. The latter action was previously considered in \cite{kalantarscarparo2020}, in order to reprove the simplicity of the Cuntz algebra $\cO_2$ using techniques from the theory of \Cstar-simplicity.
\begin{example}
  \label{ex:irreducible-inclusion-T}
  Consider Thompson's group $\rT \subseteq \mathrm{Homeo}(\rS^1)$.  It is the group of piecewise linear transformations of $\rS^1 \cong \RR \bP^1$ with breakpoints in $\exp(2\pi i \ZZ[\frac{1}{2}])$ and derivatives in $2^\ZZ$.  It acts transitively on non-trivial intervals of the circle, whose end points lie in $\exp(2\pi i \ZZ[\frac{1}{2}])$.  In particular, $\rT \grpaction{} \rS^1$ is a boundary action.  Let $\cG$ be the groupoid of germs for $\rT \grpaction{} \rS^1$.  We observe that $\cG$ is non-Hausdorff, since $\rT$ does not act topologically freely while $\rS^1$ is connected.  Denote by $\pi: \rT \to \Cstaress(\cG)$ the associated unitary representation of $\rT$.  By Theorem~\ref{thm:relative-powers-averaging-boundary-action} the inclusion $\pi(\rT) \subseteq \Cstaress(\cG)$ satisfies relative Powers averaging.  If $x \in \rS^1 \setminus \exp(2 \pi i \ZZ[\frac{1}{2}])$ then Remark~\ref{rem:identifying-representation} further identifies this inclusion with $\lambda_{\rT, \rT_x^0} (T) \subseteq \Cstaress(\cG)$.  In contrast, it was shown in \cite[Example 6.5]{kalantarscarparo2020} that the quasi-regular representation associated with the standard inclusion $[F,F] \cong \rT_1^0 \subseteq \rT$ does not even generate a simple \Cstar-algebra.
\end{example}

\begin{remark}
  \label{rem:circle-dynamics-powers-averaging}
  The example of Thompson's group $\rT$ acting on the circle should be considered in the more general context of groups of homeomorphisms of the circle and the real line, which provides many examples of boundary actions that are not topologically free.  We mention several concrete examples.  First, Monod considered in \cite{monod2013-piecewise-projective} groups of piecewise projective homeomorphisms of the real line, arising as point stabiliser of $\infty \in \RR \mathbb{P}^1 \cong \rS^1$ of $\mathrm{PSL}_2(A)$ for arbitrary countable subrings $A \subseteq \RR$.  The action of $\mathrm{PSL}_2(\ZZ[\frac{1}{2}])$ on $\RR \mathbb{P}^1$ is conjugate to the action of $\rT$ on $\rS^1$.  Second, we like to point out recent work of Hyde-Lodha \cite{hydelodha2019} producing finitely generated, simple groups of homeomorphisms of the real line, which are constructed as variations of Thompson's group $\rT$.  Finally, Navas' survey \cite[p. 2056ff]{navas2018-icm} and the recent book of Kim and Koberda \cite{kimkoberda2021-structure-regularity} contain concrete questions about and provide examples of groups of homeomorphism of the circle, and include consideration of their contraction properties.  It should be pointed out that a group $G$ acting by homeomorphisms on a one-dimensional manifold $M \in \{\RR, \rS^1\}$ is strongly proximal if and only if it is extremally proximal, in the sense that for every pair of non-trivial open intervals $I, J \subseteq M$ there is $g \in G$ such that $g I \subseteq J$.  This is because an open interval can be described by its endpoints together with a point in its interior.  Such actions are also called CO-transitive in the dynamics community.
\end{remark}

The next example considers the action of $\rT$ on a suitable totally disconnected cover of the circle, and it yields a Hausdorff groupoid of germs.  A particularly interesting feature is that it produces a unitary representation of $\rT$ into the Cuntz algebra $\cO_2$ satisfying the relative Powers averaging property.
\begin{example}
  \label{ex:irreducible-inclusion-cuntz-algebra}
  Consider the following cover of the circle,
  \begin{gather*}
    K = \bigl ( \rS^1 \setminus \exp(2\pi i \ZZ[\frac{1}{2}]) \bigr ) \cup \bigl ( \{+, -\} \times \exp(2\pi i \ZZ[\frac{1}{2}]) \bigr )
  \end{gather*}
  equipped with the natural topology arising from the cyclic order.  We write $z_+$ and $z_-$ for the elements $(+,z)$ and $(-,z)$, respectively.  The action of $T$ lifts uniquely to an action on $K$ preserving the cyclic order.  By definition, $V$ is the topological full group of this action.  It follows directly from the definitions that $\rV \grpaction{} K$ is a boundary action, so that Theorem~\ref{thm:relative-powers-averaging-boundary-action} applies to the groupoid of germs $\cG(\rV \grpaction{} K) = \cG(\rT \grpaction{} K) = \cG$.  This groupoid is Hausdorff, since $\Fix(g)^\circ$ is clopen for every $g \in \rT$.  Considering the stabiliser $\rF \cong \rT_{1_+} \leq \rT$, and employing Remark~\ref{rem:identifying-representation} we obtain an inclusion $\lambda_{\rT, \rF}(\rT) \subseteq \Cstarred(\cG)$ satisfying the relative Powers averaging property.  By \cite{brixscarparo2019} we know that $\Cstarred(\cG)$ is generated by the image of Thompson's group $\rV$, so that \cite[Proposition 5.3]{haagerupolesen17} allows to make the identification with the Cuntz algebra $\Cstarred(\cG) \cong  \cO_2$.
\end{example}

\begin{remark}
  \label{rem:irreducible-inclusions-kirchberg-algebras}
  It is known that $\mathbf{F}(\cG) \grpaction{} \Gnaught$ is extremally proximal for every purely infinite Hausdorff groupoid $\cG$.  But the work in \cite{brixscarparo2019} shows that $\mathbf{F}(\cG)$ generates $\Cstarred(\cG)$ in this case.  In view of Theorem~\ref{thm:relative-powers-averaging-boundary-action} and Example~\ref{ex:irreducible-inclusion-cuntz-algebra}, it is natural to ask the following question.  Is there a systematic approach to constructing subgroups $G \leq \mathbf{F}(\cG)$ such that the \Cstar-algebra inclusion generated by $G$ inside $\Cstarred(\cG)$ is proper?
\end{remark}


{\small
  \printbibliography
}



  


\vspace{2em}

\begin{minipage}[t]{0.45\linewidth}
  \small
  Matthew Kennedy \\
  Department of Pure Mathematics \\
  University of Waterloo \\
  200 University Avenue West \\
  Waterloo, Ontario, N2L 3G1 \\
  Canada \\[0.2em]
  {\footnotesize matt.kennedy@uwaterloo.ca}
\end{minipage}
\begin{minipage}[t]{0.45\linewidth}
  \small 
  Se-Jin Kim \\
  KU Leuven Department of Mathematics \\
  KU Leuven \\
  Celestijnenlaan 200B \\
  Leuven, 3001 \\
  Belgium \\
  {\footnotesize sam.kim@kuleuven.be}
\end{minipage}

\vspace{2em}

\begin{minipage}[t]{0.45\linewidth}
  \small
  Xin Li \\
  School of Mathematics and Statistics \\
  University of Glasgow \\
  University Place \\
  Glasgow, G12 8QQ \\
  United Kingdom \\[0.2em]
  {\footnotesize xin.li@glasgow.ac.uk}
\end{minipage}
\begin{minipage}[t]{0.45\linewidth}
  \small
  Sven Raum \\
  Institute of Mathematics \\
  University of Potsdam \\
  Campus Golm, Haus 9 \\
  Karl-Liebknecht-Str. 24-25 \\
  14476 Potsdam\\
  Germany\\[0.2em]
  {\footnotesize sven.raum@uni-potsdam.de}
\end{minipage}

\vspace{2em}

\begin{minipage}[t]{0.45\linewidth}
  \small
  Dan Ursu \\
  Mathematisches Institut \\
  University of M{\"u}nster \\ 
  Einsteinstr. 62 \\
  48149 M{\"u}nster \\
  Germany \\[0.2em]
  {\footnotesize dursu@uni-muenster.de}
\end{minipage}


\end{document}

%% file: shortcuts.tex
\newcommand{\cC}{\ensuremath{\mathcal{C}}}

\newcommand{\cG}{\ensuremath{\mathcal{G}}}
\newcommand{\cH}{\ensuremath{\mathcal{H}}}

\newcommand{\cM}{\ensuremath{\mathcal{M}}}
\newcommand{\cN}{\ensuremath{\mathcal{N}}}
\newcommand{\cO}{\ensuremath{\mathcal{O}}}
\newcommand{\cP}{\ensuremath{\mathcal{P}}}

\newcommand{\cS}{\ensuremath{\mathcal{S}}}

\newcommand{\cZ}{\ensuremath{\mathcal{Z}}}

\newcommand{\bP}{\ensuremath{\mathbb{P}}}

\newcommand{\rE}{\ensuremath{\mathrm{E}}}
\newcommand{\rF}{\ensuremath{\mathrm{F}}}

\newcommand{\rL}{\ensuremath{\mathrm{L}}}
\newcommand{\rM}{\ensuremath{\mathrm{M}}}

\newcommand{\rS}{\ensuremath{\mathrm{S}}}
\newcommand{\rT}{\ensuremath{\mathrm{T}}}

\newcommand{\rV}{\ensuremath{\mathrm{V}}}

\newcommand{\rmc}{\ensuremath{\mathrm{c}}}
\newcommand{\rmd}{\ensuremath{\mathrm{d}}}

\newcommand{\rmm}{\ensuremath{\mathrm{m}}}

\newcommand{\rmr}{\ensuremath{\mathrm{r}}}
\newcommand{\rms}{\ensuremath{\mathrm{s}}}

\newcommand{\veps}{\ensuremath{\varepsilon}}

\newcommand{\vphi}{\ensuremath{\varphi}}

\newcommand{\ol}{\overline}

\newcommand{\eqstop}{\ensuremath{\, \text{.}}}
\newcommand{\eqcomma}{\ensuremath{\, \text{,}}}

\newcommand{\NN}{\ensuremath{\mathbb{N}}}
\newcommand{\ZZ}{\ensuremath{\mathbb{Z}}}

\newcommand{\RR}{\ensuremath{\mathbb{R}}}
\newcommand{\CC}{\ensuremath{\mathbb{C}}}

\renewcommand{\Re}{\ensuremath{\mathop{\mathrm{Re}}}}

\newcommand{\id}{\ensuremath{\mathrm{id}}}
\newcommand{\ra}{\ensuremath{\rightarrow}}

\newcommand{\hra}{\ensuremath{\hookrightarrow}}
\newcommand{\thra}{\ensuremath{\twoheadrightarrow}}

\newcommand{\ev}{\ensuremath{\mathrm{ev}}}

\newcommand{\Cstar}{\ensuremath{\mathrm{C}^*}}

\newcommand{\Wstar}{\ensuremath{\mathrm{W}^*}}

\newcommand{\bo}{\ensuremath{\mathcal{B}}}

\newcommand{\supp}{\ensuremath{\mathop{\mathrm{supp}}}}
\newcommand{\im}{\ensuremath{\mathop{\mathrm{im}}}}

\newcommand{\Cstarred}{\ensuremath{\Cstar_\mathrm{red}}}
\newcommand{\Cstarmax}{\ensuremath{\Cstar_\mathrm{max}}}

\newcommand{\cont}{\ensuremath{\mathrm{C}}}
\newcommand{\contb}{\ensuremath{\mathrm{C}_\mathrm{b}}}
\newcommand{\conto}{\ensuremath{\mathrm{C}_0}}
\newcommand{\contc}{\ensuremath{\mathrm{C}_\mathrm{c}}}
\newcommand{\conv}{\ensuremath{\mathrm{conv}}}

\newcommand{\ltwo}{\ensuremath{\ell^2}}

\newcommand{\linfty}{\ensuremath{{\offinterlineskip \ell \hskip 0ex ^\infty}}}

\newcommand{\lspan}{\ensuremath{\mathop{\mathrm{span}}}}
\newcommand{\cspan}{\ensuremath{\mathop{\overline{\mathrm{span}}}}}

\newcommand{\Ad}{\ensuremath{\mathop{\mathrm{Ad}}}}

\newcommand{\grpaction}[1]{\ensuremath{\stackrel{#1}{\curvearrowright}}}

\newcommand{\Fix}{\ensuremath{\mathrm{Fix}}}

